\documentclass[11pt]{article}
\usepackage{graphicx}
\usepackage{amsmath}
\usepackage{amssymb}
\usepackage{amsthm}
\usepackage{booktabs}
\usepackage{tabularx}
\usepackage{esint}
\usepackage{mathtools}
\usepackage{mathrsfs}
\usepackage[shortlabels]{enumitem}
\usepackage{xparse}
\usepackage{resmes}
\usepackage[colorlinks=true, linkcolor=blue, citecolor=blue, urlcolor=blue,pagebackref=false]{hyperref}

\usepackage[T1]{fontenc}
\usepackage{lmodern}

\numberwithin{equation}{section}
\newtheorem{definition}{Definition}[section]
\newtheorem{lemma}[definition]{Lemma}
\newtheorem{proposition}[definition]{Proposition}
\newtheorem{proposition*}[definition]{*Proposition}
\newtheorem{definition*}[definition]{*Definition}
\newtheorem{theorem}[definition]{Theorem}
\newtheorem{example}[definition]{Example}
\newtheorem{corollary}[definition]{Corollary}
\newtheorem{remark}[definition]{Remark}
\newtheorem{question}[definition]{Question}

\newcommand{\sub}{\subseteq}

\newcommand{\R}{\mathbb{R}}

\newcommand*{\n}[1]{\left\|#1\right\|}
\newcommand*{\inn}[2]{\left \langle #1, #2 \right \rangle}

\title{
Asymptotic analysis of Dirichlet problems in supercritically perforated domains via quantification by torsional rigidity}
\author{
    Naoto Sato \thanks{Mathematical Institute, Tohoku University. \footnotesize{\href{mailto:sato.naoto.t4@dc.tohoku.ac.jp}{sato.naoto.t4@dc.tohoku.ac.jp}}}
}
\date{\today}
\begin{document}

\maketitle
\begin{abstract}
We study the Dirichlet problem $-\mathrm{div}(A^{\varepsilon}\nabla u^{\varepsilon}) = f$ on varying open sets $D^{\varepsilon}$, with uniformly elliptic coefficients $A^{\varepsilon}$, in the regime where the solutions vanish in the limit for every source. For periodically perforated domains, this is the supercritical regime. We show that the torsional rigidity $T(D^{\varepsilon})$ is the natural normalization: up to a subsequence, $u^{\varepsilon}/T(D^{\varepsilon})$ converges weakly as measures to $kf$ for every continuous source $f$, where $k$ is a nonzero finite measure independent of $f$, which we call the effective measure. The key ingredient is an estimate comparing $u^{\varepsilon}$ with the product of $f$ and the torsion function, whose rate depends on the domains only through $T(D^{\varepsilon})$. We also show that this rate is optimal. No structural assumption on the perforation is needed, and the results hold for relaxed Dirichlet problems. Every nonzero finite measure arises as an effective measure, so the continuity of the source cannot be dropped in general. For rougher sources, we characterize the convergence by a condition on the domains alone, expressed through their torsion functions or, equivalently, their Poincar\'{e} constants. We also prove that the effective measure is local, satisfies optimal two-sided bounds and is monotone with respect to the coefficients, and we study the role of the antisymmetric parts of the coefficients.
\end{abstract}
{\small

\smallskip

}
\medskip

\tableofcontents

\section{Introduction}
\subsection{Motivation}
A prototypical example of homogenization of the Poisson equation in perforated domains as presented in \cite{cioranescu1997strange,kacimi1987homogenisation} can be stated as follows. Let $D \sub \R^{d}$, $d \ge 3$,\footnote{The trichotomy also holds true when $d = 2$ with a different critical scale, see \cite[Chapitre 3]{kacimi1987homogenisation}, \cite{Jing2020unified}.} be a bounded domain and $\eta \ge 1$, and set, for $0 < \varepsilon < 1$,
\begin{equation}
    \label{eq: perforated domain Deh}
    D^{\varepsilon}_{\eta} := D \setminus \bigcup_{z \in \mathbb{Z}^{d}}B(\varepsilon z, 4^{-1}\varepsilon^{\eta}),
\end{equation}
so that $D^{\varepsilon}_{\eta}$ is obtained by removing from $D$ tiny closed balls centered at the points of $\varepsilon \mathbb{Z}^{d}$. Then the asymptotic behavior of (the zero extension of) the solution $u^{\varepsilon}$ to
    \begin{equation}
        \label{eq: intro periodic}
        -\Delta u^{\varepsilon} = f \text{ in $H^{-1}(D^{\varepsilon}_{\eta})$, $u^{\varepsilon} \in H^{1}_{0}(D^{\varepsilon}_{\eta})$}
    \end{equation} 
is described by two nonnegative constants $k_{\eta}, \mu_{\eta}$ depending only on the geometry of the perforation in the following sense:
\begin{itemize}
    \item Supercritical case ($1 \le \eta < \frac{d}{d - 2}$): For every $f \in H^{-1}(D)$, $u^{\varepsilon} \to 0$ in $H^{1}_{0}(D)$. When $f \in L^{2}(D)$, $\varepsilon^{-d + (d - 2)\eta}u^{\varepsilon} \rightharpoonup k_{\eta}f$ in $L^{2}(D)$ weak\footnote{ The convergence is strong if $1 < \eta < \frac{d}{d - 2}$, see \cite[Th\'{e}or\`{e}me 5.1 of Chapitre III]{kacimi1987homogenisation}, \cite[Theorem 2.4]{Jing2020unified}.}.
    \item Critical case ($\eta = \frac{d}{d - 2}$): For every $f \in H^{-1}(D)$, $u^{\varepsilon}$ converges in $H^{1}_{0}(D)$ weak to the solution to
    \begin{equation}
        \label{eq: schroedinger intro}
        -\Delta u + \mu_{\eta} u = f \text{ in $H^{-1}(D)$, $u \in H^{1}_{0}(D)$}
    \end{equation}
    where $0 < \mu_{\eta} < \infty$.
    \item Subcritical case ($\eta > \frac{d}{d - 2}$): The same convergence as the critical case takes place in $H^{1}_{0}(D)$ strong with $\mu_{\eta} = 0$.
\end{itemize}
There is a vast literature extending the above result (see, e.g., \cite{ansini2002asymptotic,focardi2009homogenization,focardi2010aperiodic,khrabustovskyi2018operator, giunti2018homogenization,giunti2021convergence, giunti2021derivation,marchenko2006homogenization,sato2026homogenization,scardia2025homogenisation, deangelis2026stochastic,beliaev1996darcy} and the references therein).

On the other hand, we know from the compactness theorem \cite[Theorem 4.14]{dalmaso1987wiener} for $\gamma$-convergence due to Dal Maso and Mosco that for an arbitrary sequence of open subsets $D^{\varepsilon} \sub D$ in $\R^{d}$, $d \ge 2$, the following holds: there exist a capacitary measure $\mu$ and a subsequence, still denoted by $(\varepsilon)$, such that for every $f \in H^{-1}(D)$ the solution $u^{\varepsilon}$ to
\begin{equation}
    \label{eq: poisson general intro}
    -\Delta u^{\varepsilon} = f \text{ in $H^{-1}(D^{\varepsilon})$, $u^{\varepsilon} \in H^{1}_{0}(D^{\varepsilon})$}
\end{equation}
converges in $H^{1}_{0}(D)$ weak to the solution to the so-called relaxed Dirichlet problem
\begin{equation}
    \label{eq: schroedinger eq}
    -\Delta u + \mu u = f \text{ in $(H^{1}_{\mu})^{*}$, $u \in H^{1}_{\mu}$}
\end{equation}
(for the definition of capacitary measures and the problem \eqref{eq: schroedinger eq}, see Subsection \ref{sect: capacitary meas}). Roughly speaking, if $\mu$ is infinitely large then the solution to \eqref{eq: schroedinger eq} is $0$, and if $\mu$ is Lebesgue measure times $\mu_{\eta}$ in $D$ and is infinitely large outside $D$, then $u$ is the solution to \eqref{eq: schroedinger intro}. In this way, they showed that the limit of $u^{\varepsilon}$ is always a solution to a problem of the type \eqref{eq: schroedinger eq}. Generalizations of this compactness theorem to wider classes of problems as well as applications to shape optimization have been studied intensively. Without attempting a comprehensive review, we cite \cite{baxter1987stopping, dalmaso1994new, dalmaso1997asymptotic,dalmaso2004asymptotic,casadodiaz2000asymptotic, casadodiaz2000homogenization,calvojurado2002thelimit, calvo2003homogenization,lentz2025capacitary, bucur2005variational,buttazzo1991shape}.

This raises the following natural question.
\begin{question}
    \label{q: scaling limit}
    Does a suitable scaling limit of the solution $u^{\varepsilon}$ to \eqref{eq: poisson general intro} in an appropriate topology always take the form $kf$ for some nonzero quantity $k$ independent of $f$ when the domains $D^{\varepsilon}$ are ``supercritically perforated'' in the sense that the unscaled solution $u^{\varepsilon}$ converges to zero  for any source $f \in H^{-1}(D)$?
\end{question}
In this article, we answer this question affirmatively, with $k$ a finite measure on $\overline{D}$, also for elliptic equations with variable coefficients, and investigate how $k$ depends on the domains $D^{\varepsilon}$ and on the coefficients.

\subsection{Overview of main results}
We now describe our main results. In the subsequent sections, we work within the framework of the relaxed Dirichlet problem, which includes as a special case the Dirichlet problem in a domain. However, for clarity, in this introduction we do not state our results in full generality and restrict ourselves to Dirichlet problems on varying open sets. Our notation is summarized in Subsection \ref{sect: notation}.

Let $D \sub \R^{d}$, $d \ge 2$ be nonempty bounded open, and let $(D^{\varepsilon}) \sub D$ be nonempty open subsets.  For $0 < \alpha \le \beta < \infty$, put
\begin{equation}
    M(\alpha, \beta) := \left\{ A \in L^{\infty}(\R^{d})^{d \times d}: A\xi \cdot \xi \ge \alpha |\xi|^{2} \text{ and } A^{-1}\xi\cdot \xi \ge \beta^{-1}|\xi|^{2} \text{ $\mathcal{L}^{d}$-a.e. for every $\xi \in \R^{d}$} \right\}.
\end{equation}
For a bounded open $\Omega \sub \R^{d}$, let $w_{\Omega} \in H^{1}_{0}(\Omega)$ denote the torsion function on $\Omega$, i.e., the solution to 
\begin{equation}
    -\Delta w_{\Omega} = 1 \text{ in $H^{-1}(\Omega)$},\, w_{\Omega} \in H^{1}_{0}(\Omega),
\end{equation}
and denote by $T(\Omega) := \n{w_{\Omega}}_{L^{1}(\Omega)}$ the torsional rigidity of $\Omega$.
Our first result answers Question \ref{q: scaling limit} in the affirmative in a more general setting.
\begin{theorem}[Theorem \ref{thm: main result}]
    \label{thm: convergence and error intro}
    Consider the problems
    \begin{equation}
        \label{eq: div form eq intro}
    -\mathrm{div}(A^{\varepsilon} \nabla u^{\varepsilon}) = f  \text{ in $H^{-1}(D^{\varepsilon})$},\, u^{\varepsilon} \in H^{1}_{0}(D^{\varepsilon}).
\end{equation}
Here, $A^{\varepsilon} \in M(\alpha, \beta)$. Assume that $T(D^{\varepsilon}) \to 0$, which is equivalent to the assumption that $u^{\varepsilon} \to 0$ in $H^{1}_{0}(D)$ for every $f \in H^{-1}(D)$. Then there exist a subsequence, still denoted by $(\varepsilon)$, and a nonnegative finite measure $k$ on the closure $\overline{D}$ of $D$ with $\beta^{-1} \le k(\overline{D}) \le \alpha^{-1}$ such that for every source $f \in C(\overline{D})$, the solution $u^{\varepsilon}$ to \eqref{eq: div form eq intro} satisfies
\begin{equation}
    \label{eq: qualitative convergence intro}
    \frac{u^{\varepsilon}}{T(D^{\varepsilon})} \rightharpoonup kf
\end{equation}
weakly in the space $\mathcal{M}(\overline{D})$ of Radon measures on $\overline{D}$. Further, if $p, q \in [1, 2]$ satisfy $\frac{1}{d} - \left( \frac{1}{p'} + \frac{1}{q'} \right)  > 0$, then we have the quantitative error estimate
        \begin{equation}
            \label{eq: quantitative error intro}
            \n{\frac{u^{\varepsilon}}{T(D^{\varepsilon})} - \frac{w_{A^{\varepsilon}, D^{\varepsilon}}}{T(D^{\varepsilon})}f}_{L^{q}(D)} \lesssim_{d, \alpha, \beta}  \n{\nabla f}_{L^{p'}(D^{\varepsilon})}T(D^{\varepsilon})^{\delta_{d, p, q}}
        \end{equation}
        for every $f \in C(\overline{D}) \cap W^{1, p'}(D)$. Here, $w_{A^{\varepsilon}, D^{\varepsilon}}$ denotes the solution to \eqref{eq: div form eq intro} with $f = 1$ and $\delta_{d, p, q} = \frac{d}{d + 2}\left(\frac{1}{d} - \left( \frac{1}{p'} + \frac{1}{q'} \right)\right)$.
\end{theorem}
In Subsection \ref{sect: optimality}, we show that the exponent $\delta_{d, p, q}$ in \eqref{eq: quantitative error intro} is optimal.

With the aid of the unscaled asymptotic analysis of \eqref{eq: div form eq intro} developed in \cite{dalmaso2004asymptotic}, we obtain the following dichotomy.

\begin{corollary}[Corollary \ref{thm: main thm for domains}]
    Consider the problem \eqref{eq: div form eq intro}. Up to a subsequence, one of the following holds.
    \begin{enumerate}
        \item (\cite[Theorem 5.2, Corollary 5.5]{dalmaso2004asymptotic}). There are $A \in M(\alpha, \beta)$ and a capacitary measure $\mu \in \mathcal{M}_{\mathrm{cap}}(D) \setminus \{I_{\varnothing}\}$ such that for every $f \in H^{-1}(D)$, the solution $u^{\varepsilon}$ converges in $H^{1}_{0}(D)$ weak to the solution $u$ to the relaxed Dirichlet problem
        \[-\mathrm{div}(A \nabla u) + \mu u = f  \text{ in $(H^{1}_{\mu})^{*}$, $u \in H^{1}_{\mu}$},\]
        the solution to which is nonzero whenever $f \in L^{2}(D)$ is nonnegative and bounded away from $0$ $\mathcal{L}^{d}$-a.e. on $D$ (for definitions of capacitary measures and relaxed Dirichlet problems, see Subsection \ref{sect: capacitary meas}).
        \item $u^{\varepsilon} \to 0$ in $H^{1}_{0}(D)$ for every $f \in H^{-1}(D)$. Further, there is a nonnegative finite Borel measure $k \in \mathcal{M}(\overline{D})$ with $\beta^{-1} \le k(\overline{D}) \le \alpha^{-1}$ such that for every $f \in C(\overline{D})$, the convergence \eqref{eq: qualitative convergence intro} takes place with the quantitative estimate \eqref{eq: quantitative error intro} for $p, q \in [1, 2]$ with $\frac{1}{d} - \left( \frac{1}{p'} + \frac{1}{q'} \right)  > 0$ and $f \in C(\overline{D}) \cap W^{1, p'}(D)$.
    \end{enumerate}
\end{corollary}

For $A \in M(\alpha, \beta)$ and an open $\Omega \sub \R^{d}$, we denote by $P(A, \Omega)$ the
Dirichlet problem $-\mathrm{div}(A\nabla u) = f$ in $H^{-1}(\Omega)$, $u \in H^{1}_{0}(\Omega)$;
thus \eqref{eq: div form eq intro} is $P(A^{\varepsilon}, D^{\varepsilon})$. We call the measure
$k$ in Theorem \ref{thm: convergence and error intro}, i.e., the limit of
$T(D^{\varepsilon})^{-1}w_{A^{\varepsilon}, D^{\varepsilon}}$ in $\mathcal{M}(\overline{D})$ weak,
the \emph{effective measure} of the problems $P(A^{\varepsilon}, D^{\varepsilon})$
(Definition \ref{def: eff meas}). It is an analogue of the ratio of the permeability to the
viscosity in Darcy's law (see Remark \ref{rmk: Darcy}).

We also show in Proposition \ref{thm: convergence of energy} that the scaled energy density $T(D^{\varepsilon})^{-1}A^{\varepsilon}\nabla u^{\varepsilon} \cdot \nabla u^{\varepsilon}$ converges in $\mathcal{M}(\overline{D})$ weak to $kf^{2}$ when $f \in C(\overline{D})$ with a quantitative error estimate similar to \eqref{eq: quantitative error intro} for more regular sources.

We show in Proposition \ref{thm: attainability} that any nonzero nonnegative finite Borel measure $k$ on $\overline{D}$ is the effective measure of the problems $P(k(\overline{D})^{-1} \mathrm{Id}, D^{\varepsilon})$ for some open $D^{\varepsilon} \sub D$ with vanishing torsional rigidity. Consequently, the limit $kf$ appearing in \eqref{eq: qualitative convergence intro} is meaningful only for $f \in C(\overline{D})$ without further restrictions on $D^{\varepsilon} \sub D$. In Section \ref{sect: rougher source}, we characterize when this convergence extends to rougher sources. For $1 < r \le 2$, we have $k \in L^{\frac{r}{2 - r}}(D)$ and $T(D^{\varepsilon})^{-1}u^{\varepsilon} \rightharpoonup kf$ weakly in $L^{r}(D)$ for every $f \in L^{r'}(D)$ if and only if $\n{w_{D^{\varepsilon}}}_{\frac{r}{2 - r}} \lesssim T(D^{\varepsilon})$ uniformly in $\varepsilon$, or, equivalently, the $L^{r}$ Poincar\'{e} constants of $D^{\varepsilon}$ are bounded below by a constant multiple of $T(D^{\varepsilon})^{-1}$ (Proposition \ref{thm: rougher source term}). This condition depends only on the domains, not on the coefficients.

We also obtain various properties of effective measures. In the rest of this section, we assume that the problems $P(A^{\varepsilon}, D^{\varepsilon})$ admit an effective measure $k$, i.e., that $T(D^{\varepsilon})^{-1}w_{A^{\varepsilon}, D^{\varepsilon}} \rightharpoonup k$ in $\mathcal{M}(\overline{D})$ weak. By Theorem \ref{thm: convergence and error intro}, this can always be achieved by passing to a subsequence.

\begin{proposition}[{Locality, Proposition \ref{thm: locality}}]
    Let $U \sub \R^{d}$ be open. Then 
    \[T(D^{\varepsilon})^{-1}w_{A^{\varepsilon}, D^{\varepsilon} \cap U} \rightharpoonup k\resmes U\]
    weakly* in the space $\mathcal{M}(U)$ of finite Radon measures on $U$. Here, $k\resmes U (\cdot) := k(U \cap \cdot)$ is the restriction of $k$ to $U$ (see Subsection \ref{sect: notation}), and $w_{A^{\varepsilon}, D^{\varepsilon} \cap U}$ is zero-extended to a function on $\R^{d}$.
\end{proposition}

\begin{theorem}[{Role of the antisymmetric parts, Proposition \ref{thm: adjoint} and Theorem \ref{thm: continuous antisymmetric part in the uniform limit does not contribute to k}}]
    \label{thm: antisymmetric intro}
    The measure $k$ is also the effective measure of the adjoint problems $P(A^{\varepsilon *}, D^{\varepsilon})$. If, furthermore, the antisymmetric parts $2^{-1}(A^{\varepsilon} - A^{\varepsilon *})$ converge in $L^{\infty}(D)$ to a matrix-valued $C \in \mathcal{C}_{D, k}^{d \times d}$, then $k$ is also the effective measure of the problems $P(A^{\varepsilon s}, D^{\varepsilon})$. Here, $A^{\varepsilon s} := 2^{-1}(A^{\varepsilon} + A^{\varepsilon *})$ is the symmetric part, and $\mathcal{C}_{D, k}$ denotes the class of functions $f \in L^{\infty}(\mathcal{L}^{d} \resmes D + k)$ such that there is a $(\mathcal{L}^{d}\resmes D + k)$-null set $N \sub \overline{D}$ such that $f$ is continuous on $\overline{D} \setminus N$.
\end{theorem}

Let $0 < \alpha' \le \beta' < \infty$ and $(B^{\varepsilon}) \sub M(\alpha', \beta')$ be such that the problems $P(B^{\varepsilon}, D^{\varepsilon})$ admit an effective measure $k'$. The next two results compare $k$ and $k'$.

\begin{theorem}[{Bounds on effective measures, Theorem \ref{thm: effective measures comparison}}]
    We have the following bound.
    \begin{equation}
        \label{eq: bounds on effective measures intro}
        \frac{\alpha}{\beta'}k \le k' \le \frac{\beta}{\alpha'}k.
    \end{equation}
    Furthermore, for every Borel measure $\nu$ on $\R^{d}$ with
    \begin{equation}
        \frac{\alpha}{\beta'}k \le \nu \le \frac{\beta}{\alpha'}k,
    \end{equation}
    there is a sequence $(f_{\varepsilon}) \sub C^{\infty}(\R^{d})$ with $\frac{\alpha}{\beta'} \le f_{\varepsilon} \le \frac{\beta}{\alpha'}$ such that $f_{\varepsilon}k \to \nu$ in total variation norm and $\nu$ is the effective measure of the problems $P(f_{\varepsilon}^{-1}A^{\varepsilon}, D^{\varepsilon})$, where $f_{\varepsilon}^{-1}A^{\varepsilon} \in M(\frac{\alpha}{\beta}\alpha', \frac{\beta}{\alpha}\beta')$. In particular, the bound \eqref{eq: bounds on effective measures intro} is optimal if $A^{\varepsilon}$ are the identity matrices and if we let $\alpha = \beta = 1$. If the Radon--Nikodym derivative of $\nu$ with respect to $k$ admits a version $\rho \in \mathcal{C}_{D, k}$, then $\nu$ is also the effective measure of the problems $P(\hat{\rho}^{-1}A^{\varepsilon}, D^{\varepsilon})$. Here, $\hat{\rho} = (\rho \vee \frac{\alpha}{\beta'}) \wedge \frac{\beta}{\alpha'}$.
\end{theorem}

\begin{proposition}[{Order property of effective measures, Proposition \ref{thm: effective measures comparison symmetric}}]
    Suppose that $B^{\varepsilon}$ are symmetric $\mathcal{L}^{d}$-a.e. on $D$ and that $B^{\varepsilon} \le A^{\varepsilon}$ in the sense that $B^{\varepsilon}\xi \cdot \xi \le A^{\varepsilon} \xi \cdot \xi$ holds for every $\xi \in \R^{d}$ $\mathcal{L}^{d}$-a.e. on $D$. Then $k \le k'$.
\end{proposition}

\medskip

\noindent
\textbf{Idea of the proofs.}

The core of our analysis is quantitative, and its guiding principle is to bound every error term by a power of the torsional rigidity alone. For instance, the qualitative part \eqref{eq: qualitative convergence intro} of Theorem \ref{thm: convergence and error intro} is easily deduced from the error estimate \eqref{eq: quantitative error intro} by a density argument and weak compactness of bounded sets in $\mathcal{M}(\overline{D})$. The proof of \eqref{eq: quantitative error intro} consists of the following two steps.
 
\emph{Step 1: an energy estimate for the corrector.} The difference $z^{\varepsilon} := u^{\varepsilon} - w_{A^{\varepsilon}, D^{\varepsilon}}f$ solves
\[
-\mathrm{div}(A^{\varepsilon}\nabla z^{\varepsilon}) = \mathrm{div}(w_{A^{\varepsilon}, D^{\varepsilon}}A^{\varepsilon}\nabla f) + A^{\varepsilon}\nabla w_{A^{\varepsilon}, D^{\varepsilon}} \cdot \nabla f \text{ in $D^{\varepsilon}$}, \quad z^{\varepsilon} = 0\text{ on $\partial D^{\varepsilon}$}.
\]
We test this equation with $z^{\varepsilon}$ and control the second term by the so-called torsional Hardy inequality
\[
\int_{\Omega} \frac{|\nabla w_{A, \Omega}|^{2}}{w_{A, \Omega}^{2}}\varphi^{2} \lesssim_{\alpha, \beta} \int_{\Omega} |\nabla \varphi|^{2}, \quad \varphi \in H^{1}_{0}(\Omega)
\]
(Lemma \ref{thm: torsional_Hardy}). We then pass from the energy to $L^{q}$-norms by means of $\lambda_{q}(D^{\varepsilon})$, where, for an open $\Omega \sub \R^{d}$, $\lambda_{q}(\Omega)$ denotes the largest constant $C$ such that $C\n{\varphi}_{q}^{2} \le \n{\nabla \varphi}_{2}^{2}$ for every $\varphi \in C_{c}^{\infty}(\Omega)$. This yields
\[
\n{u^{\varepsilon} - w_{A^{\varepsilon}, D^{\varepsilon}}f}_{L^{q}(D^{\varepsilon})} \lesssim_{\alpha, \beta} \lambda_{q}(D^{\varepsilon})^{-\frac{1}{2}}\n{w_{A^{\varepsilon}, D^{\varepsilon}}}_{L^{\frac{2p}{2 - p}}(D^{\varepsilon})}\n{\nabla f}_{L^{p'}(D^{\varepsilon})}
\]
(Lemma \ref{thm: energy estimate for u-wf} and Theorem \ref{thm: quantitative error estimate}).
 
\emph{Step 2: reduction to the torsional rigidity.} The two factors $\lambda_{q}(D^{\varepsilon})^{-\frac{1}{2}}$ and $\n{w_{A^{\varepsilon}, D^{\varepsilon}}}_{\frac{2p}{2 - p}}$ still depend on the geometry of $D^{\varepsilon}$. We express them through $T(D^{\varepsilon})$ alone by means of two estimates for torsion functions, which we adapt from Brasco--Ruffini \cite{brasco2017compact} and van den Berg--Bucur \cite{vandenBerg2014on} to general $A \in M(\alpha, \beta)$ and to capacitary measures:
\[
\lambda_{q}(\Omega)^{-1} \asymp_{d, \alpha, \beta, q} \n{w_{A, \Omega}}_{\frac{q}{2 - q}} \quad \text{and} \quad \n{w_{A, \Omega}}_{s} \lesssim_{d, \alpha} T(\Omega)^{\frac{2s + d}{(d + 2)s}} \quad (1 \le s \le \infty)
\]
(Lemmas \ref{thm: torsion and poincare const} and \ref{thm: Linfty_Lt_estimate_of_torsion_func}). Inserting both estimates into Step 1 bounds the right-hand side by $\n{\nabla f}_{L^{p'}(D^{\varepsilon})}T(D^{\varepsilon})^{1 + \delta_{d, p, q}}$, which is \eqref{eq: quantitative error intro}.
  
\emph{Other results.} The same principle applies to the rest of the paper. Most properties of the effective measure (Propositions \ref{thm: convergence of energy}, \ref{thm: adjoint}, and \ref{thm: locality} and Theorem \ref{thm: effective measures comparison}) are proved by testing one equation with the solution of another equation multiplied by a test function $\varphi$. The resulting error terms are $O(T(D^{\varepsilon})^{1 + \delta})$ by Step 2 for some $\varepsilon$-independent $\delta > 0$, and thus vanish after division by $T(D^{\varepsilon})$.

\medskip

\noindent
\textbf{Novelties of the present paper.}

\emph{A universal normalization.} For periodically perforated domains such as \eqref{eq: perforated domain Deh}, the scaling limit in the supercritical regime is classical. There the normalization $\varepsilon^{d - (d - 2)\eta}$ is explicit (\cite{lions1980asymptotic,kacimi1987homogenisation,Jing2020unified}). To the best of our knowledge, no comparable result was available for general sequences of domains whose unscaled solutions vanish in the limit.

We show that, for an arbitrary sequence of open sets $(D^{\varepsilon}) \sub D$ with vanishing torsional rigidity and arbitrary coefficients $(A^{\varepsilon}) \sub M(\alpha, \beta)$, the torsional rigidity is the right normalization: up to a subsequence, $T(D^{\varepsilon})^{-1}u^{\varepsilon}$ converges to $kf$ for every continuous source (Theorem \ref{thm: main result}). No periodicity, stationarity or other structural assumption on the perforation is needed, and the result holds for relaxed Dirichlet problems.

\emph{Optimal error estimate for general domains.} In general, the more abstract the setting, the less quantitative information one can expect. Quantitative results have been obtained for specific perforations: periodic ones \cite{lions1980asymptotic, kacimi1989estimation,khrabustovskyi2018operator, Jing2020unified}, local perturbations of periodic ones \cite{xavier2022homogenization}, non-periodic ones whose holes are well-separated \cite{borisov2023operator,borisov2024operator,borisov2024operatorStrange}, and random ones generated by i.i.d. random variables \cite{hoefer2024convergence}, whereas the theory for general varying open sets focuses on qualitative convergence \cite{dalmaso1987wiener,dalmaso1994new,dalmaso2004asymptotic}.

In this paper, we exploit the very assumption that the domains are ``supercritically perforated,'' i.e., that the torsional rigidity $T(D^{\varepsilon})$ vanishes in the limit, and quantify the error $T(D^{\varepsilon})^{-1}(u^{\varepsilon} - w_{A^{\varepsilon}, D^{\varepsilon}}f)$ by a power of $T(D^{\varepsilon})$ (the error estimate \eqref{eq: quantitative error intro}), under no assumption on the geometry of $D^{\varepsilon}$ other than the existence of a bounded box $D \supseteq D^{\varepsilon}$. Moreover, the exponent $\delta_{d,p,q}$ cannot be improved in this generality (Example \ref{ex: optimality}).

\subsection{Organization of the paper}
In Section \ref{sect: preliminaries}, we recall capacitary measures and relaxed Dirichlet problems and collect estimates for torsion functions.
 
Section \ref{sect: main_results} contains the main results. Subsection \ref{sect: convergence and error} states Theorem \ref{thm: main result} and the $\varepsilon$-independent estimate of Theorem \ref{thm: quantitative error estimate}, whose proof is given in Subsection \ref{sect: pf of quantitative error estimate}. Subsection \ref{sect: attainability} shows that every nonzero finite measure is attained as an effective measure, and Subsection \ref{sect: examples} collects examples. Subsections \ref{sect: properties of effective measures} and \ref{sect: comparison} establish the properties of the effective measure, and Subsection \ref{sect: optimality} shows the optimality of the error estimates. Section \ref{sect: rougher source} deals with rougher sources. The appendix contains a multiplicative inequality, an $L^{\infty}$ estimate for solutions, and an approximation lemma.
\section{Preliminaries}
\label{sect: preliminaries}
\subsection{Notations}
\label{sect: notation}
\begin{itemize}
    \item Throughout the paper, $d \ge 2$ is an integer.
    \item Throughout the paper, $\varepsilon$ varies in a sequence in $]0, 1[$ strictly decreasing to $0$.
    \item We denote by $\chi_{A}$ the characteristic function of a set $A$.
    \item We denote by $E^{*}$ the continuous dual of a Banach space $E$ equipped with the usual norm.
    \item Given a locally compact Hausdorff space $X$, we denote by $\mathcal{M}(X)$ the Banach space of finite Radon measures on $X$ equipped with the total variation norm.
    \item Given a measure $\mu$ on a measurable space $(X, \Sigma)$ and a measurable $A \in \Sigma$, we denote by $\mu \resmes A$ the measure on $(X, \Sigma)$ defined by $(\mu \resmes A)(B) := \mu(A \cap B)$ for every $B \in \Sigma$.
    \item We denote by $\mathcal{L}^{d}$ the $d$-dimensional Lebesgue measure on $\R^{d}$.
    \item For notational convenience, we sometimes omit $dx$ and write simply $\int_{A} \varphi$ to denote the integral of a Lebesgue measurable function $\varphi$ defined on a Lebesgue measurable $A \sub \R^{d}$ with respect to the Lebesgue measure $\mathcal{L}^{d}$. When the domain of integration $A = \R^{d}$, we frequently omit $A$ and write $\int \varphi = \int_{\R^{d}} \varphi \, dx$.
    \item We denote by $\n{\cdot}_{L^{p}(A)}$ the usual $L^{p}$ norm on a Lebesgue measurable $A \sub \R^{d}$ with respect to the Lebesgue measure $\mathcal{L}^{d}$. When $A = \R^{d}$, we omit $A$ and simply write $\n{\cdot}_{p} = \n{\cdot}_{L^{p}(\R^{d})}$.
    \item Given an open set $U \sub \R^{d}$ we identify $H^{1}_{0}(U)$ with a subspace of $H^{1}(\R^{d})$ via extension by zero.
    \item Given an open set $U \sub \R^{d}$, we denote by $C_{c}^{\infty}(U)$ (resp. $C_{b}(U)$) the set of $\R$-valued infinitely differentiable functions on $U$ with compact support (resp. the set of $\R$-valued continuous bounded functions on $U$).
    \item Given a point $x \in \R^{d}$ and $r > 0$, we denote by $U(x, r)$ and $B(x, r)$ the open and closed balls of radius $r$ centered at $x$, respectively.
    \item For $a, b \in [-\infty, \infty]$, $a \vee b := \mathrm{max}\{a, b\}$, $a \wedge b := \mathrm{min}\{a, b\}$, $a_{+} := a \vee 0$, and $a_{-} := -(a \wedge 0)$.
    \item We denote by the symbol $C(a_{1}, \cdots, a_{n})$ a finite positive constant that depends only on quantities $a_{1}, \cdots, a_{n}$.
    \item We also write $A \lesssim_{a_{1}, \cdots, a_{n}} B$ if $A, B \in [-\infty, \infty]$ and $A \le C(a_{1}, \cdots, a_{n})B$ holds. The expression $A \asymp_{a_{1}, \cdots, a_{n}} B$ is used if both $A \lesssim_{a_{1}, \cdots, a_{n}} B$ and $B \lesssim_{a_{1}, \cdots, a_{n}} A$ hold.
\end{itemize}
\subsection{Capacitary measures and relaxed Dirichlet problems}
\label{sect: capacitary meas}
In this subsection, we recall the notions of capacitary measures and relaxed Dirichlet problems. We refer the reader to \cite[Chapter 3]{Velichkov2015existence} and \cite[Subsection 3.3]{henrot2018shape} for detailed accounts of these subjects.

For any subset $A \sub \R^{d}$, the capacity of $A$, denoted by $\mathrm{cap}(A)$, is defined by
\begin{equation}
    \mathrm{cap}(A) := \inf\left\{\int |\nabla \varphi |^{2} + |\varphi|^{2} \, dx: \varphi \in C^{\infty}_{c}(\R^{d}), \varphi \ge 1 \text{ on a neighborhood of $A$}\right\}.
\end{equation}
A property $P(x)$ for $x \in \R^{d}$ is said to hold quasi-everywhere (shortly q.e.) if $P(x)$ holds true for every $x \in \R^{d}$ except for points in a set of capacity zero. A subset $U \sub \R^{d}$ is called quasi-open if for every $\varepsilon > 0$ there is an open $V \sub \R^{d}$ such that the symmetric difference $V \triangle U := (V \setminus U) \cup (U \setminus V)$ is of capacity less than $\varepsilon$.

Recall from \cite[Subsection 3.3.4]{henrot2018shape} that for every $u \in H^{1}(\R^{d})$, there is a function $\tilde{u}$ that coincides with $u$ $\mathcal{L}^{d}$-a.e., called a \emph{quasi-continuous representative} of $u$, such that for every $\varepsilon > 0$ there is an open set $U \sub \R^{d}$ of capacity less than $\varepsilon$ such that the restriction $\tilde{u} \restriction \R^{d} \setminus U$ to $\R^{d} \setminus U$ is continuous. It is well known (\cite[Lemma 3.3.30]{henrot2018shape}) that any two quasi-continuous representatives of $u \in H^{1}(\R^{d})$ coincide q.e. We identify $u \in H^{1}(\R^{d})$ with its quasi-continuous representative unless otherwise stated.

\begin{definition}
    \label{def: capacitary measures and H1mu}
    \begin{itemize}
    \item A nonnegative Borel measure $\mu$ on $\R^{d}$ is said to be a capacitary measure if every set of capacity zero is $\mu$-null, and for every Borel $E \sub \R^{d}$, we have
    \begin{equation}
        \label{eq: outer regularity of capacitary measure}
        \mu(E) = \inf\{\mu(A): E \sub A, \text{ $A$ is quasi-open}\}.
    \end{equation}
    \item We denote by $\mathcal{M}_{\mathrm{cap}}$ the set of capacitary measures on $\R^{d}$.
    \item For a quasi-open $U \sub \R^{d}$ we define $I_{U} \in \mathcal{M}_{\mathrm{cap}}$ by 
    \begin{equation}
        \label{def: I_U}
        I_{U}(B) := \begin{cases}
            0 & \text{ if $\mathrm{cap}(B \setminus U) = 0$}\\
            \infty & \text{ else},
        \end{cases}
    \end{equation}
    and we put
    \begin{equation}
        \label{def: McapU}
        \mathcal{M}_{\mathrm{cap}}(U) := \{\mu \in \mathcal{M}_{\mathrm{cap}}: \mu \ge I_{U}\}
    \end{equation}
    and
    \begin{equation}
        \mathcal{M}_{\mathrm{cap}}^{b} := \bigcup_{n = 1}^{\infty}\mathcal{M}_{\mathrm{cap}}(U(0, n)).
    \end{equation}
    \item By an abuse of notation, we write $U$ for $I_{U}$ when $U$ is quasi-open, and we say that $\mu \in \mathcal{M}_{\mathrm{cap}}$ is open (resp. quasi-open) if $\mu = I_{U}$ for an open (resp. quasi-open) $U \sub \R^{d}$.
    \item For $\mu \in \mathcal{M}_{\mathrm{cap}}^{b}$, $H^{1}_{\mu}$ denotes the Hilbert space $H^{1}(\R^{d}) \cap L^{2}(\mu)$ endowed with the norm
    \begin{equation}
        \label{def: dotH1mu}
        \n{u}_{\dot{H}^{1}_{\mu}}^{2} := \int_{\R^{d}} |\nabla u|^{2} \, dx + \int |u|^{2} \, d\mu.
    \end{equation}
\end{itemize}    
\end{definition}

\begin{remark}
    \label{rmk: open set as a capacitary meas}
    If $U \sub \R^{d}$ is open, then $H^{1}(\R^{d}) \cap L^{2}(I_{U}) = H^{1}_{0}(U)$, the closure of $C_{c}^{\infty}(U)$ in $H^{1}(U)$ (see, e.g., \cite[Theorem 3.3.42]{henrot2018shape}). Hence, for $\mu \in \mathcal{M}_{\mathrm{cap}}^{b}$, $\n{\cdot}_{\dot{H}^{1}_{\mu}}$ is a norm on $H^{1}_{\mu}$ equivalent to $\n{\cdot}_{\dot{H}^{1}_{\mu}} + \n{ \cdot}_{L^{2}(\R^{d})}$. That the space $H^{1}_{\mu}$ is complete with respect to these norms is an easy consequence of the completeness of $H^{1}(\R^{d})$ and $L^{2}(\mu)$, and the fact that every sequence converging to $u \in H^{1}(\R^{d})$ in $H^{1}(\R^{d})$ strong has a subsequence converging to $u$ q.e. (and hence $\mu$-a.e.) (for the proof of this fact, see, e.g., \cite[Proposition 3.3.33]{henrot2018shape}). A complete proof of the completeness of $H^{1}_{\mu}$ can be found in \cite[Proposition 3.18]{Velichkov2015existence}.
\end{remark}

\begin{remark}
    Let $\mu \in \mathcal{M}_{\mathrm{cap}}^{b}$.
    \begin{itemize}
        \item We have the continuous embedding $H^{1}_{\mu} \hookrightarrow L^{p}(\R^{d})$ and hence the dual map $L^{p'}(\R^{d}) \to (H^{1}_{\mu})^{*}$, which is never injective  since $H^{1}_{\mu} \sub H^{1}_{0}(U(0, n))$ for some $n$, for $1 \le p \le \frac{2d}{d - 2}$ if $d \ge 3$ and for $1 \le p < \infty$ if $d = 2$.
        \item More generally, if $U \sub \R^{d}$ is open with $\mu \ge I_{U}$, then we have the continuous embedding $H^{1}_{\mu} \hookrightarrow H^{1}_{0}(U)$, and thus the (not necessarily injective) dual map $H^{-1}(U) \to (H^{1}_{\mu})^{*}$.
    \end{itemize}
\end{remark}

Next, we define the so-called \emph{relaxed Dirichlet problem} associated with $\mu \in \mathcal{M}_{\mathrm{cap}}^{b}$. For $0 < \alpha \le \beta < \infty$, we put
\begin{equation}
    \label{def: Mab}
    M(\alpha, \beta) := \left\{ A \in L^{\infty}(\R^{d})^{d \times d}: A\xi \cdot \xi \ge \alpha |\xi|^{2} \text{ and } A^{-1}\xi\cdot \xi \ge \beta^{-1}|\xi|^{2} \text{ $\mathcal{L}^{d}$-a.e. for every $\xi \in \R^{d}$} \right\}.
\end{equation}

\begin{definition}[Relaxed Dirichlet problem]
    \label{def: relaxed_Dirichlet}
    Let $\mu \in \mathcal{M}_{\mathrm{cap}}^{b}$ and $A \in M(\alpha, \beta)$. Given $f \in (H^{1}_{\mu})^{*}$, we say that $u \in H^{1}_{\mu}$ is a solution to 
    \begin{equation}
        \label{eq: relaxed_Dirichlet}
        -\mathrm{div}(A \nabla u) + \mu u = f \text{ in $(H^{1}_{\mu})^{*}$, $u \in H^{1}_{\mu}$}
    \end{equation}
    if
    \begin{equation}
        \label{eq: weak formulation of relaxed Dirichlet}
        \int A\nabla u \cdot \nabla \varphi + \int u \varphi \, d\mu = \inn{f}{\varphi}_{(H^{1}_{\mu})^{*}, H^{1}_{\mu}}
    \end{equation}
    for every $\varphi \in H^{1}_{\mu}$. By the Lax--Milgram theorem, we see that \eqref{eq: weak formulation of relaxed Dirichlet} has a unique solution $u \in H^{1}_{\mu}$, which we denote by $R_{A, \mu}f$. We call \eqref{eq: relaxed_Dirichlet} the \emph{relaxed Dirichlet problem} associated with $A$ and $\mu$, and denote it by $P(A, \mu)$.
\end{definition}

\begin{remark}
    \label{rmk: relaxed dirichlet for I_U}
     The solution to the problem \eqref{eq: relaxed_Dirichlet} with $\mu = I_{U}$, where $U \sub \R^{d}$ is bounded open, and with $f \in H^{-1}(U)$ is the usual weak solution of $-\mathrm{div}(A\nabla u) = f$ in $U$ with zero Dirichlet boundary condition.
\end{remark}
\begin{remark}[Scaling properties]
    \label{rmk: scaling property}
    Clearly we have the following scaling properties. Let $t > 0$, $A \in M(\alpha, \beta)$, $\mu \in \mathcal{M}_{\mathrm{cap}}^{b}$, and $f \in (H^{1}_{\mu})^{*}$. Then
    \begin{equation}
        R_{t^{-1}A, t^{-1}\mu}f = tR_{A, \mu}f.
    \end{equation}
    Further, if $\mu = I_{U}$ is open and $f \in H^{-1}(tU)$, then
    \begin{equation}
        \label{eq: scaling property of resolvent}
        R_{A(t \cdot), U}(f(t \cdot)) = t^{-2}(R_{A, tU}f)(t \cdot).
    \end{equation}
    These relations are used extensively throughout the paper.
\end{remark}

We have the following comparison principle.
\begin{lemma}[{{\cite[Theorem 2.10]{dalmaso1986wiener}}, see also \cite[Proposition 2.5]{dalmaso1994new}}]
    \label{thm: comparison}
    Let $0 < \alpha \le \beta < \infty$, $A \in M(\alpha, \beta)$, and $D \sub \R^{d}$ be bounded open. Suppose that $f_{1}, f_{2} \in H^{-1}(D)$ and $\mu_{1}, \mu_{2} \in \mathcal{M}_{\mathrm{cap}}(D)$ satisfy $0 \le f_1 \le f_2$ and $\mu_{2} \le \mu_{1}$. Then $0 \le R_{A, \mu_{1}}f_{1} \le R_{A, \mu_{2}}f_{2}$ q.e.
\end{lemma}

\subsection{Torsion functions}
\begin{definition}[Torsion function]
    \label{def: torsion fct}
    Given $\mu \in \mathcal{M}_{\mathrm{cap}}^{b}$ and $A \in M(\alpha, \beta)$, we denote by $w_{A, \mu} \in H^{1}_{\mu}$ the unique solution of \eqref{eq: relaxed_Dirichlet} with $f = 1$ and call it the \emph{torsion function} with respect to $A$ and $\mu$. Note that $w_{A, \mu} \ge 0$ q.e. by Lemma \ref{thm: comparison}. When $A$ is the identity matrix $\mathcal{L}^{d}$-a.e., we simply write $w_{\mu} = w_{A, \mu}$. The quantity
    \begin{equation}
        \label{eq: torsional rigidity}
        T(\mu) := \int w_{\mu} = \int |\nabla w_{\mu}|^{2} + \int w_{\mu}^{2} \, d\mu
    \end{equation}
    is called the \emph{torsional rigidity} of $\mu$.
     The q.e. defined quasi-open set
\begin{equation}
        \Omega_{\mu} := \{w_{\mu} > 0\}
\end{equation}
is called \emph{the set of finiteness} of $\mu$ (see \cite[Definition 3.73]{Velichkov2015existence}). 
\end{definition}

\begin{remark}
    \begin{itemize}
        \item When $\mu$ is of the form $\mu = I_{U} + V \mathcal{L}^{d}$ with $U$ open and $V \in L^{1}_{\mathrm{loc}}(\R^{d})$ nonnegative, the torsion function $w_{A, \mu}$ is also called the landscape function and is used to predict the localization of the eigenfunctions of the operator $-\mathrm{div}A \nabla + V$. For this line of research, see, e.g., \cite{filoche2012universal,arnold2019localization,david2021the}.
        \item Our definition of the torsional rigidity $T(\mu)$ is twice the quantity in \cite[Definition 3.34]{Velichkov2015existence}.
        \item For $\mu \in \mathcal{M}_{\mathrm{cap}}^{b}$ and $u \in H^{1}_{\mu}$, it holds that $u = 0$ q.e. on $\R^{d} \setminus \Omega_{\mu}$ (see \cite[Proposition 3.72]{Velichkov2015existence}).
    \end{itemize}
\end{remark}

The next estimates are direct consequences of Definition \ref{def: torsion fct}.
\begin{lemma}
    \label{thm: energy bound by torsion and L1 norm of wA}
    Let $A \in M(\alpha, \beta)$ and $\mu \in \mathcal{M}_{\mathrm{cap}}^{b}$. Then 
    \begin{enumerate}
        \item \begin{equation}
        \label{eq: energy bound by torsion}
        \n{A\nabla w_{A, \mu}}_{2}^{2} + \beta \n{w_{A, \mu}}_{L^{2}(\mu)}^{2} \le \beta \n{w_{A, \mu}}_{1}.
    \end{equation}
    \item If either $\alpha \le 1 \le \beta$ or $\mu$ is quasi-open, then 
    \begin{equation}
        \label{eq: L1 norm of wA}
        \beta^{-1}T(\mu) \le \n{w_{A, \mu}}_{1} \le \alpha^{-1} T(\mu).
    \end{equation}
    \end{enumerate}
\end{lemma}
\begin{proof}
    Since $A \in M(\alpha, \beta)$, $|A\eta|^{2} \le \beta A^{-1}A\eta \cdot A\eta = \beta A\eta \cdot \eta$ $\mathcal{L}^{d}$-a.e. for every $\eta \in \R^{d}$. Hence, 
    \[\int|A\nabla w_{A, \mu}|^{2} + \beta \n{w_{A, \mu}}_{L^{2}(\mu)}^{2} \le \beta \left( \int A \nabla w_{A, \mu} \cdot \nabla w_{A, \mu} + \int w_{A, \mu}^{2} \, d\mu \right) = \beta \n{w_{A, \mu}}_{1}.\]
    Observe that $\n{u}_{L^{2}(\mu)}^{2} \le \eta \n{u}_{L^{2}(\mu)}^{2}$ for every $u \in L^{2}(\mu)$ whenever $\eta \ge 1$ or $\mu$ is quasi-open. Indeed, every $\mu$-integrable function vanishes $\mu$-a.e. if $\mu$ is quasi-open. Hence, if either $\alpha \le 1 \le \beta$ or $\mu$ is quasi-open, we have
    \begin{equation*}
        \begin{split}
            T(\mu) &= \int A\nabla w_{A, \mu} \cdot \nabla w_{\mu} + \int w_{A, \mu}w_{\mu} \, d\mu \le \left( \n{A \nabla w_{A, \mu}}_{2}^{2} + \n{w_{A, \mu}}_{L^{2}(\mu)}^{2} \right)^{\frac{1}{2}}T(\mu)^{\frac{1}{2}}\\
            &\le \beta^{\frac{1}{2}}\n{w_{A, \mu}}_{1}^{\frac{1}{2}}T(\mu)^{\frac{1}{2}}
        \end{split}
    \end{equation*}
    and similarly
    \[\begin{split}
        \n{w_{A, \mu}}_{1} &= \int \nabla w_{\mu} \cdot \nabla w_{A, \mu} + \int w_{\mu}w_{A, \mu} \, d\mu \le \n{w_{\mu}}_{\dot{H}^{1}_{\mu}}\left( \alpha^{-1}\int A\nabla w_{A, \mu} \cdot \nabla w_{A, \mu} + \int w_{A, \mu}^{2} \, d\mu \right)^{\frac{1}{2}}\\
        &\le \alpha^{-\frac{1}{2}}T(\mu)^{\frac{1}{2}}\n{w_{A, \mu}}_{1}^{\frac{1}{2}}.
    \end{split}\]
    Combining the last two estimates yields \eqref{eq: L1 norm of wA}.
\end{proof}

The following lemmas were proved in \cite{brasco2017compact} by Brasco and Ruffini and \cite{vandenBerg2014on} by van den Berg and Bucur for torsion functions with respect to the $p$-Laplacian $(1 < p < \infty)$ and an open $\mu \in \mathcal{M}_{\mathrm{cap}}$. The proofs carry over with minor changes to general $A$ and $\mu$. We include them for completeness.

\begin{lemma}[cf. {\cite[Proposition 3.1]{brasco2017compact}}]
    \label{thm: Linfty_Lt_estimate_of_torsion_func}
    Let $A \in M(1, \beta)$ and $\mu \in \mathcal{M}_{\mathrm{cap}}^{b}$. Then
    \begin{equation}
        \label{eq: Ls_Lt estimate of torsion func}
        \n{w_{A, \mu}}_{s} \lesssim_{d} \n{w_{A, \mu}}_{t}^{\gamma_{t} + \frac{t}{s}(1 - \gamma_{t})}
    \end{equation}
    for every $1 \le t \le s \le \infty$. Here, 
    \begin{equation}
        \label{def: gamma}
        \gamma_{t} := \frac{2t}{d + 2t}.
    \end{equation}
    In particular, 
    \begin{equation}
        \label{eq: Ls_L1 estimate of torsion func}
        \n{w_{A, \mu}}_{s} \lesssim_{d} \n{w_{A, \mu}}_{1}^{\gamma_{1} + \frac{1}{s}(1 - \gamma_{1})} \le T(\mu)^{\gamma_{1} + \frac{1}{s}(1 - \gamma_{1})}
    \end{equation}
    for every $1 \le s \le \infty$.
\end{lemma}

\begin{proof}
    \eqref{eq: Ls_Lt estimate of torsion func} follows from Lemma \ref{thm: Linfty Lr estimate for solutions} and an interpolation. Bounding the right-hand side of \eqref{eq: Ls_Lt estimate of torsion func} for $t = 1$ by using \eqref{eq: L1 norm of wA}, we get \eqref{eq: Ls_L1 estimate of torsion func}.
\end{proof}

\begin{lemma}[Torsional Hardy inequalities, cf. {\cite[Proposition 4.1]{brasco2017compact}}]
    \label{thm: torsional_Hardy}
    Let $0 < \alpha \le \beta < \infty$, $A \in M(\alpha, \beta)$, and $\mu \in \mathcal{M}_{\mathrm{cap}}^{b}$. Then for every $u \in H^{1}_{\mu}$, $u = 0$ $\mathcal{L}^{d}$-a.e. on $\{w_{A, \mu} = 0\}$, 
    \begin{equation}
        \label{eq: torsional_Hardy_1}
        \int \frac{u^{2}}{w_{A, \mu}} \le \beta \int |\nabla u|^{2} + \int u^{2} \, d\mu,
    \end{equation}
    and 
    \begin{equation}
        \label{eq: torsional_Hardy_2}
        \alpha\int \left| \frac{\nabla w_{A, \mu}}{w_{A, \mu}} \right|^{2}u^{2} \le \int \frac{A \nabla w_{A, \mu} \cdot \nabla w_{A, \mu}}{w_{A, \mu}^{2}}u^{2} \le 4 \beta \int |\nabla u|^{2} + 2\int u^{2} \, d\mu.
    \end{equation}
\end{lemma}

\begin{proof}
    Notice that given $u \in H^{1}_{\mu}$, $(u \wedge k)\vee(-k) \to u$ in $H^{1}_{\mu}$ as $k \to \infty$. Thus, we may assume that $u \in H^{1}_{\mu}$ is bounded. Fix $\varepsilon >0$ temporarily. Using $\varphi = \frac{u^{2}}{w_{A, \mu} + \varepsilon} \in H^{1}_{\mu}$ as a test function for \eqref{eq: weak formulation of relaxed Dirichlet} with $f = 1$, we have
    \[\int \frac{2uA\nabla w_{A, \mu} \cdot \nabla u}{w_{A, \mu} + \varepsilon} - \int \frac{A\nabla w_{A, \mu}\cdot \nabla w_{A, \mu}}{(w_{A, \mu} + \varepsilon)^{2}}u^{2} + \int \frac{w_{A, \mu}}{w_{A, \mu} + \varepsilon}u^{2} \, d\mu = \int \frac{u^{2}}{w_{A, \mu} + \varepsilon}.\]
    Hence, by Young's inequality and the assumption $A \in M(\alpha, \beta)$, 
    \begin{equation}
        \begin{split}
            \int \frac{A \nabla w_{A, \mu} \cdot \nabla w_{A, \mu}}{(w_{A, \mu} + \varepsilon)^{2}}u^{2} + \int \frac{u^{2}}{w_{A, \mu} + \varepsilon} &\le \delta \int \frac{|A \nabla w_{A, \mu}|^{2}}{(w_{A, \mu} + \varepsilon)^{2}}u^{2} + \frac{1}{\delta}\int |\nabla u|^{2} + \int\frac{w_{A, \mu}}{w_{A, \mu} + \varepsilon}u^{2} \, d\mu\\
            &\le \delta\beta \int \frac{A \nabla w_{A, \mu} \cdot \nabla w_{A, \mu}}{(w_{A, \mu} + \varepsilon)^{2}}u^{2} + \frac{1}{\delta}\int  |\nabla u|^{2} + \int u^{2} \, d\mu
        \end{split}
    \end{equation}
    for any $\delta > 0$. By letting $\delta = \beta^{-1}$ and $\varepsilon \searrow 0$, we see that $u = 0$ $\mathcal{L}^{d}$-a.e. on $\{w_{A, \mu} = 0\}$ and that \eqref{eq: torsional_Hardy_1} holds.
    On the other hand, letting $\delta = 2^{-1}\beta^{-1}$ in the right-hand side of 
    \[\int  \frac{A\nabla w_{A, \mu} \cdot \nabla w_{A, \mu}}{(w_{A, \mu} + \varepsilon)^{2}}u^{2} \le \frac{1}{\delta(1 - \delta\beta)}\int  |\nabla u|^{2} + \frac{1}{1 - \delta\beta}\int u^{2} \, d\mu\]
    and letting $\varepsilon \searrow 0$, we obtain the second inequality of \eqref{eq: torsional_Hardy_2}.
\end{proof}

Given $\mu \in \mathcal{M}_{\mathrm{cap}}^{b}$, define 
    \begin{equation}
        \lambda_{q}^{(\eta)}(\mu) := \inf_{u \in H^{1}_{\mu}, u \neq 0}\frac{\eta\n{\nabla u}_{2}^{2} + \n{u}_{L^{2}(\mu)}^{2}}{\n{u}_{q}^{2}}
    \end{equation}
    for $1 \le q \le 2$ and $\eta > 0$. We omit $\eta$ and write $\lambda_{q}(\mu) = \lambda_{q}^{(\eta)}(\mu)$ if $\eta = 1$.

\begin{lemma}[cf. {\cite[Theorem 1.2 and 1.3]{brasco2017compact}}, {\cite[Theorem 9]{vandenBerg2014on}}]
    \label{thm: torsion and poincare const}
    Let $\mu \in \mathcal{M}_{\mathrm{cap}}^{b}$ and $A \in M(\alpha, \beta)$, $0 < \alpha \le \beta < \infty$. 
    Then $w_{A, \mu} = 0$ if and only if $H^{1}_{\mu}$ is the null space. Further, if $w_{A, \mu} \neq 0$, then we have
    \begin{equation}
        \label{eq: torsion poincare q < 2}
        1 \le \lambda_{q}^{(\beta)}(\mu)\n{w_{A, \mu}}_{\frac{q}{2 - q}} \quad\text{and}\quad \lambda_{q}^{(\alpha)}(\mu)\n{w_{A, \mu}}_{\frac{q}{2 - q}} \le \frac{1}{q}\frac{1}{2 - q}
    \end{equation}
    if $1 \le q < 2$, and
    \begin{equation}
        \label{eq: torsion poincare q = 2}
        1\le \lambda_{2}^{(\beta)}(\mu)\n{w_{A, \mu}}_{\infty} \quad\text{and}\quad \lambda_{2}^{(\alpha)}(\mu)\n{w_{A, \mu}}_{\infty} \le C\left(d, \frac{\beta}{\alpha}\right)
    \end{equation}
    if $q = 2$.
    
    In particular, 
    \begin{equation}
        \n{w_{A, \mu}}_{r} \asymp_{d, \alpha, \beta, r} \lambda_{\frac{2r}{1 + r}}(\mu)^{-1} \asymp_{d, r} \n{w_{\mu}}_{r}
    \end{equation}
    for every $1 \le r \le \infty$.
\end{lemma}

\begin{proof}
    By Lemma \ref{thm: torsional_Hardy}, $w_{A, \mu} = 0$ implies $H^{1}_{\mu} = (0)$. The other implication is obvious. 
    
    For $1 \le q \le 2$ and $u \in H^{1}_{\mu}$, 
    \begin{equation}
        \begin{split}
            \int  |u|^{q} = \int  \frac{|u|^{q}}{w_{A, \mu}^{\frac{q}{2}}}w_{A, \mu}^{\frac{q}{2}} &\le \left( \int  \frac{|u|^{2}}{w_{A, \mu}} \right)^{\frac{q}{2}}\n{w_{A, \mu}}_{\frac{q}{2}(\frac{2}{q})'}^{\frac{q}{2}}\\
            &\le \left( \beta\int  |\nabla u|^{2} + \int u^{2} \, d\mu \right)^{\frac{q}{2}}\n{w_{A, \mu}}_{\frac{q}{2 - q}}^{\frac{q}{2}},
        \end{split}
    \end{equation}
    where we have used the H\"{o}lder inequality for the first inequality and \eqref{eq: torsional_Hardy_1} for the second.
    Hence, the lower bounds in \eqref{eq: torsion poincare q < 2} and \eqref{eq: torsion poincare q = 2} follow.

    Next, we show the upper bound in \eqref{eq: torsion poincare q < 2}. Let $1 \le q < 2$. Fix $\varkappa \ge 1$ temporarily. Since $w_{A, \mu} \in L^{\infty}(\R^{d}) \cap H^{1}_{\mu}$, $w_{A, \mu}^{\varkappa} \in H^{1}_{\mu}$. Hence, 
    \begin{equation}
        \begin{split}
            \int  w_{A, \mu}^{\varkappa} &= \int  A \nabla (w_{A, \mu}) \cdot \nabla (w_{A, \mu}^{\varkappa}) + \int w_{A, \mu}^{\varkappa + 1} \, d\mu\\
            &= \varkappa \left( \frac{2}{\varkappa + 1} \right)^{2}\int  A \nabla (w_{A, \mu}^{\frac{\varkappa + 1}{2}}) \cdot \nabla (w_{A, \mu}^{\frac{\varkappa + 1}{2}}) + \int |w_{A, \mu}^{\frac{\varkappa + 1}{2}}|^{2} \, d\mu\\
            &\ge \varkappa \left( \frac{2}{\varkappa + 1} \right)^{2}\left( \alpha \int |\nabla w_{A, \mu}^{\frac{\varkappa + 1}{2}}|^{2} + \int |w_{A, \mu}^{\frac{\varkappa + 1}{2}}|^{2} \, d\mu \right)\\
            &\ge \varkappa \left( \frac{2}{\varkappa + 1} \right)^{2}\lambda_{q}^{(\alpha)}(\mu)\n{w_{A, \mu}^{\frac{\varkappa + 1}{2}}}_{q}^{2}\\
            &= \varkappa \left( \frac{2}{\varkappa + 1} \right)^{2}\lambda_{q}^{(\alpha)}(\mu)\n{w_{A, \mu}}_{q\frac{\varkappa + 1}{2}}^{\varkappa + 1}.
        \end{split}
    \end{equation}
    Thus, by letting $q\frac{\varkappa + 1}{2} = \varkappa$, namely $\varkappa = \frac{q}{2 - q}$, we obtain the upper bound in \eqref{eq: torsion poincare q < 2}.

    Finally, we show the upper bound in \eqref{eq: torsion poincare q = 2}. To this end, we may assume that $\alpha = 1$. We know from \cite[Proposition 2.6]{dalmaso1986wiener} (see also \cite[Proposition 2.6]{dalmaso1994new}) that $-\mathrm{div}(A \nabla w_{A, \mu}) \le 1$ in $H^{-1}(\R^{d})$. Indeed, given a nonnegative $\varphi \in C^{\infty}_{c}(\R^{d})$, $\varphi_{h} := \varphi\wedge (h w_{A, \mu}) \in H^{1}_{\mu}$ for every $h > 0$, and thus
    \[\begin{split}
        \int_{\R^{d}} \varphi &\ge \int_{\R^{d}} \varphi_{h} = \int_{\R^{d}} A \nabla w_{A, \mu} \cdot \nabla \varphi_{h} + \int w_{A, \mu} \varphi_{h} \, d\mu\ \ge \int_{\varphi < hw_{A, \mu}}A\nabla w_{A, \mu} \cdot \nabla \varphi\\
        &\xrightarrow{h \to \infty} \int_{w_{A, \mu} > 0}A \nabla w_{A, \mu} \cdot \nabla \varphi = \int_{\R^{d}}A \nabla w_{A, \mu} \cdot \nabla \varphi.
    \end{split}\]

    Therefore, by \cite[Theorem 8.17]{gilbarg1983elliptic}, there is a finite constant $C_{1} = C(d, \beta) > 0$ such that 
    \begin{equation*}
        \n{w_{A, \mu}}_{L^{\infty}(U(y, \frac{r}{2}))} \le C_{1}\left(r^{-\frac{d}{2}}\n{w_{A, \mu}}_{L^{2}(U(y, r))} + r^{2}\right)
    \end{equation*}
    for every $y \in \R^{d}$ and $r > 0$. Now, let $0 < \eta < \n{w_{A, \mu}}_{L^{\infty}(\R^{d})}$ and fix $y \in \R^{d}$ such that $\n{w_{A, \mu}}_{L^{\infty}(U(y, r))} > \eta$ for every $r > 0$. By choosing $r > 0$ so that $r^{2} = (2C_{1})^{-1}\eta$, we obtain
    \begin{equation}
        \label{eq: denominator estimate of rayleigh q}
        r^{-\frac{d}{2}}\n{w_{A, \mu}}_{L^{2}(U(y, r))} \ge C(d, \beta) \eta.
    \end{equation}

    On the other hand, by Lemma \ref{thm: cacciopoli torsion} below as applied to $\theta \in C^{\infty}_{c}(U(y, 2r))$ such that $\chi_{U(y, r)} \le \theta \le 1$ and $|\nabla \theta| \le \frac{2}{r}$, we see that
    \begin{equation}
        \label{eq: numerator estimate of rayleigh q}
        \n{\theta w_{A, \mu}}_{\dot{H}^{1}_{\mu}}^{2} \le C(\beta)\left( \int_{U(y, 2r)}w_{A, \mu} + \frac{1}{r^{2}}\int_{U(y, 2r)}w_{A, \mu}^{2} \right).
    \end{equation}
    Combining \eqref{eq: denominator estimate of rayleigh q} and \eqref{eq: numerator estimate of rayleigh q}, we bound 
    \begin{equation*}
        \begin{split}
            \lambda_{2}(\mu) &\le  \frac{\n{\theta w_{A, \mu}}_{\dot{H}^{1}_{\mu}}^{2}}{\n{\theta w_{A, \mu}}_{2}^{2}} \le C(d, \beta)\frac{\fint_{U(y, 2r)}w_{A, \mu} + \frac{1}{r^{2}}\fint_{U(y, 2r)}w_{A, \mu}^{2}}{r^{-d}\n{w_{A, \mu}}_{L^{2}(U(y, r))}^{2}}
            \le C(d, \beta)\frac{\n{w_{A, \mu}}_{\infty} + \eta^{-1}\n{w_{A, \mu}}_{\infty}^{2}}{\eta^{2}}.
        \end{split}
    \end{equation*}
    Letting $\eta \nearrow \n{w_{A, \mu}}_{\infty}$ in the last display, we get the upper bound for \eqref{eq: torsion poincare q = 2}.
\end{proof}

\begin{lemma}[cf. {\cite[Lemma 10]{vandenBerg2014on}}]
    \label{thm: cacciopoli torsion}
    Let $A \in M(1, \beta)$, $\mu \in \mathcal{M}_{\mathrm{cap}}^{b}$, and $\theta \in C^{\infty}_{c}(\R^{d})$. Then
    \begin{equation}
        \n{\theta w_{A, \mu}}_{\dot{H}^{1}_{\mu}}^{2} \le C(\beta)\left( \int w_{A, \mu}\theta^{2} + \int w_{A, \mu}^{2}|\nabla \theta|^{2} \right).
    \end{equation}
\end{lemma}

\begin{proof}
    By the coercivity of $A$ and the Leibniz rule, 
    \[\int \theta^{2}|\nabla w_{A, \mu}|^{2} \le \int \theta^{2}A\nabla w_{A, \mu} \cdot \nabla w_{A, \mu} = \int A\nabla w_{A, \mu} \cdot \nabla (w_{A, \mu} \theta^{2}) - 2\int w_{A, \mu}\theta A \nabla w_{A, \mu} \cdot \nabla \theta.\]
    By applying Young's inequality to the second term of the right-hand side, we get
    \[\int \theta^{2}|\nabla w_{A, \mu}|^{2} \le C(\beta)\left( \int A\nabla w_{A, \mu} \cdot \nabla (w_{A, \mu} \theta^{2}) + \int w_{A, \mu}^{2}|\nabla \theta|^{2} \right).\]
    Hence, 
    \begin{equation*}
        \begin{split}
            \n{\theta w_{A, \mu}}_{\dot{H}^{1}_{\mu}}^{2} &\le C(\beta)\left( \int A \nabla w_{A, \mu} \cdot \nabla (w_{A, \mu}\theta^{2}) + \int w_{A, \mu}^{2}\theta^{2} \, d\mu + \int w_{A, \mu}^{2}|\nabla \theta|^{2} \right)\\
            &= C(\beta)\left( \int w_{A, \mu}\theta^{2} + \int w_{A, \mu}^{2}|\nabla \theta|^{2} \right).
        \end{split}
    \end{equation*}
\end{proof}
\section{Main results}
\label{sect: main_results}
\subsection{Convergence results and quantitative error estimates}
\label{sect: convergence and error}
Let $0 < \alpha \le \beta < \infty$, $(A^{\varepsilon}) \sub M(\alpha, \beta)$, $D \sub \R^{d}$ nonempty bounded open, and $(\mu^{\varepsilon}) \sub \mathcal{M}_{\mathrm{cap}}(D)$. Then Dal Maso and Murat showed in \cite[Section 5]{dalmaso2004asymptotic} that there is a subsequence $(\varepsilon)$, $A \in M(\alpha, \beta)$, and $\mu \in \mathcal{M}_{\mathrm{cap}}(D)$ such that for every sequence $(f^{\varepsilon}) \sub H^{-1}(D)$ strongly converging to $f \in H^{-1}(D)$ in $H^{-1}(D)$, the solution $u^{\varepsilon} \in H^{1}_{\mu^{\varepsilon}}$ to the relaxed Dirichlet problem \eqref{eq: relaxed_Dirichlet} for $A^{\varepsilon}$, $f^{\varepsilon}$, and $\mu^{\varepsilon}$ converges weakly in $H^{1}_{0}(D)$ to the solution to \eqref{eq: relaxed_Dirichlet} for $A$, $f$, and $\mu$. We study the finer asymptotics when this limit vanishes. Recall from \cite[Proposition 3.4]{dalmaso1994new} that $T(\nu) \neq 0$ whenever $\nu \in \mathcal{M}_{\mathrm{cap}}^{b} \setminus \{I_{\varnothing}\}$.

\begin{theorem}
    \label{thm: main result}
    Suppose that $\mu = I_{\varnothing}$, i.e., that $T(\mu^{\varepsilon}) \to 0$, which is also equivalent to the assumption that $R_{A^{\varepsilon}, \mu^{\varepsilon}}f \xrightharpoonup{H^{1}_{0}(D)} 0$ for every $f \in H^{-1}(D)$ (see \cite[Theorem 5.1]{dalmaso2004asymptotic}). Suppose also that $\mu^{\varepsilon} \neq I_{\varnothing}$ for every $\varepsilon$. Then there is a subsequence, still denoted by $(\varepsilon)$, and a nonnegative Borel measure $k$ supported in $\overline{D}$ with $\beta^{-1} \wedge 1 \le k(\overline{D}) \le \alpha^{-1} \vee 1$ such that for every $f \in C_{b}(\R^{d})$ 
    \begin{equation}
        \label{eq: convergence to kf}
        \frac{R_{A^{\varepsilon}, \mu^{\varepsilon}}f}{T(\mu^{\varepsilon})} \xrightharpoonup{\mathcal{M}(\R^{d})} kf
    \end{equation}
    weakly in $\mathcal{M}(\R^{d})$ with a qualitative corrector result
    \begin{equation}
        \label{eq: qualitative corrector}
        \n{\frac{R_{A^{\varepsilon}, \mu^{\varepsilon}}f}{T(\mu^{\varepsilon})} - \frac{w_{A^{\varepsilon}, \mu^{\varepsilon}}}{T(\mu^{\varepsilon})}f}_{L^{1}(\R^{d})} \to 0
    \end{equation}
    as $\varepsilon \searrow 0$. Further, if $p, q \in [1, 2]$ satisfy
    \begin{equation}
        \label{eq: condition on p q}
        \frac{1}{d} - \left(\frac{1}{p'} + \frac{1}{q'}\right) > 0,
    \end{equation}
    then we have a quantitative error estimate
    \begin{equation}
        \label{eq: quantitative error mue}
        \n{\frac{R_{A^{\varepsilon}, \mu^{\varepsilon}}f}{T(\mu^{\varepsilon})} - \frac{w_{A^{\varepsilon}, \mu^{\varepsilon}}}{T(\mu^{\varepsilon})}f}_{L^{q}(\R^{d})} \lesssim_{d, \alpha, \beta}\n{\nabla f}_{L^{p'}(\Omega_{\mu^{\varepsilon}})}T(\mu^{\varepsilon})^{\delta_{d, p, q}}
    \end{equation}
    for every $f \in C_{b}(\R^{d}) \cap W^{1, p'}(D)$. Here,
    \begin{equation}
        \label{def: delta dpq}
        \delta_{d, p, q} := \frac{d}{d + 2}\left( \frac{1}{d} - \left( \frac{1}{p'} + \frac{1}{q'} \right) \right).
    \end{equation}
    If $\mu^{\varepsilon}$ are quasi-open, then $\beta^{-1} \le k(\overline{D}) \le \alpha^{-1}$.
\end{theorem}

Theorem \ref{thm: main result} is an immediate consequence of the following $\varepsilon$-independent estimate.

\begin{theorem}
    \label{thm: quantitative error estimate}
    Let $p, q \in [1, 2]$, $A \in M(1, \beta)$, $D \sub \R^{d}$ bounded open, $\mu \in \mathcal{M}_{\mathrm{cap}}(D) \setminus \{I_{\varnothing}\}$, and
    \begin{equation}
        \label{eq: admissible integrability of source}
        f \in \begin{cases}
            L^{\frac{2d}{d + 2}}(D) & \text{ if $d \ge 3$}\\
            L^{1+}(D) := \bigcup_{\delta > 0}L^{1 + \delta}(D) & \text{ if $d = 2$}.
        \end{cases}
    \end{equation}
    Assume that $\nabla f \in L^{p'}(D)$. Then
    \begin{equation}
        \label{eq: quantitative error estimate}
        \begin{split}
            \n{R_{A, \mu}f - w_{A, \mu}f}_{L^{q}(D)} &\le 3 \beta \n{\nabla f}_{L^{p'}(\Omega_{\mu})}\lambda_{q}(\mu)^{-\frac{1}{2}}\n{w_{A, \mu}}_{\frac{2p}{2 - p}}\\
            &\lesssim_{d}\beta\n{\nabla f}_{L^{p'}(\Omega_{\mu})}T(\mu)^{1 + \delta_{d, p, q}}.
        \end{split}
    \end{equation}
\end{theorem}

We postpone the proof of Theorem \ref{thm: quantitative error estimate} to Subsection \ref{sect: pf of quantitative error estimate} and first complete the proof of Theorem \ref{thm: main result}.

\begin{proof}[Proof of Theorem \ref{thm: main result}]
    The error estimate \eqref{eq: quantitative error mue} follows from Theorem \ref{thm: quantitative error estimate} (see Remark \ref{rmk: scaling property}). Note that for any $h \in L^{\infty}(\R^{d})$, $A \in M(\alpha, \beta)$, and $\mu \in \mathcal{M}_{\mathrm{cap}}(D)$, we have
    \[\n{R_{A, \mu}h}_{1} \underset{\eqref{eq: torsion poincare q < 2}}{\le}T(\mu)^{\frac{1}{2}}\n{R_{A, \mu}h}_{\dot{H}^{1}_{\mu}} \underset{\eqref{eq: weak formulation of relaxed Dirichlet}}{\lesssim_{\alpha}} T(\mu)^{\frac{1}{2}}\n{h R_{A, \mu}h}_{1}^{\frac{1}{2}} \le T(\mu)^{\frac{1}{2}}\n{h}_{L^{\infty}(\Omega_{\mu})}^{\frac{1}{2}}\n{R_{A, \mu}h}_{1}^{\frac{1}{2}},\]
    and thus
    \begin{equation}
        \label{eq: a priori estimate}
        \n{R_{A, \mu}h}_{1} \lesssim_{\alpha} T(\mu)\n{h}_{L^{\infty}(\Omega_{\mu})}.
    \end{equation}
    Hence, given $f \in L^{\infty}(\R^{d})$ and $g \in W^{1, \infty}(\R^{d})$, we bound
    \[\n{\frac{R_{A^{\varepsilon}, \mu^{\varepsilon}}f}{T(\mu^{\varepsilon})} - \frac{w_{A^{\varepsilon}, \mu^{\varepsilon}}}{T(\mu^{\varepsilon})}f}_{L^{1}(\R^{d})} \lesssim_{d, \alpha, \beta} \left( \n{f - g}_{L^{\infty}(\Omega_{\mu^{\varepsilon}})} + \n{\nabla g}_{L^{\infty}(\Omega_{\mu^{\varepsilon}})}T(\mu^{\varepsilon})^{\delta_{d, 1, 1}} + \n{f - g}_{L^{\infty}(\Omega_{\mu^{\varepsilon}})} \right)\]
    by \eqref{eq: a priori estimate}, Theorem \ref{thm: quantitative error estimate}, and the triangle inequality. Since every function $f \in C_{b}(\R^{d})$ can be approximated uniformly on $\overline{D}$ by a sequence of $W^{1, \infty}(\R^{d})$ functions, the last inequality implies the convergence \eqref{eq: qualitative corrector} when $f \in C_{b}(\R^{d})$.

    Since $D$ is bounded, there is a subsequence, still denoted by $(\varepsilon)$, and a nonnegative finite Borel measure $k$ supported in $\overline{D}$ such that $\frac{w_{A^{\varepsilon}, \mu^{\varepsilon}}}{T(\mu^{\varepsilon})} \xrightharpoonup{\mathcal{M}(\R^{d})} k$ weakly. The weak convergence \eqref{eq: convergence to kf} now follows from \eqref{eq: qualitative corrector}. Since $\frac{\n{w_{A^{\varepsilon}, \mu^{\varepsilon}}}_{1}}{T(\mu^{\varepsilon})} \to k(\overline{D})$, \eqref{eq: L1 norm of wA} implies the bounds for $k(\overline{D})$.
\end{proof}

\begin{remark}[Variants of Theorem \ref{thm: main result} and \ref{thm: quantitative error estimate}]
    \label{rmk: var of main thm}
    \begin{enumerate}
        \item Note that \eqref{eq: Ls_L1 estimate of torsion func} can be slightly improved: we have
    \begin{equation}
        \label{eq: torsion lorentz norm estimate}
        \n{w_{A, \mu}}_{{s, r}} \lesssim_{d} s^{\frac{1}{r}}\n{w_{A, \mu}}_{1}^{\gamma_{1} + \frac{1}{s}(1 - \gamma_{1})} \le s^{\frac{1}{r}}T(\mu)^{\gamma_{1} + \frac{1}{s}(1 - \gamma_{1})}
    \end{equation}
    for every $A \in M(1, \beta)$, $\mu \in \mathcal{M}_{\mathrm{cap}}^{b}$, and $1 \le r \le s < \infty$. Here, $\n{\cdot}_{s, r}$ is the $(s, r)$-Lorentz quasi-norm defined for any Lebesgue measurable $v: \R^{d} \to \R$ by
    \begin{equation}
        \n{v}_{s, r} := s^{\frac{1}{r}}\n{\lambda \mathcal{L}^{d}\left(\{|v| > \lambda\} \right)^{\frac{1}{s}}}_{L^{r}(\lambda^{-1} d\lambda\resmes ]0, \infty[)}.
    \end{equation}By using the H\"{o}lder inequality in Lorentz spaces (\cite[Theorem 4.5]{hunt1966on}) and \eqref{eq: torsion lorentz norm estimate} with $s = \frac{2p}{2 - p}$ and $r = 2$ instead of the usual H\"{o}lder and \eqref{eq: Ls_L1 estimate of torsion func} in the proofs of Lemmas \ref{thm: gradient of torsion negative sobolev estimate} and \ref{thm: energy estimate for u-wf} and Theorem \ref{thm: quantitative error estimate}, we obtain a partially improved estimate over \eqref{eq: quantitative error estimate}:
    \begin{equation}
        \label{eq: quantitative error estimate lorentz}
        \n{\frac{R_{A, \mu}f}{T(\mu)} - \frac{w_{A, \mu}}{T(\mu)}f}_{L^{q}(D)}
        \lesssim_{d, p} \beta\n{|\nabla f|\chi_{\Omega_{\mu}}}_{p', \infty}T(\mu)^{\delta_{d, p, q}}
    \end{equation}
    whenever $A \in M(1, \beta)$, $1 \le q \le 2$, $1 \le p < 2$, and $f \in H^{-1}(D)$ satisfy \eqref{eq: admissible integrability of source}, the distributional derivative $\nabla f \in L^{2}(D)$, and $\n{|\nabla f| \chi_{\Omega_{\mu}}}_{p', \infty} < \infty$.

    Using \eqref{eq: quantitative error estimate lorentz}, we may formulate a version of \eqref{eq: quantitative error mue}:
    \begin{equation}
        \label{eq: quantitative error mue lorentz}
        \n{\frac{R_{A^{\varepsilon}, \mu^{\varepsilon}}f}{T(\mu^{\varepsilon})} - \frac{w_{A^{\varepsilon}, \mu^{\varepsilon}}}{T(\mu^{\varepsilon})}f}_{L^{q}(D)} \lesssim_{d, \alpha, \beta, p}\n{|\nabla f|\chi_{\Omega_{\mu^{\varepsilon}}}}_{p', \infty}T(\mu^{\varepsilon})^{\delta_{d, p, q}},
    \end{equation}
    where $A^{\varepsilon}, \mu^{\varepsilon}, p, q, \alpha, \beta$ are as in Theorem \ref{thm: main result}, and $f \in C_{b}(\R^{d})$ satisfies $\nabla f \in L^{1}_{\mathrm{loc}}(D)$ and $\n{|\nabla f|\chi_{D}}_{p', \infty} < \infty$. We shall point out in Remark \ref{rmk: exact sharpness lorentz} that the rate of convergence of the left-hand side of \eqref{eq: quantitative error mue lorentz} is exactly of order $O(T(\mu^{\varepsilon})^{\delta_{d, p, q}})$ for suitable choices of $(\mu^{\varepsilon})$ and $f \in C(\R^{d})$ with $\n{|\nabla f| \chi_{D}}_{p', \infty} < \infty$.

    To obtain \eqref{eq: torsion lorentz norm estimate}, observe first that
    \begin{equation*}
        \begin{split}
            \n{w}_{s, r}^{r} &= s\int_{0}^{\n{w}_{\infty}}\left( \lambda\mathcal{L}^{d}\left(\{w > \lambda \}\right)^{\frac{1}{s}} \right)^{r} \, \frac{d\lambda}{\lambda}
            \le s\n{w}_{\infty}^{1 - \frac{r}{s}}\left( \int_{0}^{\infty} \lambda^{(r - 1)\frac{s}{r}} \mathcal{L}^{d}\left(\{w > \lambda\} \right) \, d\lambda \right)^{\frac{r}{s}}\\
            &= s\left( (r - 1)\frac{s}{r} + 1 \right)^{-\frac{r}{s}}\n{w}_{\infty}^{1 - \frac{r}{s}}\left( \n{w}_{(r - 1)\frac{s}{r} + 1}^{(r - 1)\frac{s}{r} + 1} \right)^{\frac{r}{s}}
        \end{split}
    \end{equation*}
    for every Lebesgue measurable $w: \R^{d} \to \R_{\ge 0}$. Hence, 
    \[\n{w_{A, \mu}}_{s, r} \le s^{\frac{1}{r}}\n{w_{A, \mu}}_{\infty}^{\theta}\n{w_{A, \mu}}_{s(1 - \theta)}^{1 - \theta},\]
    where $\theta = \frac{1}{r} - \frac{1}{s}$. Combining the last estimate and \eqref{eq: Ls_L1 estimate of torsion func} yields \eqref{eq: torsion lorentz norm estimate}.

    \item As we will see in Subsection \ref{sect: pf of quantitative error estimate}, the crucial ingredients of the proofs of Theorem \ref{thm: quantitative error estimate} and the preceding lemmas are the use of the weak formulation of the boundary value problems we consider (namely, \eqref{eq: weak formulation of relaxed Dirichlet}) and Lemma \ref{thm: multi general}. A few modifications which we describe below yield the following analogue of Theorem \ref{thm: main result} for Robin problems: let $\alpha, \beta, A^{\varepsilon}$ be as in Theorem \ref{thm: main result}, $D$ be bounded Lipschitz, and $(U^{\varepsilon})_{\varepsilon} \sub D$ be a sequence of bounded Lipschitz domains such that the solutions $w^{\varepsilon}$ to 
    \begin{equation}
        \label{eq: Robin torsion fct}
        -\Delta w^{\varepsilon} = 1  \text{ in $U^{\varepsilon}$, } \nabla w^{\varepsilon} \cdot n + w^{\varepsilon} = 0  \text{ on $\partial U^{\varepsilon}$}
    \end{equation}
    are such that $\n{w^{\varepsilon}}_{L^{1}(U^{\varepsilon})} \to 0$ as $\varepsilon \searrow 0$. Here, $n$ is the outward unit normal. Then there is a subsequence $(\varepsilon)$ and a nonnegative Borel measure $k$ on $\overline{D}$ with $\beta^{-1} \wedge 1 \le k(\overline{D}) \le \alpha^{-1} \vee 1$ such that for every $f \in C(\overline{D})$, the solution $u^{\varepsilon}$ to $-\mathrm{div}(A^{\varepsilon}\nabla u^{\varepsilon}) = f$ in $U^{\varepsilon}$, $A^{\varepsilon}\nabla u^{\varepsilon} \cdot n + u^{\varepsilon}= 0$ on $\partial U^{\varepsilon}$ satisfies
    \begin{equation}
        \n{w^{\varepsilon}}_{L^{1}(U^{\varepsilon})}^{-1}\bar{u^{\varepsilon}} \xrightharpoonup{\mathcal{M}(\overline{D})} kf
    \end{equation}
    weakly. Here, $\bar{u^{\varepsilon}} \in L^{1}(\R^{d})$ denotes the zero extension of $u^{\varepsilon}$.
    
     The necessary modifications to get the result above are the following.
     \begin{itemize}
        \item Instead of Lemma \ref{thm: multi general}, use the inequality
    \begin{equation}
        \label{eq: multi robin}
        \n{u}_{L^{r}} \le C(d, \mathcal{L}^{d}(U))\left( \n{\nabla u}_{L^{2}(U)}^{2} + \n{u}_{L^{2}(\partial U, \mathcal{H}^{d - 1})}^{2} \right)^{\frac{1}{2}\theta} \n{u}_{L^{s}(U)}^{1 - \theta},\, \frac{1}{r} = \frac{d - 1}{2d}\theta + \frac{1 - \theta}{s}
    \end{equation}
    which is valid for every bounded Lipschitz domain $U$, $u \in H^{1}(U)$, and $1 \le s \le r \le \frac{2d}{d - 1}.$ The estimate \eqref{eq: multi robin} follows from an inequality due to Maz'ya (\cite[Corollary 6.11.1/2]{Mazya2011sobolevspaces}, see also \cite[Section 3]{Daners2000Robin}) and an interpolation between Lebesgue spaces. This enables us to bound the $L^{\infty}(U^{\varepsilon})$ norm of $w^{\varepsilon}$ by a power of $\n{w^{\varepsilon}}_{L^{1}(U^{\varepsilon})}$ times a constant $C(d, \mathcal{L}^{d}(U^{\varepsilon}))$.
    \item Replace every occurrence of the capacitary measure $\mu$ (resp. the space $H^{1}_{\mu}$) in the proofs of Theorem \ref{thm: quantitative error estimate}, Lemma \ref{thm: torsional_Hardy}, \eqref{eq: torsion poincare q < 2}, and lemmas in Subsection \ref{sect: pf of quantitative error estimate} with $\mathcal{H}^{d - 1} \resmes \partial U^{\varepsilon}$ (resp. $H^{1}(U^{\varepsilon})$). Here, $\mathcal{H}^{d - 1}$ denotes the $(d - 1)$-dimensional Hausdorff measure on $\R^{d}$.
     \end{itemize}
     
    However, we will not give a detailed account of this direction as the scaling limit problem for Robin problems seems to be somewhat artificial since no general theory of Robin problems on varying domains comparable to that of \cite{dalmaso1987wiener, dalmaso2004asymptotic} for Dirichlet problems is available.
    \end{enumerate}
\end{remark}

Combining the results in \cite[Section 5]{dalmaso2004asymptotic} with Theorem \ref{thm: main result}, we obtain the following
\begin{corollary}
    \label{thm: main thm for domains}
    Let $0 < \alpha \le \beta < \infty$ and $D \sub \R^{d}$ be bounded open. For $A \in M(\alpha, \beta)$, nonempty open $U \sub D$, and $f \in H^{-1}(D)$, let $R_{A, U}f$ denote the unique (weak) solution $u \in H^{1}_{0}(U)$ to the problem
    \begin{equation}
        -\mathrm{div}(A\nabla u) = f \text{ in $H^{-1}(U)$, $u \in H^{1}_{0}(U)$}.
    \end{equation}
    Let $(A^{\varepsilon}) \sub M(\alpha, \beta)$, and $(D^{\varepsilon}) \sub D$ be arbitrary nonempty open subsets. Here, $\varepsilon$ varies in a sequence in $]0, 1[$ strictly decreasing to $0$ (see Subsection \ref{sect: notation}). Then up to a subsequence, one of the following holds.
    \begin{enumerate}
        \item (\cite[Theorem 5.2, Corollary 5.5]{dalmaso2004asymptotic}). There are $A \in M(\alpha, \beta)$ and a capacitary measure $\mu \in \mathcal{M}_{\mathrm{cap}}(D) \setminus \{I_{\varnothing}\}$ such that for every $f \in H^{-1}(D)$, $R_{A^{\varepsilon}, D^{\varepsilon}}f \xrightharpoonup{H^{1}_{0}(D)} u$ weakly, where $u$ is the solution to the relaxed Dirichlet problem
        \[-\mathrm{div}(A \nabla u) + \mu u = f  \text{ in $(H^{1}_{\mu})^{*}$, $u \in H^{1}_{\mu}$},\]
        the solution to which for every nonnegative $f \in L^{2}(D)$ bounded away from $0$ $\mathcal{L}^{d}$-a.e. on $D$ is nonzero.
        \item $R_{A^{\varepsilon}, D^{\varepsilon}}f \to 0$ in $H^{1}_{0}(D)$ for every $f \in H^{-1}(D)$. Further, there is a nonnegative finite Borel measure $k \in \mathcal{M}(\overline{D})$ with $\beta^{-1} \le k(\overline{D}) \le \alpha^{-1}$ such that for every $f \in C(\overline{D})$, 
        \[\frac{R_{A^{\varepsilon}, D^{\varepsilon}}f}{T(D^{\varepsilon})} \xrightharpoonup{\mathcal{M}(\overline{D})} kf\]
        weakly in $\mathcal{M}(\overline{D})$, where $T(D^{\varepsilon}) = \n{R_{\mathrm{Id}, D^{\varepsilon}}(1)}_{L^{1}(D)}$ is the torsional rigidity. Further, if $p, q \in [1, 2]$ satisfy $\frac{1}{d} - \left( \frac{1}{p'} + \frac{1}{q'} \right)  > 0$, then 
        \begin{equation}
            \label{eq: quantitative error estimate domains}
            \n{\frac{R_{A^{\varepsilon}, D^{\varepsilon}}f}{T(D^{\varepsilon})} - \frac{w_{A^{\varepsilon}, D^{\varepsilon}}}{T(D^{\varepsilon})}f}_{L^{q}(D)} \lesssim_{d, \alpha, \beta} \n{\nabla f}_{L^{p'}(D^{\varepsilon})}T(D^{\varepsilon})^{\delta_{d, p, q}}
        \end{equation}
        for every $f \in C(\overline{D}) \cap W^{1, p'}(D)$, where $\delta_{d, p, q} = \frac{d}{d + 2}\left(\frac{1}{d} - \left( \frac{1}{p'} + \frac{1}{q'} \right)\right)$ and $w_{A^{\varepsilon}, D^{\varepsilon}} = R_{A^{\varepsilon}, D^{\varepsilon}}(1)$.
    \end{enumerate}
\end{corollary}

Motivated by Theorem \ref{thm: main result}, we make the following
\begin{definition}
    \label{def: eff meas}
    Let $0 < \alpha \le \beta < \infty$, $(A^{\varepsilon}) \sub M(\alpha, \beta)$, and $(\mu^{\varepsilon}) \sub \mathcal{M}_{\mathrm{cap}}^{b}$. We say that the problems $P(A^{\varepsilon}, \mu^{\varepsilon})$ admit an \emph{effective measure} $k \in \mathcal{M}(\R^{d})$ if there is a bounded open $D \sub \R^{d}$ such that $\mu^{\varepsilon} \in \mathcal{M}_{\mathrm{cap}}(D) \setminus \{I_{\varnothing}\}$ except for finitely many $\varepsilon$, $T(\mu^{\varepsilon}) \to 0$, and
    \[
        T(\mu^{\varepsilon})^{-1}w_{A^{\varepsilon}, \mu^{\varepsilon}} \rightharpoonup k
        \quad \text{in $\mathcal{M}(\R^{d})$ weak}
    \]
    as $\varepsilon \searrow 0$.
\end{definition}

By Theorem \ref{thm: main result}, every sequence $(A^{\varepsilon}, \mu^{\varepsilon})$ as in that theorem has a subsequence along
which the problems $P(A^{\varepsilon}, \mu^{\varepsilon})$ admit an effective measure. Moreover, if the problems $P(A^{\varepsilon}, \mu^{\varepsilon})$ admit an effective measure $k$, then, applying Theorem \ref{thm: main result} to subsequences, we obtain $T(\mu^{\varepsilon})^{-1}R_{A^{\varepsilon}, \mu^{\varepsilon}}f \rightharpoonup kf$ in $\mathcal{M}(\R^{d})$ weak for every $f \in C_{b}(\R^{d})$, along the whole sequence. For rougher sources, see Section \ref{sect: rougher source}.
 
\begin{remark}
    \label{rmk: Darcy}
    The analogous problem for the Stokes equations in supercritically perforated domains of the type \eqref{eq: perforated domain Deh} leads to Darcy's law (see, e.g., \cite{sanchezpalencia1980nonhomogeneous, tartar1979incompressible, Allaire1989homogenization, Allaire1990homogenization, allaire1997homogenization}). Since the coefficients $(A^{\varepsilon})$ play the role of the viscosity, the effective measure $k$ is an analogue of the ratio of the permeability to the viscosity in Darcy's law. The normalizations also agree: if $D^{\varepsilon}_{\eta}$ is defined as in \eqref{eq: perforated domain Deh}, then $T(D^{\varepsilon}_{\eta}) \asymp_{d, D} \varepsilon^{d - (d - 2)\eta}$ for the supercritical regime $1 \le \eta < \frac{d}{d - 2}$ (see Example \ref{ex: periodically perforated domains}). This is comparable to the factor by which the velocity is rescaled (see \cite[Subsection 3.1.1 and Theorem 3.1]{allaire1997homogenization}).
\end{remark}

The corrector for \eqref{eq: qualitative corrector} and the effective measure also describe the asymptotic behavior of the energy density.
\begin{proposition}[Convergence of the scaled energy density]
    \label{thm: convergence of energy}
    Let $0 < \alpha \le \beta < \infty$, $A \in M(\alpha, \beta)$, $\mu \in \mathcal{M}_{\mathrm{cap}}^{b}$, and $f \in L^{\infty}(\R^{d})$. Then
    \begin{equation}
        \label{eq: quantitative core of energy convergence}
        \sup_{\varphi \in C_{c}^{\infty}(\R^{d}), \n{\nabla \varphi}_{L^{q'}(\Omega_{\mu})} \le 1}T(\mu)^{-1}\left| \int \varphi A \nabla u \cdot \nabla u + \int \varphi u^{2} \, d\mu - \int\varphi uf \right| \lesssim_{d, \alpha, \beta} \n{f}_{\infty}^{2}T(\mu)^{\delta_{d, 1, q}}
    \end{equation}
    for every $1 \le q \le 2$, where $u := R_{A, \mu}f$. If, furthermore, $1 < q < \frac{d}{d - 1}$ and $f \in W^{1, q'}(\R^{d})$, then
    \begin{equation}
        \n{\frac{A \nabla u \cdot \nabla u \mathcal{L}^{d} + u^{2}\mu}{T(\mu)} - \frac{w_{A, \mu}}{T(\mu)}f^{2}}_{W^{-1, q}(\R^{d})} \lesssim_{d, \alpha, \beta, q} \n{f}_{L^{\infty}(\Omega_{\mu})}\left( \n{f}_{L^{\infty}(\Omega_{\mu})} + \n{\nabla f}_{L^{q'}(\Omega_{\mu})} \right) T(\mu)^{\delta_{d, 1, q}}.
    \end{equation}
    As a consequence, if $(A^{\varepsilon}) \sub M(\alpha, \beta)$ and $(\mu^{\varepsilon}) \sub \mathcal{M}_{\mathrm{cap}}^{b}$ are such that the problems $P(A^{\varepsilon}, \mu^{\varepsilon})$ admit an effective measure $k$, then for every $f \in C_{b}(\R^{d})$ and $u^{\varepsilon} := R_{A^{\varepsilon}, \mu^{\varepsilon}}f$,
    \begin{equation}
        \label{eq: weak convergence of energy}
        T(\mu^{\varepsilon})^{-1}\left( A^{\varepsilon}\nabla u^{\varepsilon} \cdot \nabla u^{\varepsilon} \mathcal{L}^{d} + |u^{\varepsilon}|^{2}\mu^{\varepsilon} \right) \to kf^{2}
    \end{equation}
    in $\mathcal{M}(\R^{d})$ weak and in $W^{-1, q}(\R^{d})$ strong for $1 < q < \frac{d}{d - 1}$, and if $f \in W^{1, q'}(\R^{d})$, then 
    \begin{equation}
        \begin{split}
            \Bigg\|\frac{A^{\varepsilon}\nabla u^{\varepsilon} \cdot \nabla u^{\varepsilon} \mathcal{L}^{d} + |u^{\varepsilon}|^{2}\mu^{\varepsilon}}{T(\mu^{\varepsilon})} &- \frac{w_{A^{\varepsilon}, \mu^{\varepsilon}}}{T(\mu^{\varepsilon})}f^{2}\Bigg\|_{W^{-1, q}(\R^{d})}\\ &\lesssim_{d, \alpha, \beta, q} \n{f}_{L^{\infty}(\Omega_{\mu^{\varepsilon}})}\left( \n{f}_{L^{\infty}(\Omega_{\mu^{\varepsilon}})} + \n{ \nabla f}_{L^{q'}(\Omega_{\mu^{\varepsilon}})} \right)T(\mu^{\varepsilon})^{\delta_{d, 1, q}}
        \end{split}
    \end{equation}
    for every $1 < q < \frac{d}{d - 1}$.
\end{proposition}

\begin{proof}
    Given $\varphi \in C_{c}^{\infty}(\R^{d})$, we calculate
    \begin{equation}
        \label{eq: convergence of energy Leibniz}
        \begin{split}
            \int \varphi A \nabla u \cdot \nabla u + \int \varphi u^{2}\, d\mu - \int \varphi u f&= \int A \nabla u \cdot \nabla (\varphi u) + \int \varphi u^{2} \, d\mu -\int \varphi u f - \int u A \nabla u \cdot \nabla \varphi\\
            &= - \int u A \nabla u \cdot \nabla \varphi\\
            &\le \n{u}_{\frac{2q}{2 - q}}\n{A \nabla u}_{2}\n{\nabla \varphi}_{L^{q'}(\Omega_{\mu})}.
        \end{split}
    \end{equation}
    Note that $\n{A \nabla u}_{2}^{2} \le \beta \int A \nabla u \cdot \nabla u + \beta\int u^{2} \, d\mu = \beta\int fu \le \beta \n{f}_{\infty}\n{u}_{1}$. Observe also that Lemma \ref{thm: comparison} implies $\n{u}_{r} \le \n{f}_{\infty}\n{w_{A, \mu}}_{r}$ for every $1 \le r \le \infty$. Hence, the right-hand side of the last display can be further bounded by
    \begin{equation*}
        C(\beta)\n{f}_{\infty}^{2}\n{w_{A, \mu}}_{\frac{2q}{2 - q}} \n{w_{A, \mu}}_{1}^{\frac{1}{2}} \n{ \nabla \varphi}_{L^{q'}(\Omega_{\mu})} \underset{\text{\eqref{eq: Ls_L1 estimate of torsion func}}}{\le} C(d, \alpha, \beta)\n{f}_{\infty}^{2}T(\mu)^{1 + \delta_{d, 1, q}}\n{\nabla \varphi}_{L^{q'}(\Omega_{\mu})},
    \end{equation*}
    and \eqref{eq: quantitative core of energy convergence} is established.
    
    Now suppose that $f \in W^{1, q'}(\R^{d})$ with $1 < q < \frac{d}{d - 1}$. Then by Theorem \ref{thm: quantitative error estimate}, 
    \[\left| \int \varphi u f - \int \varphi w_{A, \mu}f^{2} \right| \lesssim_{d, \alpha, \beta} \n{\varphi f}_{L^{\infty}(\Omega_{\mu})} \n{\nabla f}_{L^{q'}(\Omega_{\mu})}T(\mu)^{1 + \delta_{d, 1, q}}.\]
    From the last display and \eqref{eq: quantitative core of energy convergence}, we obtain 
    \begin{equation}
        \begin{split}
            &\left| \int \varphi A \nabla u \cdot \nabla u + \int \varphi u^{2} \, d\mu - \int \varphi w_{A, \mu}f^{2} \right|\\ &\lesssim_{d, \alpha, \beta} \n{f}_{L^{\infty}(\Omega_{\mu})}\left( \n{\varphi}_{L^{\infty}(\Omega_{\mu})} \n{\nabla f}_{L^{q'}(\Omega_{\mu})} + \n{\nabla \varphi}_{L^{q'}(\Omega_{\mu})} \n{f}_{L^{\infty}(\Omega_{\mu})}\right)T(\mu)^{1 + \delta_{d, 1, q}}\\
            &\lesssim_{d, q} \n{\varphi}_{W^{1, q'}(\R^{d})}\n{f}_{L^{\infty}(\Omega_{\mu})}\left( \n{f}_{L^{\infty}(\Omega_{\mu})} + \n{\nabla f}_{L^{q'}(\Omega_{\mu})} \right) T(\mu)^{1 + \delta_{d, 1, q}}.
        \end{split}
    \end{equation}

    Finally suppose that $f \in C_{b}(\R^{d})$ and the problems $P(A^{\varepsilon}, \mu^{\varepsilon})$ admit an effective measure $k$. Then the left-hand side of \eqref{eq: weak convergence of energy} is concentrated on a compact subset of $\R^{d}$ with total mass bounded by $\n{f}_{\infty}\sup_{\varepsilon}T(\mu^{\varepsilon})^{-1}\n{u^{\varepsilon}}_{1} < \infty$. Hence, up to a subsequence, it is weakly convergent in $\mathcal{M}(\R^{d})$ and, by the Rellich--Kondrachov theorem, is strongly convergent in $W^{-1, q}(\R^{d})$ for $1 < q < \frac{d}{d - 1}$. By \eqref{eq: convergence to kf} and \eqref{eq: quantitative core of energy convergence} with $1 \le q < \frac{d}{d - 1}$, we identify the limit as $kf^{2}$.
\end{proof}

\subsection{Attainability of any finite measure as an effective measure}
\label{sect: attainability}

The next proposition shows that any nonzero nonnegative finite Borel measure on $\R^{d}$ with compact support is the effective measure of suitable problems.

\begin{proposition}[Attainability]
    \label{thm: attainability}
    Let $D \sub \R^{d}$ be a nonempty bounded open set and $k$ a nonnegative Borel measure supported in $\overline{D}$ with $0 < k(\overline{D}) < \infty$. Then there is a sequence $(\Omega^{\varepsilon})_{\varepsilon}$ of open subsets of $D$ such that $k$ is the effective measure of the problems $P(k(\overline{D})^{-1}\mathrm{Id}, \Omega^{\varepsilon})$.
\end{proposition}

\begin{proof}
    We may assume that $k(\overline{D}) = 1$ (see Remark \ref{rmk: scaling property}). Let $\rho$ be a metric for the weak topology on the space $\mathrm{pr}(\overline{D})$ of Borel probability measures on $\overline{D}$ such that
    \begin{equation}
        \rho\left(\sum_{i = 1}^{n}a_{i}\mu_{i}, \sum_{i = 1}^{n}a_{i}\nu_{i}\right) \le \sum_{i = 1}^{n}a_{i}\rho(\mu_{i}, \nu_{i})
    \end{equation}
    for every $a_{i} \ge 0$ and $\mu_{i}, \nu_{i} \in \mathrm{pr}(\overline{D})$ such that $\sum_{i}a_{i} = 1$. For instance, set 
    \[\rho(\mu, \nu) = \sum_{i = 1}^{\infty}2^{-i}\left|\int_{\overline{D}}f_{i} \, (d\nu - d\mu)\right|,\]
    where $(f_{i})_{i}$ is a countable dense family of the unit ball of $C(\overline{D})$.

    Fix $\varepsilon >0$. Then there is a positive integer $N$, $a_{i} \in ]0, 1]$, and distinct $x_{i} \in D$ $(i = 1, \cdots N)$ such that 
    \begin{equation}
        \rho\left( k, \sum_{i = 1}^{N}a_{i}\delta_{x_{i}} \right) < \varepsilon
    \end{equation}
    and $\sum_{i = 1}^{N}a_{i} = 1$. By Example \ref{ex: torsion fct on ball} below, we may find $r_{1}, \cdots r_{N} \in ]0, \infty[$ and $\delta \in ]0,  \varepsilon[$ such that $(U(x_{i}, r_{i}))_{i}$ is a disjoint family of subsets of $D$, 
    \[\rho\left(\delta_{x_{i}}, \frac{w_{U(x_{i}, r_{i})}}{T(U(x_{i}, r_{i}))}\right) \le \varepsilon,\]
    and
    \[T(U(x_{i}, r_{i})) = \delta a_{i}.\]
    Define
    \begin{equation}
        \Omega^{\varepsilon} := \bigcup_{i = 1}^{N}U(x_{i}, r_{i}).
    \end{equation}
    Then $T(\Omega^{\varepsilon}) = \sum_{i = 1}^{N}T(U(x_{i}, r_{i})) = \delta < \varepsilon$, and
    \begin{equation*}
        \begin{split}
            \rho\left( \sum_{i = 1}^{N}a_{i}\delta_{x_{i}},  \frac{w_{\Omega^{\varepsilon}}}{T(\Omega^{\varepsilon})}\right) = \rho\left( \sum_{i = 1}^{N}a_{i}\delta_{x_{i}}, \sum_{i = 1}^{N}a_{i}\frac{w_{U(x_{i}, r_{i})}}{T(U(x_{i}, r_{i}))} \right)\le \sum_{i = 1}^{N}a_{i}\rho\left(\delta_{x_{i}}, \frac{w_{U(x_{i}, r_{i})}}{T(U(x_{i}, r_{i}))}\right) \le \varepsilon.
        \end{split}
    \end{equation*}
\end{proof}

\begin{example}
    \label{ex: torsion fct on ball}
    Let $x \in \R^{d}$. Then $T(U(x, r)) = r^{d + 2}T(U(0, 1))$ for every $r > 0$, and we have 
    \begin{equation}
        \label{eq: effective measure of shrinking balls}
        \frac{w_{U(x, r)}}{T(U(x, r))} \xrightharpoonup{\mathcal{M}(\R^{d})} \delta_{x}
    \end{equation}
    weakly as $r \searrow 0$, where $\delta_{x}$ denotes the Dirac measure at $x$.
\end{example}

\begin{proof}
    Since $r^{2}w_{U(0, 1)}(r^{-1} \cdot) = w_{U(0, r)}$,
    \[T(U(x, r)) = T(U(0, r)) = r^{2}\int w_{U(0, 1)}(r^{-1}y) \, dy = r^{d + 2}\int w_{U(0, 1)} = r^{d + 2}T(U(0, 1)).\]
    Since the left-hand side of \eqref{eq: effective measure of shrinking balls} is supported in $B(x, 1)$ for all $0 < r < 1$, it converges in $\mathcal{M}(\R^{d})$ weak up to a subsequence, say $k$. But the support of $k$ is contained in $\bigcap_{r > 0}B(x, r) = \{x\}$, and $k(\R^{d}) = \int \frac{w_{U(x, r)}}{T(U(x, r))} = 1$. These identify $k = \delta_{x}$.
\end{proof}

\subsection{Examples}
\label{sect: examples}
\begin{example}[Periodically perforated domains]
    \label{ex: periodically perforated domains}
    Let $d \ge 3$ and $D \sub \R^{d}$ be nonempty bounded open. Define $D^{\varepsilon}_{\eta} := D \setminus \bigcup_{z \in \mathbb{Z}^{d}}B(\varepsilon z, 4^{-1}\varepsilon^{\eta})$ for $0 < \varepsilon < 1$ and $\eta \ge 1$. Then $T(D^{\varepsilon}_{\eta}) \asymp_{d, D} \varepsilon^{(d - (d - 2)\eta)_{+}}$ for $\varepsilon$ small enough. Further, the following trichotomy holds.
    \begin{itemize}
        \item If $\eta > \frac{d}{d - 2}$, then $R_{\mathrm{Id}, D^{\varepsilon}_{\eta}}f \xrightarrow{H^{1}_{0}(D)} R_{\mathrm{Id}, D}f$ for every $f \in H^{-1}(D)$.
        \item If $\eta = \frac{d}{d - 2}$, then $R_{\mathrm{Id}, D^{\varepsilon}_{\eta}}f \xrightharpoonup{H^{1}_{0}(D)} R_{\mathrm{Id}, D + c\, \mathcal{L}^{d}}f$ weakly for every $f \in H^{-1}(D)$. Here, $c$ is a positive constant independent of $f$.
        \item If $1 \le \eta < \frac{d}{d - 2}$, then
        \[\frac{R_{\mathrm{Id}, D^{\varepsilon}_{\eta}}f}{T(D^{\varepsilon}_{\eta})} \xrightharpoonup{L^{2}(D)} \frac{1}{\mathcal{L}^{d}(D)}f\]
        for every $f \in L^{2}(D)$. If, furthermore, $\eta > 1$, then the last convergence takes place in $L^{2}(D)$ strong.
    \end{itemize}
\end{example}
\begin{proof}
    For the cases $\eta \ge \frac{d}{d - 2}$, see \cite{cioranescu1997strange}. The case $1 < \eta < \frac{d}{d - 2}$ follows from \cite[Th\'{e}or\`{e}me 5.1 of Chapitre III]{kacimi1987homogenisation}  (see also \cite[Theorem 2.2]{Jing2020unified}). The case $\eta = 1$ follows from \cite[Th\'{e}or\`{e}me 8.1 of Chapitre III]{kacimi1987homogenisation} (see also \cite[Theorem 3.1]{lions1980asymptotic}, \cite[Section 11 of Chapter 1]{lions1981some}, \cite[Theorem 2.2]{Jing2020unified}).
\end{proof}

\begin{example}[Dimension reduction]
    \label{ex: dimension reduction}
    Let $D \sub \R^{d}$, $d \ge 2$ be the open unit ball and $D^{\varepsilon}$ be the annulus $\{1 - \varepsilon < |x| < 1\}$. Then 
    \begin{equation}
        \frac{w_{D^{\varepsilon}}}{T(D^{\varepsilon})} \xrightharpoonup{\mathcal{M}(\R^{d})} \frac{1}{\mathcal{H}^{d - 1}(\partial D)}\mathcal{H}^{d - 1}\resmes \partial D
    \end{equation}
    weakly. Here, $\mathcal{H}^{d - 1}$ denotes the $(d - 1)$-dimensional Hausdorff measure. Indeed, the limit $k$ of the left-hand side of the last display must be a rotation-invariant probability measure supported on $\partial D$.
\end{example}

\begin{example}[Disjoint balls]
    \label{ex: disjoint balls}
    Let $d \ge 2$ and $D = ]0, 1[^{d}$. For each positive integer $h$, put $\varepsilon_{h} := 2^{-(h + 1)}$ and let $I^{h}_{1}, \cdots, I^{h}_{2^{hd}}$ be translates of $2\varepsilon_{h}[0, 1[^{d}$ such that $[0, 1[^{d}$ is a disjoint union of $I^{h}_{1}, \cdots, I^{h}_{2^{hd}}$. For each $h$ and a real $\varkappa \ge 1$, let $D_{h, \varkappa}$ be the union of open balls of radii $\varepsilon_{h}^{\varkappa}$ centered at the centers $x_{1}^{h}, \cdots, x_{2^{hd}}^{h}$ of $I^{h}_{1}, \cdots, I^{h}_{2^{hd}}$, respectively.
    \begin{enumerate}
        \item 
        \begin{equation}
            \label{eq: Ls norm of torsion of disjoint balls}
            \n{w_{D_{h, \varkappa}}}_{s} = \left( \frac{1}{2} \right)^{\frac{d}{s}}\varepsilon_{h}^{(\varkappa - 1)\frac{d}{s} + 2\varkappa}\n{w_{U(0, 1)}}_{s}
        \end{equation}
        for every $1 \le s \le \infty$.
        \item As $h \to \infty$, 
        \begin{equation}
            \frac{w_{D_{h, \varkappa}}}{T(D_{h, \varkappa})} \xrightharpoonup{\mathcal{M}(\R^{d})} \mathcal{L}^{d} \resmes D.
        \end{equation}
        Let $0 < \alpha \le b \le \beta < \infty$ and $B \in M(\alpha, \beta)$ such that $B = B(r\cdot)$ $\mathcal{L}^{d}$-a.e. for all $r > 0$. Define $A_{h} = A_{B, b, \varkappa, h} \in M(\alpha, \beta)$ by $A_{h}(x) = B(x - x_{i}^{h})$ on the balls $U(x_{i}^{h}, \varepsilon_{h}^{\varkappa})$, $1 \le i \le 2^{hd}$, and $A_{h} = b \mathrm{Id}$ elsewhere. Then 
        \begin{equation}
            \frac{w_{A_{h}, D_{h, \varkappa}}}{T(D_{h, \varkappa})} \xrightharpoonup{\mathcal{M}(\R^{d})}k_{B}\mathcal{L}^{d}\resmes D.
        \end{equation}
        Here, $k_{B} := T(U(0, 1))^{-1}\n{w_{B, U(0, 1)}}_{1}$.
        \item If $\varkappa > 1$, then $A_{h} \to b \mathrm{Id}$ in $L^{p}(D)$ for every $1 \le p < \infty$, while the effective measure $k_{B}\mathcal{L}^{d}\resmes D$ of the problems $P(A_{B, b, \varkappa, h}, D_{h, \varkappa})$ does not depend on $b$.
    \end{enumerate}
\end{example}

\begin{proof}
    Item 1 follows from \eqref{eq: scaling property of resolvent}. To prove 2, extract a subsequence of $(h)$, which we still denote by $(h)$, and a Radon measure $k$ supported in $\overline{D}$ such that
    \[\frac{w_{A_{h}, D_{h, \varkappa}}}{T(D_{h, \varkappa})} \xrightharpoonup{\mathcal{M}(\R^{d})} k\]
    weakly. Then for every pair of indices $h \ge h_{0}$ and $1 \le i \le 2^{h_{0}d}$, we have
    \begin{equation}
        \int_{\overline{I^{h_{0}}_{i}}} \frac{w_{A_{h}, D_{h, \varkappa}}}{T(D_{h, \varkappa})} = k_{B}\mathcal{L}^{d}(I^{h_{0}}_{i}) = \int_{\mathrm{int}I^{h_{0}}_{i}} \frac{w_{A_{h}, D_{h, \varkappa}}}{T(D_{h, \varkappa})},
    \end{equation}
    where $\mathrm{int}I^{h_{0}}_{i}$ denotes the interior of $I^{h_{0}}_{i}$. Hence $k(\overline{I^{h_{0}}_{i}}) \ge k_{B}\mathcal{L}^{d}(I^{h_{0}}_{{i}}) \ge k(\mathrm{int}(I^{h_{0}}_{i}))$ for every index $h_{0}$ and $1 \le i \le 2^{h_{0}d}$. This in particular implies that $k\resmes D$ is absolutely continuous with respect to $\mathcal{L}^{d}\resmes D$, and thus $k \resmes D = k_{B}\mathcal{L}^{d}\resmes D$. Since $k(\overline{D}) = k_{B}$, this proves 2.

    Now assume that $\varkappa > 1$. Then $\mathcal{L}^{d}(D_{h, \varkappa}) \asymp_{d}\varepsilon_{h}^{(\varkappa - 1)d} \to 0$ as $h \to \infty$, and thus $A_{h} \to b \mathrm{Id}$ in measure, which in turn implies the $L^{p}(D)$ convergence for finite $p$, since $|A_{h}| \lesssim_{d} \beta$ $\mathcal{L}^{d}$-a.e. for every $h$.
\end{proof}

\subsection{Properties of effective measures: locality and the contribution of the antisymmetric parts of the coefficients}
\label{sect: properties of effective measures}
Let $D \sub \R^{d}$ be nonempty bounded open, $0 < \alpha \le \beta < \infty$, and $(A^{\varepsilon}) \sub M(\alpha, \beta)$. In this and the next subsection, we fix a sequence $(\mu^{\varepsilon}) \sub \mathcal{M}_{\mathrm{cap}}(D)$ such that the problems $P(A^{\varepsilon}, \mu^{\varepsilon})$ admit an effective measure $k$.

The first result computes the effective measure of the adjoint problems $P(A^{\varepsilon *}, \mu^{\varepsilon})$.
\begin{proposition}
    \label{thm: adjoint}
    Let $A \in M(\alpha, \beta)$, $\mu \in \mathcal{M}_{\mathrm{cap}}^{b}$, and $1 \le q \le 2$. Then
    \begin{equation}
        \label{eq: quantitative diff of adjoints}
        \left| \int(w_{A, \mu} - w_{A^{*}, \mu})\varphi \right| \lesssim_{d, \alpha, \beta}\n{\nabla \varphi}_{L^{q'}(\Omega_{\mu})} T(\mu)^{1 + \delta_{d, 1, q}}
    \end{equation}
    for every $\varphi \in C_{c}^{\infty}(\R^{d})$. Here, $\delta_{d, 1, q}$ is defined by \eqref{def: delta dpq}. In particular, the measure $k$ is also the effective measure of the adjoint problems $P(A^{\varepsilon *}, \mu^{\varepsilon})$.
\end{proposition}

\begin{proof}
    For every $\varphi \in C_{c}^{\infty}(\R^{d})$, 
    \begin{equation}
        \label{eq: adjoint Leibniz}
        \begin{split}
            &\int w_{A, \mu}\varphi\\
            &= \int A^{*}\nabla w_{A^{*}, \mu} \cdot \nabla (w_{A, \mu}\varphi) + \int w_{A^{*}, \mu}w_{A, \mu}\varphi \, d\mu\\
            &= \int \nabla(\varphi w_{A^{*}, \mu}) \cdot A\nabla w_{A, \mu} + \int w_{A^{*}, \mu} w_{A, \mu}\varphi \, d\mu - \int w_{A^{*}, \mu}\nabla \varphi \cdot A\nabla w_{A, \mu} + \int w_{A, \mu}\nabla w_{A^{*}, \mu} \cdot A\nabla \varphi\\
            &= \int w_{A^{*}, \mu}\varphi - \int w_{A^{*}, \mu}\nabla \varphi \cdot A\nabla w_{A, \mu} + \int w_{A, \mu}\nabla w_{A^{*}, \mu} \cdot A\nabla \varphi.
        \end{split}
    \end{equation}
    
    By \eqref{eq: adjoint Leibniz}, \eqref{eq: energy bound by torsion}, and the H\"{o}lder inequality, we have
    \begin{equation}
        \int (w_{A, \mu} - w_{A^{*}, \mu})\varphi \lesssim_{d, \alpha, \beta} \left( \n{w_{A^{*}, \mu}}_{\frac{2q}{2 - q}}\n{w_{A, \mu}}_{1}^{\frac{1}{2}} + \n{w_{A, \mu}}_{\frac{2q}{2 - q}} \n{w_{A^{*}, \mu}}_{1}^{\frac{1}{2}} \right)\n{\nabla \varphi}_{L^{q'}(\Omega_{\mu})}.
    \end{equation}
    The estimate \eqref{eq: quantitative diff of adjoints} now follows from the last display combined with \eqref{eq: Ls_L1 estimate of torsion func}.
\end{proof}

Since the effective measure $k$ is defined through boundary value problems $P(A^{\varepsilon}, \mu^{\varepsilon})$, the local behavior of $k$ a priori depends on $A^{\varepsilon}$ and $\mu^{\varepsilon}$ everywhere. However, the next proposition and its corollary show that it is, to some extent, determined locally.

\begin{proposition}[Locality of the effective measure]
    \label{thm: locality}
    Let $A \in M(\alpha, \beta)$, $\mu \in \mathcal{M}_{\mathrm{cap}}^{b}$, $1 \le q \le 2$, and $U \sub \R^{d}$ open. Then 
    \begin{equation}
        \label{eq: quantitative diff loc}
        \left| \int(w_{A, \mu} - w_{A, \mu + I_{U}})\varphi \right| \lesssim_{d, \alpha, \beta}\n{\nabla \varphi}_{L^{q'}(\Omega_{\mu} \cap U)}T(\mu)^{1 + \delta_{d, 1, q}}
    \end{equation}
    for every $\varphi \in C_{c}^{\infty}(U)$. In particular, 
    \begin{equation}
        \label{eq: local formula for k}
        \frac{w_{A^{\varepsilon}, \mu^{\varepsilon} + I_{U}}}{T(\mu^{\varepsilon})} \xrightharpoonup[*]{\mathcal{M}(U)} k \resmes U
    \end{equation}
    weakly* in $\mathcal{M}(U)$. If $\mu^{\varepsilon}$ is of the form $\mu^{\varepsilon} = I_{D^{\varepsilon}}$ with $D^{\varepsilon} \sub D$ open, then \eqref{eq: local formula for k} reads
    \begin{equation}
        \frac{w_{A^{\varepsilon}, D^{\varepsilon} \cap U}}{T(D^{\varepsilon})} \xrightharpoonup[*]{\mathcal{M}(U)} k \resmes U.
    \end{equation}
\end{proposition}

\begin{proof}
    The proof is similar to that of Proposition \ref{thm: adjoint}. For every $\varphi \in C^{\infty}_{c}(U)$, since $\varphi w_{A, \mu} \in H^{1}_{\mu + I_{U}}$, 
    \begin{equation}
        \label{eq: locality test function leibniz}
        \begin{split}
            \int w_{A, \mu}\varphi &= \int A^{*}\nabla w_{A^{*}, \mu + I_{U}} \cdot \nabla (w_{A, \mu}\varphi) + \int w_{A^{*}, \mu + I_{U}}w_{A, \mu}\varphi \, d(\mu + I_{U})\\
            &=\int \nabla w_{A^{*}, \mu + I_{U}} \cdot A\nabla (w_{A, \mu}\varphi) + \int w_{A^{*}, \mu + I_{U}}w_{A, \mu}\varphi \, d\mu\\
            &= \int w_{A^{*}, \mu + I_{U}}\varphi - \int w_{A^{*}, \mu + I_{U}}\nabla \varphi \cdot A\nabla w_{A, \mu} + \int w_{A, \mu}\nabla w_{A^{*}, \mu +I_{U}} \cdot A\nabla \varphi.
        \end{split}
    \end{equation}
    Note that $0 \le w_{A^{*}, \mu + I_{U}} \le w_{A^{*}, \mu}$ $\mathcal{L}^{d}$-a.e. by Lemma \ref{thm: comparison}, and thus $\n{w_{A^{*}, \mu + I_{U}}}_{1} \lesssim_{\alpha} T(\mu)$ by \eqref{eq: L1 norm of wA}.  Hence, by applying the H\"{o}lder inequality to the right-hand side of \eqref{eq: locality test function leibniz} and then using \eqref{eq: energy bound by torsion}, \eqref{eq: L1 norm of wA}, and \eqref{eq: Ls_L1 estimate of torsion func}, we obtain 
    \[\left|\int(w_{A, \mu} - w_{A^{*}, \mu + I_{U}})\varphi \right| \lesssim_{d, \alpha, \beta} \n{\nabla \varphi}_{L^{q'}(\Omega_{\mu} \cap U)}T(\mu)^{1 + \delta_{d, 1, q}}\]
    for $1 \le q \le 2$. The last display combined with \eqref{eq: quantitative diff of adjoints} for $\mu + I_{U}$ shows \eqref{eq: quantitative diff loc}. To see \eqref{eq: local formula for k}, it suffices to observe that $\n{\frac{w_{A^{\varepsilon}, \mu^{\varepsilon} + I_{U}}}{T(\mu^{\varepsilon})}}_{\mathcal{M}(U)} = T(\mu^{\varepsilon})^{-1}\n{w_{A^{\varepsilon}, \mu^{\varepsilon} + I_{U}}}_{1} \le \alpha^{-1} \vee 1$, and hence that the left-hand side of \eqref{eq: local formula for k} is weak* convergent in $\mathcal{M}(U)$ up to a subsequence.
\end{proof}

We repeatedly use the next corollary to compare effective measures.
\begin{corollary}
    \label{thm: locality uniform approx}
    Let $(B^{\varepsilon}) \sub M(\alpha, \beta)$ be such that the problems $P(B^{\varepsilon}, \mu^{\varepsilon})$ admit an effective measure $k'$. If $U \sub \R^{d}$ is open and $\n{A^{\varepsilon} - B^{\varepsilon}}_{L^{\infty}(\Omega_{\mu^{\varepsilon}} \cap U)} \to 0$, then $k'\resmes U = k\resmes U$.
\end{corollary}

\begin{proof}
    Observe first that the bound
    \begin{equation}
        \label{eq: quantify the difference of diffusion coefficients}
        \n{w_{A, \mu} - w_{B, \mu}}_{1} \lesssim_{d, \alpha, \beta} T(\mu)\n{A - B}_{L^{\infty}(\Omega_{\mu})}
    \end{equation}
    is valid for every $A, B \in M(\alpha, \beta)$ and $\mu \in \mathcal{M}_{\mathrm{cap}}^{b}$. Indeed, if we put $z := w_{A, \mu} - w_{B, \mu}$, we have
    \[\begin{split}
        &\n{ z}_{\dot{H}^{1}_{\mu}}^{2}
        \lesssim_{\alpha} \int A \nabla z \cdot \nabla z + \int |z|^{2} \, d\mu\\
        &= \int A \nabla w_{A, \mu} \cdot \nabla z + \int w_{A, \mu} z \, d\mu - \left( \int B\nabla w_{B, \mu} \cdot \nabla z + \int w_{B, \mu}z \, d\mu \right)  - \int (A - B)\nabla w_{B, \mu} \cdot \nabla z\\
        &= -\int (A - B)\nabla w_{B, \mu} \cdot \nabla z.
    \end{split}\]
    The last display combined with \eqref{eq: torsion poincare q < 2} yields \eqref{eq: quantify the difference of diffusion coefficients}. From \eqref{eq: quantify the difference of diffusion coefficients} and \eqref{eq: quantitative diff loc}, we obtain 
    \begin{equation}
        \label{eq: locality A - B quantitative}
        \left| \int (w_{A, \mu} - w_{B, \mu})\varphi \right| \lesssim_{d, \alpha, \beta} T(\mu)\left( \n{\nabla \varphi}_{L^{q'}(\Omega_{\mu} \cap U)}T(\mu)^{\delta_{d, 1, q}} + \n{\varphi}_{L^{\infty}(\Omega_{\mu + I_{U}})} \n{A - B}_{L^{\infty}(\Omega_{\mu + I_{U}})} \right)
    \end{equation}
    for every $A, B \in M(\alpha, \beta)$, $\mu \in \mathcal{M}_{\mathrm{cap}}^{b}$, open $U \sub \R^{d}$, $\varphi \in C_{c}^{\infty}(U)$, and $1 \le q \le 2$. Since $\Omega_{\mu + I_{U}} \sub \Omega_{\mu} \cap U$ up to a set of capacity zero, we deduce from \eqref{eq: locality A - B quantitative} that $k' \resmes U = k \resmes U$ whenever $\n{A^{\varepsilon} - B^{\varepsilon}}_{L^{\infty}(\Omega_{\mu^{\varepsilon}} \cap U)} \to 0$.
\end{proof}

\begin{remark}
    Example \ref{ex: disjoint balls} shows that the $L^{\infty}$-convergence assumption in Corollary \ref{thm: locality uniform approx} cannot be replaced with the condition $\n{A^{\varepsilon} - B^{\varepsilon}}_{L^{p}(D \cap U)} \to 0$ for a finite $p$.
\end{remark}

We now present an application of Corollary \ref{thm: locality uniform approx}. Given a nonnegative Borel measure $m$ on $\R^{d}$, we put 
\begin{equation}
    \label{def: Cm}
    \begin{split}
        \mathcal{C}_{m} := \{u \in L^{\infty}(m)&: \text{There is an $m$-null set $N$ such that the restriction $u \restriction \R^{d} \setminus N$ is continuous}\}.
    \end{split}
\end{equation}
Note that if $\mathcal{C}_{D, k}$ is defined as in Theorem \ref{thm: antisymmetric intro}, then the zero extension $\mathcal{C}_{D, k} \to \mathcal{C}_{\mathcal{L}^{d}\resmes D + k}$ is a bijection.
\begin{theorem}
    \label{thm: continuous antisymmetric part in the uniform limit does not contribute to k}
    Suppose that the antisymmetric part of $A^{\varepsilon}$ converges in $L^{\infty}(D)$ to a matrix-valued function $C \in \mathcal{C}_{\mathcal{L}^{d} \resmes D + k}^{d \times d}$. Then $k$ is also the effective measure of the problems $P(\frac{A^{\varepsilon} + A^{\varepsilon *}}{2}, \mu^{\varepsilon})$.
\end{theorem}

\begin{proof}
    Put $m := \mathcal{L}^{d} \resmes D + k$. For a matrix $\mathcal{A}$, we denote by $\mathcal{A}^{s}$ and $\mathcal{A}^{a}$ the symmetric and antisymmetric parts of $\mathcal{A}$, respectively. By Corollary \ref{thm: locality uniform approx}, we may assume that $A^{\varepsilon a} = C$ $\mathcal{L}^{d}$-a.e. on $D$ for every $\varepsilon$. Passing to a subsequence, we may assume that the problems $P(A^{\varepsilon s}, \mu^{\varepsilon})$ admit an effective measure $k'$. It suffices to show that $k' = k$. If $C$ is constant $\mathcal{L}^{d}$-a.e. on $D$, then we see from Definition \ref{def: relaxed_Dirichlet} of relaxed Dirichlet problem that $w_{A^{\varepsilon}, \mu^{\varepsilon}} = w_{A^{\varepsilon s}, \mu^{\varepsilon}}$, and hence $k' = k$. Next, suppose that $C \in \mathcal{S}^{d \times d}$, where $\mathcal{S} := \mathcal{S}_{m}$ is the class defined in Lemma \ref{thm: approx by simple fcts}. Then we may write 
    \begin{equation}
        C = \sum_{i = 1}^{n}C_{i}\chi_{U_{i}},
    \end{equation}
    where $n \ge 1$, the $C_{i}$ are constant antisymmetric matrices, $U_{i}$ are open sets in $\R^{d}$, and $\bigcup_{i \neq j}(U_{i} \cap U_{j})$ and $\R^{d} \setminus \bigcup_{i}U_{i}$ are of $m$ measure zero. Then by Corollary \ref{thm: locality uniform approx}, $k' \resmes U_{i} = k\resmes U_{i}$ for all $i$. We shall see from Theorem \ref{thm: effective measures comparison} below, the proof of which is independent of the present theorem, that $k$ and $k'$ are mutually absolutely continuous. Hence $k' = k$. Finally, consider the general case where $C \in \mathcal{C}_{m}$. Then by a diagonal argument we may construct a sequence $(C^{\varepsilon}) \sub \mathcal{S}^{d \times d}$ such that $\n{C^{\varepsilon} - C}_{L^{\infty}(D)} \le \n{C^{\varepsilon} - C}_{L^{\infty}(m)} \to 0$ and $k'$ is the effective measure of the problems $P(A^{\varepsilon s} + C^{\varepsilon}, \mu^{\varepsilon})$. Hence, again by Corollary \ref{thm: locality uniform approx}, we have that $k' = k$.
\end{proof}

\begin{example}
    In the same setting as Example \ref{ex: disjoint balls}, let $B$ be of the form
    \begin{equation}
        B(x) = B(x_{1}, \cdots, x_{d}) = \begin{cases}
            \mathrm{Id} & \text{ if $x_{d} > 0$}\\
            \mathrm{Id} + C & \text{ if $x_{d} \le 0$},
        \end{cases}
    \end{equation}
    where $C = (c_{ij})$ is a $d \times d$ antisymmetric constant matrix such that $(c_{1 d}, \cdots, c_{d d}) \neq 0$. Then a direct computation shows that $w_{U(0, 1)} = \frac{1 - |x|^{2}}{2d} \neq w_{B, U(0, 1)}$. Hence, 
    \[\begin{split}
        -T(U(0, 1)) &= \int |\nabla w_{U(0, 1)}|^{2} - 2 \int w_{U(0, 1)} < \int |\nabla w_{B, U(0, 1)}|^{2} -2\int w_{B, U(0, 1)}\\
        &= \int B\nabla w_{B, U(0, 1)} \cdot \nabla w_{B, U(0, 1)} -2\int w_{B, U(0, 1)} = - \n{w_{B, U(0, 1)}}_{1},
    \end{split}\]
    and $k_{B} = T(U(0, 1))^{-1}\n{w_{B, U(0, 1)}}_{1} < 1$. Hence, the problems $P(A_{h}, D_{h, \varkappa})$ and $P(\frac{A_{h} + A_{h}^{*}}{2}, D_{h, \varkappa})$ have different effective measures.
\end{example}
\subsection{Comparisons of effective measures}
\label{sect: comparison}
In this subsection, we compare the effective measures of the problems $P(A^{\varepsilon}, \mu^{\varepsilon})$ and $P(B^{\varepsilon}, \mu^{\varepsilon})$, which have the same $(\mu^{\varepsilon})$ but different coefficients.
\begin{theorem}[Bounds on effective measures]
    \label{thm: effective measures comparison}
    Let $0 < \alpha' \le \beta' < \infty$ and $(B^{\varepsilon}) \sub M(\alpha', \beta')$ be such that the problems $P(B^{\varepsilon}, \mu^{\varepsilon})$ admit an effective measure $k'$. Then 
    \begin{equation}
        \label{eq: effective measures comparison general}
    \left(  \frac{\beta'}{\alpha} \vee 1 \right)^{-1}k \le    k' \le \left( \frac{\beta}{\alpha'} \vee 1 \right)k.
    \end{equation}
    Now, suppose that $\mu^{\varepsilon}$ are quasi-open. Then
    \begin{equation}
        \label{eq: effective measures comparison domains}
        \frac{\alpha}{\beta'}k \le k' \le \frac{\beta}{\alpha'}k.
    \end{equation}
    Furthermore, for every Borel measure $\nu$ on $\R^{d}$ with
    \begin{equation}
        \label{eq: nu bound}
        \frac{\alpha}{\beta'}k \le \nu \le \frac{\beta}{\alpha'}k,
    \end{equation}
    there is a sequence $(f_{\varepsilon}) \sub C^{\infty}(\R^{d})$ with $\frac{\alpha}{\beta'} \le f_{\varepsilon} \le \frac{\beta}{\alpha'}$ such that $f_{\varepsilon}k \to \nu$ in total variation norm and $\nu$ is the effective measure of the problems $P(\tilde{B}^{\varepsilon}, \mu^{\varepsilon})$, where $\tilde{B}^{\varepsilon} := f_{\varepsilon}^{-1}A^{\varepsilon} \in M(\frac{\alpha}{\beta}\alpha', \frac{\beta}{\alpha}\beta')$. If the Radon--Nikodym derivative of $\nu$ with respect to $k$ admits a version $\rho \in \mathcal{C}_{\mathcal{L}^{d} \resmes D + k}$ (see \eqref{def: Cm}), then $\nu$ also arises as the effective measure of the problems $P(\hat{\rho}^{-1}A^{\varepsilon}, \mu^{\varepsilon})$. Here, $\hat{\rho} := (\rho \vee \frac{\alpha}{\beta'}) \wedge \frac{\beta}{\alpha'}$.
\end{theorem}

\begin{proof}
    We first show that for every $A \in M(\alpha, \beta)$, $B \in M(\alpha', \beta')$, and $\mu \in \mathcal{M}_{\mathrm{cap}}^{b}$, we have
    \begin{equation}
        \label{eq: comparison prop quantitative core}
        \sup_{\varphi \in C_{c}^{\infty}(\R^{d}), \n{\nabla \varphi}_{L^{q'}(\Omega_{\mu})} \le 1, \varphi \ge 0}\int \left( w_{A, \mu} - \frac{\beta'}{\alpha}w_{B, \mu} \right)\varphi \lesssim_{d, \alpha, \alpha', \beta'}T(\mu)^{1 + \delta_{d, 1, q}},
    \end{equation}
where $1 \le q \le 2$ and $\delta_{d, 1, q}$ is defined by \eqref{def: delta dpq}, if either $\alpha \le \beta'$ or $\mu$ is quasi-open. To this end, we may assume that $\alpha = 1$ and that either  $1 \le \beta'$ or $\mu$ is quasi-open. For every nonnegative $\varphi \in C_{c}^{\infty}(\R^{d})$ with $\n{\nabla \varphi}_{L^{q'}(\Omega_{\mu})} \le 1$, 
\begin{equation}
    \label{eq: bound by I + II + III}
    \begin{split}
        \int w_{A, \mu}\varphi &= \int B\nabla w_{B, \mu} \cdot \nabla (w_{A, \mu}\varphi) + \int w_{B, \mu}w_{A, \mu}\varphi \, d\mu\\
        &= \int w_{A, \mu}B\nabla w_{B, \mu}\cdot \nabla \varphi + \int \varphi B\nabla w_{B, \mu}\cdot \nabla w_{A, \mu} + \int w_{B, \mu}w_{A, \mu} \varphi \, d\mu\\
        &\underset{\text{Young's inequality}}{\le} \int_{D}w_{A, \mu}B\nabla w_{B, \mu}\cdot \nabla \varphi\\
        &+ \frac{1}{2\delta}\left( \int \varphi|B\nabla w_{B, \mu}|^{2} + \int w_{B, \mu}^{2}\varphi \, d\mu\right) + \frac{\delta}{2}\left( \int \varphi|\nabla w_{A, \mu}|^{2} + \int w_{A, \mu}^{2} \varphi \, d\mu \right)\\
        &=: I + II + III.
    \end{split}
\end{equation}
Using the H\"{o}lder inequality and then \eqref{eq: Ls_L1 estimate of torsion func}, \eqref{eq: energy bound by torsion}, and \eqref{eq: L1 norm of wA}, we bound, $I \lesssim_{d, \alpha', \beta'} T(\mu)^{1 + \delta_{d, 1, q}}$. We also bound, by using Proposition \ref{thm: convergence of energy},
\begin{equation}
    \label{eq: comparison II}
    \begin{split}
        II &\le \frac{\beta'}{2\delta}\left( \int \varphi B\nabla w_{B, \mu} \cdot \nabla w_{B, \mu} + \int w_{B, \mu}^{2}\varphi \, d\mu \right) \underset{\text{\eqref{eq: quantitative core of energy convergence}}}{\lesssim_{d, \alpha'}} \frac{\beta'}{2\delta}\left( T(\mu)^{1 + \delta_{d, 1, q}} + \int \varphi w_{B, \mu} \right),
    \end{split}
\end{equation}
where the first inequality follows from the assumption that either $\beta' \ge 1$ or $\mu$ is quasi-open.
Similarly, $III \lesssim_{d} \frac{\delta}{2}\left( T(\mu)^{1 + \delta_{d, 1, q}} + \int \varphi w_{A, \mu} \right)$.
Combining \eqref{eq: bound by I + II + III} and the estimates for $I$, $II$, and $III$ with $\delta = 1$, we obtain \eqref{eq: comparison prop quantitative core}.

From \eqref{eq: comparison prop quantitative core}, we deduce that 
\[\int \varphi \, dk - \frac{\beta'}{\alpha} \int \varphi \, dk' \le 0\]
for every nonnegative $\varphi \in C_{c}^{\infty}(\R^{d})$, and thus $k \le \alpha^{-1}\beta'k'$, if either $\alpha \le \beta'$ or $\mu^{\varepsilon}$ are quasi-open. This proves the first inequalities of \eqref{eq: effective measures comparison general} and \eqref{eq: effective measures comparison domains}. The second inequalities follow similarly.

Now, suppose that $\mu^{\varepsilon}$ are quasi-open and let a Borel measure $\nu$ on $\R^{d}$ satisfy the bound \eqref{eq: nu bound}. We first assume that the Radon--Nikodym derivative of $\nu$ with respect to $k$ admits a version $\rho \in \mathcal{C}_{m}$, and show that $\nu$ is the effective measure of the problems $P(\hat{\rho}^{-1}A^{\varepsilon}, \mu^{\varepsilon})$. Here, $m := \mathcal{L}^{d} \resmes D + k$. Since $\hat{\rho} = \rho$ $k$-a.e. and $\hat{\rho} \in \mathcal{C}_{m}$, we may assume that $\frac{\alpha}{\beta'} \le \rho \le \frac{\beta}{\alpha'}$ pointwise. Passing to a subsequence, we may assume that the problems $P(\rho^{-1}A^{\varepsilon}, \mu^{\varepsilon})$ admit an effective measure $k''$; it suffices to show $k'' = \nu$. If $\rho$ is constant, then $w_{\rho^{-1}A^{\varepsilon}, \mu^{\varepsilon}} = \rho w_{A^{\varepsilon}, \mu^{\varepsilon}}$, and hence $k'' = \rho k = \nu$. Next, suppose that $\rho$ is of the form $\rho = \sum_{i = 1}^{n}r_{i}\chi_{U_{i}},$ where $n \ge 1$, $0 < r_{i} < \infty$, and $U_{i}$ are open sets such that $\bigcup_{i \neq j}(U_{i} \cap U_{j})$ and $\R^{d} \setminus \bigcup_{i}U_{i}$ are $m$-null. Then by Corollary \ref{thm: locality uniform approx}, $k''\resmes U_{i} =r_{i}k\resmes U_{i} = \nu\resmes U_{i}$ for all $i$. Since $k''$ and $\nu$ satisfy the bounds \eqref{eq: effective measures comparison domains} and \eqref{eq: nu bound}, respectively, they are absolutely continuous with respect to $m$. Hence, $k'' = \nu$. For the general case where $\rho \in \mathcal{C}_{m}$, Lemma \ref{thm: approx by simple fcts} allows us to construct by a diagonal argument a sequence $(\rho_{\varepsilon}) \sub \mathcal{S}_{m}$ such that $\frac{\alpha}{\beta'} \le \rho_{\varepsilon} \le \frac{\beta}{\alpha'}$ $m$-a.e., $\n{\rho_{\varepsilon} - \rho}_{L^{\infty}(m)} \to 0$, and $T(\mu^{\varepsilon})^{-1}w_{\rho_{\varepsilon}^{-1}A^{\varepsilon}, \mu^{\varepsilon}} - \rho_{\varepsilon}k \rightharpoonup 0$ in $\mathcal{M}(\R^{d})$ weak. Hence, by Corollary \ref{thm: locality uniform approx}, $k'' = \nu$.

Finally, consider the case where $\nu$ satisfies the bound \eqref{eq: nu bound} only. Let $f \in L^{1}(k)$ be such that $\nu = fk$. Then there is a sequence $(f_{h})_{h} \sub C^{\infty}(\R^{d})$ such that $\frac{\alpha}{\beta'} \le f_{h} \le \frac{\beta}{\alpha'}$ and $f_{h} \to f$ in $L^{1}(k)$. Since $(f_{h}) \sub \mathcal{C}_{m}$, we may construct by a diagonal argument a subsequence $(h_{\varepsilon})$ such that $f_{\varepsilon} := f_{h_{\varepsilon}}$ satisfies $\frac{\alpha}{\beta'} \le f_{\varepsilon} \le \frac{\beta}{\alpha'}$ pointwise, $\n{f_{\varepsilon} - f}_{L^{1}(k)} \to 0$, and $T(\mu^{\varepsilon})^{-1}w_{f_{\varepsilon}^{-1}A^{\varepsilon}, \mu^{\varepsilon}} - f_{\varepsilon}k \rightharpoonup 0$ in $\mathcal{M}(\R^{d})$ weak. Hence, $\nu = fk$ is the effective measure of the problems $P(f_{\varepsilon}^{-1}A^{\varepsilon}, \mu^{\varepsilon})$.
\end{proof}

For $A, B \in M(\alpha, \beta)$, we write $A \le B$ if $A\xi \cdot \xi \le B\xi \cdot \xi$ $\mathcal{L}^{d}$-a.e. for every $\xi \in \R^{d}$. Effective measures are monotone (order-reversing) with respect to this order, in analogy with $H$-convergence (see \cite[Lemma 10.7]{tartar2009the}).
\begin{proposition}[Order property of effective measures]
    \label{thm: effective measures comparison symmetric}
    Let $(B^{\varepsilon}) \sub M(\alpha, \beta)$ be such that the problems $P(B^{\varepsilon}, \mu^{\varepsilon})$ admit an effective measure $k'$. If each $B^{\varepsilon}$ is symmetric $\mathcal{L}^{d}$-a.e. and $B^{\varepsilon} \le A^{\varepsilon}$ for every $\varepsilon$, then $k \le k'$.
\end{proposition}

\begin{proof}
    Let $f \in C_{b}(\R^{d})$ and put $u^{\varepsilon} := R_{A^{\varepsilon}, \mu^{\varepsilon}}f$ and $v^{\varepsilon} := R_{B^{\varepsilon}, \mu^{\varepsilon}}f$. Then 
    \begin{equation}
        \begin{split}
            \int A^{\varepsilon} \nabla u^{\varepsilon} \cdot \nabla u^{\varepsilon} + \int |u^{\varepsilon}|^{2} \, d\mu^{\varepsilon} -2\int f u^{\varepsilon} &\ge \int B^{\varepsilon} \nabla u^{\varepsilon} \cdot \nabla u^{\varepsilon} + \int |u^{\varepsilon}|^{2} \, d\mu^{\varepsilon} - 2\int f u^{\varepsilon}\\
            &\ge \int B^{\varepsilon} \nabla v^{\varepsilon} \cdot \nabla v^{\varepsilon} + \int |v^{\varepsilon}|^{2} \, d\mu^{\varepsilon} - 2 \int f v^{\varepsilon}.
        \end{split}
    \end{equation}
    Dividing both sides of the last display by $T(\mu^{\varepsilon})$ and passing to the limit by virtue of \eqref{eq: convergence to kf} and \eqref{eq: weak convergence of energy}, we obtain $-\int f^{2} \, dk \ge - \int f^{2} \, dk'$. Since $f \in C_{b}(\R^{d})$ was arbitrary, this implies $k \le k'$.
\end{proof}

\begin{remark}
    Similarly to Theorem \ref{thm: effective measures comparison}, one can show the following estimate. For $A, B \in M(\alpha, \beta)$ with $B \le A$ and $B$ symmetric $\mathcal{L}^{d}$-a.e. and $\mu \in \mathcal{M}_{\mathrm{cap}}^{b}$, we have 
    \begin{equation}
        \label{eq: quantitative core symmetric ver}
        \sup_{\varphi \in C_{c}^{\infty}(\R^{d}), \n{\nabla \varphi}_{L^{q'}(\Omega_{\mu})} \le 1, \varphi \ge 0}\int(w_{A, \mu} - w_{B, \mu})\varphi \lesssim_{d, \alpha, \beta}T(\mu)^{1 + \delta_{d, 1, q}}
    \end{equation}
    for every $1 \le q \le 2$. This is another proof of Proposition \ref{thm: effective measures comparison symmetric}. To see \eqref{eq: quantitative core symmetric ver}, note that $|B\xi \cdot \eta| = |B^{\frac{1}{2}}\xi \cdot B^{\frac{1}{2}}\eta| \le (B\xi \cdot \xi)^{\frac{1}{2}}(B\eta \cdot \eta)^{\frac{1}{2}}$ $\mathcal{L}^{d}$-a.e. for every $\xi, \eta \in \R^{d}$, since $B$ is symmetric $\mathcal{L}^{d}$-a.e. Taking into account this relation and observing that the first and the second equalities of \eqref{eq: bound by I + II + III} are still valid for general $A, B \in M(\alpha, \beta)$ and $\mu \in \mathcal{M}_{\mathrm{cap}}^{b}$, we obtain, for every nonnegative $\varphi \in C_{c}^{\infty}(\R^{d})$, 
    \begin{equation*}
        \begin{split}
            \int w_{A, \mu}\varphi &\le \int w_{A, \mu}B\nabla w_{B, \mu}\cdot \nabla \varphi\\
            &+ \frac{1}{2}\left( \int \varphi B\nabla w_{B, \mu} \cdot \nabla w_{B, \mu} + \int w_{B, \mu}^{2}\varphi \, d\mu\right) + \frac{1}{2}\left( \int \varphi B \nabla w_{A, \mu} \cdot \nabla w_{A, \mu} + \int w_{A, \mu}^{2}\varphi \, d\mu \right)\\
            &=: I + II' + III'.
        \end{split}
    \end{equation*}
    Use also the assumption $B \le A$ to bound $III' \le 2^{-1}(\int \varphi A \nabla w_{A, \mu} \cdot \nabla w_{A, \mu} + \int \varphi w_{A, \mu}^{2} \, d\mu)$.
  From the last two bounds and \eqref{eq: quantitative core of energy convergence}, we obtain \eqref{eq: quantitative core symmetric ver}.  
\end{remark}

\subsection{Proof of Theorem \ref{thm: quantitative error estimate}}
\label{sect: pf of quantitative error estimate}
\begin{lemma}
    \label{thm: gradient of torsion negative sobolev estimate}
    If $A \in M(1, \beta)$ and $\mu \in \mathcal{M}_{\mathrm{cap}}^{b}$, then for $1 \le p \le 2$ and $\varphi \in H^{1}_{\mu}$, 
    \begin{equation}
        \label{eq: Lp H1 estimate of multiplication by energy density of torsion func}
        \n{|A\nabla w_{A, \mu}|\varphi}_{p} \le \beta^{\frac{1}{2}}\n{w_{A, \mu}}_{\frac{2p}{2 - p}}\left(4\beta \int |\nabla \varphi|^{2} + 2\int \varphi^{2} \, d\mu\right)^{\frac{1}{2}}.
    \end{equation}
\end{lemma}

\begin{proof}
    We bound
    \begin{equation*}
        \begin{split}
            \n{|A\nabla w_{A, \mu}| \varphi}_{p} & = \n{\frac{|A\nabla w_{A, \mu}|}{w_{A, \mu}}\varphi w_{A, \mu}}_{p}\\
            &\le \n{\frac{|A\nabla w_{A, \mu}|}{w_{A, \mu}}\varphi}_{2}\n{w_{A, \mu}}_{\frac{2p}{2 - p}}\\
            &\le \left( \beta \int\frac{A \nabla w_{A, \mu} \cdot \nabla w_{A, \mu}}{w_{A, \mu}^{2}}\varphi^{2} \right)^{\frac{1}{2}}\n{w_{A, \mu}}_{\frac{2p}{2 - p}}\\
            &\underset{\eqref{eq: torsional_Hardy_2}}{\le} \beta^{\frac{1}{2}}\n{w_{A, \mu}}_{\frac{2p}{2 - p}}\left( 4\beta\int |\nabla \varphi|^{2} + 2\int \varphi^{2} \, d\mu \right)^{\frac{1}{2}}.
        \end{split}
    \end{equation*}
\end{proof}

\begin{lemma}
    \label{thm: energy estimate for u-wf}
    Let $A \in M(1, \beta)$ and $\mu \in \mathcal{M}_{\mathrm{cap}}(D)$, $D \sub \R^{d}$ bounded open. Let $f \in L^{\infty}(D) \cap W^{1, p'}(D)$, $1 \le p \le 2$. For a real-valued function $g$ defined on a subset of $\R^{d}$, let $\bar{g}:\R^{d} \to \R$ denote the zero extension. Then the product $w_{A, \mu} \bar{f} \in H^{1}_{\mu}$ and
    \begin{equation}
        \label{eq: energy estimate for u-wf}
        \n{u - w_{A, \mu}\bar{f}}_{\dot{H}^{1}_{\mu}} \le 3\beta\n{w_{A, \mu}}_{\frac{2p}{2 - p}}\n{\nabla f}_{L^{p'}(\Omega_{\mu})}.
    \end{equation}
    Here, we put $u = R_{A, \mu}f$.
\end{lemma}

\begin{proof}
    In the present proof, let us distinguish an element $v \in H^{1}(\R^{d})$ from its quasi-continuous representative, which we denote by $\tilde{v}$. Let $\varphi \in H^{1}_{\mu}$ be bounded $\mathcal{L}^{d}$-a.e. Since $f \in H^{1}(D) \cap L^{\infty}(D)$ and $\varphi \in H^{1}_{\mu} \sub H^{1}_{0}(D)$, $\bar{f}\varphi = \overline{f\varphi} \in H^{1}_{0}(D)$ and $\nabla (\bar{f} \varphi) = \bar{f} \nabla \varphi + \varphi \overline{\nabla f}$ in $L^{2}(\R^{d})$. Further, since $|\bar{f}\varphi| \le \n{f}_{L^{\infty}(D)}|\varphi|$ $\mathcal{L}^{d}$-a.e., $|\widetilde{\bar{f}\varphi}| \le \n{f}_{L^{\infty}(D)}|\tilde{\varphi}|$ q.e. Hence, $\bar{f}\varphi \in H^{1}_{\mu}$. Note also that $\widetilde{\bar{f}w_{A, \mu}}\tilde{\varphi} = \widetilde{w_{A, \mu}}\widetilde{\bar{f}\varphi}$ q.e., since both are quasi-continuous functions that coincide with $\bar{f}w_{A, \mu}\varphi$ $\mathcal{L}^{d}$-a.e. (\cite[Lemma 3.3.30]{henrot2018shape}). Now, we calculate 
    \begin{equation*}
        \begin{split}
            &\int_{\R^{d}} A\nabla(\bar{f}w_{A, \mu})\cdot \nabla \varphi + \int \widetilde{(\bar{f}w_{A, \mu})}\tilde{\varphi} \, d\mu\\ &= \int_{\R^{d}} A\nabla w_{A, \mu}\cdot \nabla(\bar{f} \varphi) + \int_{\R^{d}} w_{{A, \mu}}A{\overline{\nabla f}} \cdot \nabla \varphi - \int_{\R^{d}} \varphi A\nabla w_{A, \mu} \cdot {\overline{\nabla f}} + \int \widetilde{w_{A, \mu}}\widetilde{(\bar{f}\varphi)} \, d\mu\\
            &= \int_{\R^{d}} \bar{f}\varphi + \int_{D} w_{A, \mu}A{\nabla f} \cdot \nabla \varphi - \int_{D} \varphi A\nabla w_{A, \mu} \cdot {\nabla f},
        \end{split}
    \end{equation*}
    while
    \[\int_{\R^{d}} A\nabla u \cdot \nabla \varphi + \int \tilde{u}\tilde{\varphi} \, d\mu = \int_{D} f \varphi.\]
    Hence, 
    \begin{equation*}
        \begin{split}
            &\int_{\R^{d}} A\nabla(u - w_{A, \mu}\bar{f})\cdot \nabla \varphi + \int (\tilde{u} - \widetilde{w_{A, \mu}\bar{f}})\tilde{\varphi} \, d\mu\\
            &= -\int_{D} w_{A, \mu}A {\nabla f} \cdot \nabla \varphi + \int_{D} \varphi A \nabla w_{A, \mu} \cdot {\nabla f}\\
            &\le \beta\n{w_{A, \mu}}_{\frac{2p}{2 - p}}\n{\nabla f}_{L^{p'}(\Omega_{\mu})}\n{\nabla \varphi}_{2} + \n{\nabla f}_{L^{p'}(\Omega_{\mu})}\n{|A\nabla w_{A, \mu}|\varphi}_{p}\\
            &\underset{\eqref{eq: Lp H1 estimate of multiplication by energy density of torsion func}}{\le} \n{w_{A, \mu}}_{\frac{2p}{2 - p}}\n{\nabla f}_{L^{p'}(\Omega_{\mu})}\left( \beta\n{\nabla \varphi}_{2} +  \beta^{\frac{1}{2}}\sqrt{4\beta \n{\nabla \varphi}_{2}^{2} + 2\n{\varphi}_{L^{2}(\mu)}^{2}}\right)\\
            &\le 3\beta\n{w_{A, \mu}}_{\frac{2p}{2 - p}}\n{\nabla f}_{L^{p'}(\Omega_{\mu})}\n{\varphi}_{\dot{H}^{1}_{\mu}}.
        \end{split}
    \end{equation*}
    By Lemma \ref{thm: comparison}, we have the bound $\n{u}_{\infty} \le \n{f}_{L^{\infty}(D)}\n{w_{A, \mu}}_{\infty}$. Hence, we may substitute $\varphi = u - w_{A, \mu}\bar{f}$ in the previous display to obtain \eqref{eq: energy estimate for u-wf}.
\end{proof}

\begin{proof}[Proof of Theorem \ref{thm: quantitative error estimate}]
    The second inequality of \eqref{eq: quantitative error estimate} follows from Lemma \ref{thm: torsion and poincare const} and \eqref{eq: Ls_L1 estimate of torsion func}. Note that the assumed regularity of $f$ guarantees that the functions $f_{k} := (f \wedge k) \vee (-k)$ are such that $f_{k} \to f$ in $H^{-1}(D)$ strong as $k \to \infty$. Note also that $|\nabla f_{k}| \le |\nabla f|$ $\mathcal{L}^{d}$-a.e. on $D$. Thus, by virtue of Fatou's lemma, we may assume that $f \in L^{\infty}(D) \cap W^{1, p'}(D)$ to show the first inequality of \eqref{eq: quantitative error estimate}, in which case the desired estimate follows from Lemma \ref{thm: energy estimate for u-wf}.
\end{proof}

\subsection{Optimality of the quantitative error estimates}
\label{sect: optimality}
The optimality of the estimates \eqref{eq: quantitative error mue} and \eqref{eq: quantitative error estimate domains} is discussed in the next

\begin{example}[Perturbation by shrinking sets]
    \label{ex: optimality}
    Let $U \sub U(0,1) \sub D \sub \R^{d}$ be nonempty bounded open subsets containing the origin with $U$ connected. Let $(D^{\varepsilon})$ be a sequence of nonempty open subsets of $D \setminus \overline{U(0, 1)}$ with vanishing torsional rigidity. For each $\varepsilon$ and $0 < \rho < 1$, let $D^{\varepsilon, \rho} := D^{\varepsilon} \cup \rho U$.
    \begin{enumerate}
    \item Given $A \in M(1, \beta)$, 
    \begin{equation}
        \n{w_{A, D^{\varepsilon, \rho}}}_{s} \asymp_{d, U, \beta, s}\rho^{2 + \frac{d}{s}} + \n{w_{A, D^{\varepsilon}}}_{s}
    \end{equation}
    for every $1 \le s \le \infty$. 

    \item For $\delta > -\frac{d + 2}{2}$, define $f_{\delta} \in H^{-1}_{\mathrm{loc}}(\R^{d})$ by $f_{\delta}(x) := |x|^{\delta}$. Then we have
    \begin{equation}
        \label{eq: |u - wf| > r^2+ delta + d/q}
        \begin{split}
            \n{R_{A, D^{\varepsilon, \rho}}f_{\delta} - w_{A, D^{\varepsilon, \rho}}f_{\delta}}_{L^{q}(D)} &\ge \rho^{2 + \delta + \frac{d}{q}}\n{R_{A(\rho \cdot), U}f_{\delta} - w_{A(\rho \cdot), U}f_{\delta}}_{L^{q}(U)}\\
            &\gtrsim_{d, \beta, \delta, U}\rho^{2 + \delta + \frac{d}{q}}
        \end{split}
    \end{equation}
    for $q \in [1, \infty]$, $A \in M(1, \beta),$ and a nonzero $\delta > -\frac{d + 2}{2}$.
    \item Now let $(\rho_{\varepsilon}) \sub ]0, 1[$ be a sequence decreasing to $0$ such that there exists a constant $C_{1} > 0$ such that $C_{1}^{-1} \le \rho_{\varepsilon}^{-1}T(D^{\varepsilon})^{\frac{1}{2 + d}} \le C_{1}$ for every $\varepsilon$. Put $\tilde{D}^{\varepsilon} := D^{\varepsilon, \rho_{\varepsilon}}$. Then $T(\tilde{D}^{\varepsilon}) \asymp_{d, U, C_{1}} T(D^{\varepsilon})$ and
    \begin{equation}
        \n{R_{A, \tilde{D}^{\varepsilon}}f_{\delta} - w_{A, \tilde{D}^{\varepsilon}}f_{\delta}}_{L^{q}(D)} \gtrsim_{d, U, C_{1}, \beta, \delta}T(\tilde{D}^{\varepsilon})^{\frac{2 + \delta + \frac{d}{q}}{2 + d}}
    \end{equation}
    for $q \in [1, \infty]$, $A \in M(1, \beta)$, and a nonzero $\delta > -\frac{d + 2}{2}$.
    \item Let $p, q \in [1, 2]$. For $\varkappa \in \R$, define $\delta_{\varkappa}$ by $\varkappa = \frac{2 + \delta_{\varkappa} + \frac{d}{q}}{2 + d}$. Then $\varkappa > (\text{resp.}=) 1 + \delta_{d, p, q}$ if and only if $\delta_{\varkappa} - 1 >(\text{resp.}=) -\frac{d}{p'}$. Further, the torsional rigidity $T(\tilde{D}^{\varepsilon})$ vanishes as $\varepsilon \searrow 0$, and for every sequence $(A^{\varepsilon}) \sub M(\alpha, \beta)$, we have the following two-sided estimates.
    \begin{itemize}
        \item When $p = 1$, $f_{1} \in W^{1, \infty}(D)$ and 
        \begin{equation}
            \label{eq: sharpness f_1}
            T(\tilde{D}^{\varepsilon})^{\delta_{d, p, q}} \lesssim_{\substack{d, C_{1}, U,\\ \alpha, \beta}} \n{\frac{R_{A^{\varepsilon}, \tilde{D}^{\varepsilon}}f_{1}}{T(\tilde{D}^{\varepsilon})} - \frac{w_{A^{\varepsilon}, \tilde{D}^{\varepsilon}}}{T(\tilde{D}^{\varepsilon})}f_{1}}_{L^{q}(D)} \lesssim_{d, \alpha, \beta}\n{\nabla f_{1}}_{L^{\infty}(\tilde{D}^{\varepsilon})}T(\tilde{D}^{\varepsilon})^{\delta_{d, p, q}}.
        \end{equation}
        \item When $1 < p \le 2$, there is a $\varkappa_{1} = \varkappa_{1}(d, p, q) > 1 + \delta_{{d, p, q}}$ such that for every $1 + \delta_{d, p, q} < \varkappa < \varkappa_{1}$, $f_{\delta_{\varkappa}} \in W^{1, p'}(D)$ and 
        \begin{equation}
            \label{eq: almost sharpness f_delta}
            T(\tilde{D}^{\varepsilon})^{\varkappa - 1} \lesssim_{\substack{d, C_{1}, U,\\ \alpha, \beta, \varkappa, q}} \n{\frac{R_{A^{\varepsilon}, \tilde{D}^{\varepsilon}}f_{\delta_{\varkappa}}}{T(\tilde{D}^{\varepsilon})} - \frac{w_{A^{\varepsilon}, \tilde{D}^{\varepsilon}}}{T(\tilde{D}^{\varepsilon})}f_{\delta_{\varkappa}}}_{L^{q}(D)} \lesssim_{d, \alpha, \beta}\n{\nabla f_{\delta_{\varkappa}}}_{L^{p'}(\tilde{D}^{\varepsilon})}T(\tilde{D}^{\varepsilon})^{\delta_{d, p, q}}.
        \end{equation}
    \end{itemize}
    \end{enumerate}    
\end{example}

\begin{proof}
    To prove 1, we use \eqref{eq: scaling property of resolvent} and Lemma \ref{thm: torsion and poincare const} to calculate
    \[\n{w_{A, D^{\varepsilon, \rho}}}_{s} \asymp \n{w_{A, \rho U}}_{s} + \n{w_{A, D^{\varepsilon}}}_{s} = \rho^{\frac{d}{s} + 2}\n{w_{A(\rho \cdot), U}}_{s} + \n{w_{A, D^{\varepsilon}}}_{s} \asymp_{d, U, \beta, s} \rho^{\frac{d}{s} + 2} + \n{w_{A, D^{\varepsilon}}}_{s}.\]

    To prove 2, we bound
    \begin{equation*}
        \begin{split}
            \n{R_{A, D^{\varepsilon, \rho}}f_{\delta} - w_{A, D^{\varepsilon, \rho}}f_{\delta}}_{L^{q}(D)} &\ge \n{R_{A, \rho U}f_{\delta} - w_{A, \rho U}f_{\delta}}_{L^{q}(\rho U)}\\
            &= \rho^{\frac{d}{q}} \n{\left( R_{A, \rho U}f_{\delta} - w_{A, \rho U}f_{\delta} \right)(\rho \cdot)}_{L^{q}(U)}\\
            &\underset{\eqref{eq: scaling property of resolvent}}{=} \rho^{\frac{d}{q}}\n{\rho^{2}R_{A(\rho \cdot), U}(f_{\delta}(\rho \cdot)) - \rho^{2}w_{A(\rho \cdot), U}f_{\delta}(\rho \cdot)}_{L^{q}(U)}\\
            &= \rho^{2 + \delta + \frac{d}{q}}\n{R_{A(\rho \cdot), U}f_{\delta} - w_{A(\rho \cdot), U}f_{\delta}}_{L^{q}(U)}.
        \end{split}
    \end{equation*}
    
    The second inequality of \eqref{eq: |u - wf| > r^2+ delta + d/q} results from the following estimate:
    \begin{equation}
        \label{eq: inf B u - wf > 0}
        \inf_{B \in M(1, \beta)}\n{R_{B, U}f_{\delta} - w_{B, U}f_{\delta}}_{L^{1}(U)} > 0
    \end{equation}
    for every nonzero $\delta > -\frac{d + 2}{2}$. Since $U$ is bounded, the compactness of $M(1, \beta)$ with respect to $H$-convergence (see, e.g., \cite[Theorem 6.5]{tartar2009the}) guarantees the existence of a minimizer of the left-hand side of \eqref{eq: inf B u - wf > 0}, so that it suffices to show that
     $R_{B, U}f_{\delta} \neq w_{B, U}f_{\delta}$ for each $B \in M(1, \beta)$ and nonzero $\delta > -\frac{d + 2}{2}$. Fix $B \in M(1, \beta)$ and a nonzero $\delta > -\frac{d + 2}{2}$, and put $u := R_{B, U}f_{\delta}$. Toward a contradiction, suppose that $u = w_{B, U}f_{\delta}$. Then for every $1 \le s \le \infty$, $f_{\delta} \in L^{s}_{\mathrm{loc}}(U)$ whenever $u \in L^{s}_{\mathrm{loc}}(U)$, since $\inf_{\omega}w_{B, U} > 0$ for any $\omega \Subset U$. On the other hand, by \cite[Th\'{e}or\`{e}me 5.4]{Stampacchia1965le}, $u \in L^{s + \frac{2s^{2}}{d - 2s}}_{\mathrm{loc}}(U)$ whenever $\frac{2d}{d + 2} < s < \frac{d}{2}$ and $f_{\delta} \in L^{s}_{\mathrm{loc}}(U)$. From the above two considerations, we deduce that $f_{\delta} \in L^{s}_{\mathrm{loc}}(U)$ for some $s > \frac{d}{2}$, which in turn implies $f_{\delta} \in L^{\infty}_{\mathrm{loc}}(U)$ by virtue of \cite[Th\'{e}or\`{e}me 5.4]{Stampacchia1965le}. Hence, $\delta > 0$, and $\inf_{\omega}u = 0 = \inf_{U}u$ for any open $\omega \Subset U$ containing the origin. This implies, by virtue of the strong maximum principle (\cite[Theorem 8.19]{gilbarg1983elliptic}), that $u = 0$, and we come to a contradiction.

    Item 3 follows from  1 and 2. As for Item 4, choose $\varkappa_{0}$ so that $\delta_{\varkappa_{0}} = 0$. Set $\varkappa_{1} := \varkappa_{0}$ if $\varkappa_{0} > 1 + \delta_{d, p, q}$ and $\varkappa_{1} := \infty$ otherwise. The remaining parts are direct consequences of 3 and Theorem \ref{thm: quantitative error estimate}.
\end{proof}

\begin{remark}
    \label{rmk: exact sharpness lorentz}
    By using \eqref{eq: quantitative error estimate lorentz} instead of Theorem \ref{thm: quantitative error estimate} in the proof of Item 4 of Example \ref{ex: optimality}, we may prove the following two-sided estimate: when $p \in [1, 2] \setminus \{\frac{d}{d - 1}, 2\}$ and $1 \le q \le 2$, $f_{\delta_{1 + \delta_{d, p, q}}} = f_{1 - \frac{d}{p'}} \in L^{p'}(D)$, $\n{\left|\nabla f_{1 - \frac{d}{p'}}\right|\chi_{D}}_{{p', \infty}} < \infty$, and 
        \begin{equation}
            \label{eq: sharpness weak norm}
            T(\tilde{D}^{\varepsilon})^{\delta_{d, p, q}} \lesssim_{\substack{d, C_{1}, U,\\ \alpha, \beta}} \n{\frac{R_{A^{\varepsilon}, \tilde{D}^{\varepsilon}}f_{1 - \frac{d}{p'}}}{T(\tilde{D}^{\varepsilon})} - \frac{w_{A^{\varepsilon}, \tilde{D}^{\varepsilon}}}{T(\tilde{D}^{\varepsilon})}f_{1 - \frac{d}{p'}}}_{L^{q}(D)}
            \lesssim_{d, \alpha, \beta, p}\n{\left|\nabla f_{1 - \frac{d}{p'}}\right|\chi_{\tilde{D}^{\varepsilon}}}_{p', \infty}T(\tilde{D}^{\varepsilon})^{\delta_{d, p, q}}.
        \end{equation}
\end{remark}
\section{Problems for rougher source terms}
\label{sect: rougher source}
In this section, we fix $0 < \alpha \le \beta < \infty$, $(A^{\varepsilon}) \sub M(\alpha, \beta)$, a bounded open $D \sub \R^{d}$, and a sequence $(\mu^{\varepsilon}) \sub \mathcal{M}_{\mathrm{cap}}(D)$ such that the problems $P(A^{\varepsilon}, \mu^{\varepsilon})$ admit an effective measure $k$.

Since the effective measure $k$ can be any nonzero nonnegative finite Borel measure supported in $\overline{D}$ by Proposition \ref{thm: attainability}, the source term $f$ in Theorem \ref{thm: main result} has to be continuous on $\overline{D}$ to make the expression $kf$ well-defined without imposing further assumptions. In other words, we need regularity of the effective measure to obtain convergence of the form \eqref{eq: convergence to kf} for rougher $f$. The next proposition characterizes this in terms of uniform $L^{p}$--$L^{1}$ estimates for torsion functions or, equivalently, of Poincar\'{e} constants. We emphasize that the condition does not depend on the coefficients $A^{\varepsilon}$.
\begin{proposition}
    \label{thm: rougher source term}
    Let $1 < r \le 2$. The following conditions are equivalent.
    \begin{enumerate}
        \item $k \in L^{\frac{r}{2 - r}}(\R^{d})$ and for every $f \in L^{r'}(\R^{d})$, $\frac{R_{A^{\varepsilon}, \mu^{\varepsilon}}f}{T(\mu^{\varepsilon})} \xrightharpoonup{L^{r}(\R^{d})} kf$ weakly.
        \item There is a constant $B > 0$ such that 
        \begin{equation}
            \label{eq: torsion Lr/(2-r) L1 bound}
            \n{w_{A^{\varepsilon}, \mu^{\varepsilon}}}_{\frac{r}{2 - r}} \le B \n{w_{A^{\varepsilon}, \mu^{\varepsilon}}}_{1}
        \end{equation}
        for every $\varepsilon$.
        \item There is a constant $B' > 0$ such that 
        \begin{equation}
            \n{w_{\mu^{\varepsilon}}}_{\frac{r}{2 - r}} \le B' \n{w_{\mu^{\varepsilon}}}_{1}
        \end{equation}
        for every $\varepsilon$.
        \item There is a constant $B'' > 0$ such that 
        \begin{equation}
            \lambda_{r}(\mu^{\varepsilon}) \ge B'' \lambda_{1}(\mu^{\varepsilon})
        \end{equation}
        for every $\varepsilon$.
    \end{enumerate}
    If these conditions are satisfied, then
    \begin{equation}
        \label{eq: quantitative error estimate general}
        \n{\frac{R_{A^{\varepsilon}, \mu^{\varepsilon}}f}{T(\mu^{\varepsilon})} - \frac{w_{A^{\varepsilon}, \mu^{\varepsilon}}}{T(\mu^{\varepsilon})}f}_{q} \lesssim_{d, \alpha, \beta}\left( 1 + B^{\frac{3}{2}} \right)\n{\nabla f}_{L^{p'}(\Omega_{\mu^{\varepsilon}})}T(\mu^{\varepsilon})^{\delta_{d, p, q, r}}
    \end{equation}
    for every $p, q \in [1, 2]$ and $f \in H^{-1}(D)$ satisfying \eqref{eq: admissible integrability of source} and $\nabla f \in L^{p'}(D)$. Here, 
    \begin{equation}
        \delta_{d, p, q, r} := \frac{1}{2} - \frac{d}{2 + \frac{2 - r}{r}d}\left( \frac{1}{2}\left( \frac{2 - r}{r} - \frac{2 - q}{q} \right)_{+} + \left( \frac{2 - r}{r} - \frac{2 - p}{2p} \right)_{+} \right).
    \end{equation}
    In particular, $\delta_{d, p, q, r} = \frac{1}{2}$ when $r = 2$.
\end{proposition}

\begin{proof}
    By Lemma \ref{thm: comparison between op norm of resolvent and poincare const} below, we see that Conditions 2, 3, and 4 are equivalent to the condition that $k \in L^{\frac{r}{2 - r}}(\R^{d})$ and $\sup_{\varepsilon}T(\mu^{\varepsilon})^{-1}\n{R_{A^{\varepsilon}, \mu^{\varepsilon}}}_{r' \to r} < \infty$ hold. The last condition is in turn equivalent to Condition 1 by the uniform boundedness principle, the convergence \eqref{eq: convergence to kf}, and the density of $C(\overline{D})$ in $L^{r'}(D)$.

    Next, we show \eqref{eq: quantitative error estimate general}. To this end, we may assume that $\alpha = 1$. Suppose first that $r < 2$. Note that
    \begin{equation}
        \label{eq: Ls le B times torsion power}
        \lambda_{\frac{2s}{1 + s}}(\mu^{\varepsilon})^{-1} \underset{\text{Lemma \ref{thm: torsion and poincare const}}}{\lesssim_{\beta}} \n{w_{A^{\varepsilon}, \mu^{\varepsilon}}}_{s} \underset{\eqref{eq: Ls_Lt estimate of torsion func}, \eqref{eq: torsion Lr/(2-r) L1 bound}}{\lesssim_{d}} (1 + B)\n{w_{A^{\varepsilon}, \mu^{\varepsilon}}}_{1}^{1 - \left(1 - \gamma_{\frac{r}{2 - r}}\right)\left( 1 - \frac{1}{s}\frac{r}{2 - r} \right)_{+}}
    \end{equation}
    for every $\varepsilon$ and $1 \le s \le \infty$. By the first line of \eqref{eq: quantitative error estimate}, \eqref{eq: Ls le B times torsion power} with $s = \frac{q}{2 - q}$ and with $s = \frac{2p}{2 - p}$, and \eqref{eq: L1 norm of wA}, we get
    \[\n{\frac{R_{A^{\varepsilon}, \mu^{\varepsilon}}f}{T(\mu^{\varepsilon})} - \frac{w_{A^{\varepsilon}, \mu^{\varepsilon}}}{T(\mu^{\varepsilon})}f}_{q} \lesssim_{d, \beta}(1 + B)^{\frac{3}{2}}\n{\nabla f}_{L^{p'}(\Omega_{\mu^{\varepsilon}})}T(\mu^{\varepsilon})^{\delta},\]
    where
    \[\delta = \frac{1}{2} - \left( 1 - \gamma_{\frac{r}{2 - r}} \right)\left(\frac{1}{2} \left( 1 - \frac{2 - q}{q}\frac{r}{2 - r} \right)_{+} + \left( 1 - \frac{2- p}{2p}\frac{r}{2 - r} \right)_{+} \right).\]
    Further calculation of the right-hand side of the last display shows that $\delta = \delta_{d, p, q, r}$. If $r = 2$, the same argument works since \eqref{eq: Ls le B times torsion power} still holds true if we set $(1 - \frac{1}{s}\frac{r}{2 - r})_{+} = 0$.
\end{proof}

\begin{lemma}
    \label{thm: comparison between op norm of resolvent and poincare const}
    For $A \in M(\alpha, \beta)$, $\mu \in \mathcal{M}_{\mathrm{cap}}^{b} \setminus \{I_{\varnothing}\}$, and $1 \le q \le 2$, 
    \begin{equation}
        \label{eq: comparison between op norm of resolvent and poincare const}
        1 \le \lambda^{(\beta)}_{q}(\mu)\n{R_{A, \mu}}_{q' \to q},\, \lambda^{(\alpha)}_{q}(\mu)\n{R_{A, \mu}}_{q' \to q} \le 1.
    \end{equation}
    Hence, 
    \begin{equation}
        \label{eq: comparison poincare const torsion resolvent}
        \n{w_{A, \mu}}_{\frac{q}{2 - q}} \asymp_{\alpha, \beta, q} \lambda_{q}^{(\alpha)}(\mu)^{-1} \asymp_{\alpha} \lambda_{q}^{(1)}(\mu)^{-1} \asymp_{\beta} \lambda_{q}^{(\beta)}(\mu)^{-1} \asymp_{\alpha, \beta} \n{R_{A, \mu}}_{q' \to q}.
    \end{equation}
\end{lemma}

\begin{proof}
    Let $u \in H^{1}_{\mu}$ be nonnegative bounded and consider the solution $v \in H^{1}_{\mu}$ to the problem
    \begin{equation}
        \label{eq: aux prob for estimating op norm of resolvent}
            -\mathrm{div}(A\nabla v) + \mu v = u^{q - 1} \text{ in $(H^{1}_{\mu})^{*}$, $v \in H^{1}_{\mu}$}.
    \end{equation}
    Testing \eqref{eq: aux prob for estimating op norm of resolvent} with $u$ and $v$, we obtain
    \[\n{u}_{q}^{q} = \int A\nabla v \cdot \nabla u + \int vu \, d\mu \le \left( \frac{1}{\beta}\int |A \nabla v|^{2} + \int v^{2} \, d\mu \right)^{\frac{1}{2}}\left( \beta\int |\nabla u|^{2} + \int u^{2} \, d\mu \right)^{\frac{1}{2}}\]
    and
    \[\begin{split}
        \frac{1}{\beta}\int |A \nabla v|^{2} + \int v^{2} \, d\mu &\le \int A\nabla v \cdot \nabla v + \int v^{2} \, d\mu = \int u^{q - 1}v\\
        &\le \n{u}_{q}^{q - 1}\n{v}_{q} \le \n{u}_{q}^{q - 1}\n{R_{A, \mu}}_{q' \to q}\n{u}_{q}^{q - 1},
    \end{split}\]
    respectively. Combining these two estimates, we get $\n{u}_{q} \le (\beta \n{\nabla u}_{2}^{2} + \n{u}_{L^{2}(\mu)}^{2})^{\frac{1}{2}}\n{R_{A, \mu}}_{q' \to q}^{\frac{1}{2}}$. This yields the first inequality of \eqref{eq: comparison between op norm of resolvent and poincare const}.

    On the other hand, for every $f \in L^{q'}(\R^{d})$ and $u := R_{A, \mu}f$,  
    \[\n{u}_{q}^{2} \le \lambda_{q}^{(\alpha)}(\mu)^{-1}\left( \alpha \int |\nabla u|^{2} + \int u^{2} \, d\mu \right) \le \lambda_{q}^{(\alpha)}(\mu)^{-1}\int fu,\]
    which implies $\n{R_{A, \mu}f}_{q} \le \lambda_{q}^{(\alpha)}(\mu)^{-1}\n{f}_{q'}$.

    The estimate \eqref{eq: comparison poincare const torsion resolvent} now follows from \eqref{eq: comparison between op norm of resolvent and poincare const} and Lemma \ref{thm: torsion and poincare const}.
\end{proof}

\cite[Lemma 6]{vandenBerg2022on}  by van den Berg, Bucur, and Kappeler provides a sufficient condition for the stability of the conditions in Proposition \ref{thm: rougher source term} under a perturbation by nonnegative potentials.
\begin{proposition}[{\cite[Lemma 6]{vandenBerg2022on}}]
    \label{thm: stability}
    Suppose that $\mu^{\varepsilon} = I_{D^{\varepsilon}}$, $D^{\varepsilon} \sub D$ nonempty bounded open. Let a Lebesgue measurable $V^{\varepsilon}:D \to [0, \infty[$ be such that 
    \begin{equation}
        \Lambda := \sup_{\varepsilon}\frac{\n{V^{\varepsilon}}_{L^{\infty}(D^{\varepsilon})}}{\lambda_{2}(D^{\varepsilon})} < \infty.
    \end{equation}
    Then 
    \begin{equation}
        w_{D^{\varepsilon}} \lesssim_{\Lambda} w_{D^{\varepsilon} + V^{\varepsilon}\mathcal{L}^{d}\resmes D^{\varepsilon}} \le w_{D^{\varepsilon}}
    \end{equation}
    in $D^{\varepsilon}$. As a consequence,
    \begin{equation}
        \n{w_{D^{\varepsilon} + V^{\varepsilon}\mathcal{L}^{d}\resmes D^{\varepsilon}}}_{s} \asymp_{\Lambda} \n{w_{D^{\varepsilon}}}_{s}
    \end{equation}
    for every $1 \le s \le \infty$. Hence, Condition 3 in Proposition \ref{thm: rougher source term} holds for $(D^{\varepsilon})_{\varepsilon}$ if and only if the corresponding condition for $(D^{\varepsilon} + V^{\varepsilon}\mathcal{L}^{d}\resmes D^{\varepsilon})_{\varepsilon}$ holds.
\end{proposition}

\begin{remark}
    In the setting of Proposition \ref{thm: rougher source term}, consider the case $\mu^{\varepsilon} = I_{D^{\varepsilon}}$, $D^{\varepsilon} \sub D$ nonempty open. Then clearly the validity of Condition 3 for some $B' < \infty$ is equivalent to the following condition: $\inf_{\varepsilon}\mathcal{L}^{d}(D^{\varepsilon}) > 0$ and $\inf_{\varepsilon}F_{\frac{r}{2 - r}}(D^{\varepsilon}) > 0$, where for every $1 \le s \le \infty$, $F_{s}$ is a scale-invariant functional defined for each nonempty bounded open $\Omega \sub \R^{d}$ by 
    \begin{equation}
        F_{s}(\Omega) = \frac{\n{w_{\Omega}}_{1}}{\n{w_{\Omega}}_{s}\mathcal{L}^{d}(\Omega)^{\frac{1}{s'}}}.
    \end{equation}
    In the literature, $F_{\infty}(\Omega)$ is called the \emph{efficiency} or \emph{mean-to-max ratio} of the torsion function on $\Omega$ (see \cite{vandenBerg2022on,Briani2023mean} and the references therein).
\end{remark}

We now present two concrete examples satisfying the conditions in Proposition \ref{thm: rougher source term}.
\begin{example}
    Let $d \ge 3$ and $D^{\varepsilon}_{\eta}$ be as in Example \ref{ex: periodically perforated domains}. Then for $1 \le \eta < \frac{d}{d - 2}$, $D^{\varepsilon}_{\eta}$ satisfies the conditions in Proposition \ref{thm: rougher source term} for $r = 2$. Hence the error is $O(T(D^{\varepsilon}_{\eta})^{\delta_{d, p, q, 2}}) = O(\varepsilon^{1 - \frac{d - 2}{2}(\eta - 1)})$. Indeed, the convergence stated in Example \ref{ex: periodically perforated domains} for $1 \le \eta < \frac{d}{d - 2}$ is exactly Condition 1 with $r = 2$ and $k = \mathcal{L}^{d}(D)^{-1}\mathcal{L}^{d} \resmes D \in L^{\infty}(D)$.
\end{example}

\begin{example}
    Let $D_{h, \varkappa}$ be as in Example \ref{ex: disjoint balls}. Then the conditions in Proposition \ref{thm: rougher source term} for some (every) $1 < r \le 2$ are satisfied if and only if $\varkappa = 1$, while the effective measure $k = \lim_{h}\frac{w_{D_{h, \varkappa}}}{T(D_{h, \varkappa})} = \frac{1}{\mathcal{L}^{d}(D)}\mathcal{L}^{d}\resmes D$ does not depend on $\varkappa$. This follows from \eqref{eq: Ls norm of torsion of disjoint balls}.
\end{example}

The optimality of \eqref{eq: quantitative error estimate general} is discussed in the next

\begin{example}[Optimality of \eqref{eq: quantitative error estimate general}]
    \label{ex: optimality general}
    Let $U, D, D^{\varepsilon}, D^{\varepsilon, \rho}$ be as in Example \ref{ex: optimality}. Let $1 < r \le 2$ and suppose that $D^{\varepsilon}$ satisfies the conditions in Proposition \ref{thm: rougher source term} for $r$.
    \begin{enumerate}
        \item Let $A^{\varepsilon} \in M(\alpha, \beta)$ and $(\rho_{\varepsilon}) \sub ]0, 1[$ be a sequence decreasing to $0$ such that there exists a constant $C_{1} < \infty$ such that
        \begin{equation}
            C_{1}^{-1} \le T(D^{\varepsilon})^{-\frac{1}{2 + \frac{2 - r}{r}d}} \rho_{\varepsilon} \le C_{1}
        \end{equation}
        for every $\varepsilon$. Then
        \begin{equation}
            \n{w_{A^{\varepsilon}, \tilde{D}^{\varepsilon}}}_{s} \asymp_{d, U, \alpha, \beta, s} \rho_{\varepsilon}^{2 + \frac{d}{s}} + \n{w_{A^{\varepsilon}, D^{\varepsilon}}}_{s} 
            \begin{cases}
                \asymp_{\alpha, \beta, C_{1}, B'} T(D^{\varepsilon}) & \text{ if $1 \le s \le \frac{r}{2 - r}$}\\
                \ge \rho_{\varepsilon}^{2 + \frac{d}{s}} \gg T(\tilde{D}^{\varepsilon}) \text{ as $\varepsilon \searrow 0$} & \text{ if $\frac{r}{2 - r} < s \le \infty$}.
            \end{cases}
        \end{equation}
        Here, $B' := \sup_{\varepsilon}\frac{\n{w_{D^{\varepsilon}}}_{\frac{r}{2 - r}}}{\n{w_{D^{\varepsilon}}}_{1}}$, and we put $\tilde{D}^{\varepsilon} := D^{\varepsilon, \rho_{\varepsilon}}$.
        \item Let $p, q \in [1, 2]$ such that $q \ge r$ and $\frac{2p}{2 - p} \ge \frac{r}{2 - r}$. For $\varkappa \in \R$, define $\delta_{\varkappa}$ by $ \varkappa = \frac{2 + \delta_{\varkappa} + \frac{d}{q}}{2 + \frac{2 - r}{r}d}$. Then $\varkappa > (\text{resp.} =) 1 + \delta_{d, p, q, r}$ if and only if $\delta_{\varkappa} - 1 > (\text{ resp.} =) - \frac{d}{p'}$. Further, for an arbitrary sequence $(A^{\varepsilon}) \sub M(\alpha, \beta)$, the torsional rigidity $T(\tilde{D}^{\varepsilon})$ vanishes as $\varepsilon \searrow 0$, and we have the following two-sided estimates.
        \begin{itemize}
            \item When $p = 1$, $f_{1} \in W^{1, \infty}(D)$ and
            \begin{equation}
                T(\tilde{D}^{\varepsilon})^{\delta_{d, p, q, r}} \lesssim_{\substack{d, C_{1}, U,\\ \alpha, \beta}} \n{\frac{R_{A^{\varepsilon}, \tilde{D}^{\varepsilon}}f_{1}}{T(\tilde{D}^{\varepsilon})} - \frac{w_{A^{\varepsilon}, \tilde{D}^{\varepsilon}}}{T(\tilde{D}^{\varepsilon})}f_{1}}_{L^{q}(D)} \lesssim_{\substack{d, \alpha,\\ \beta, B'}}\n{\nabla f_{1}}_{L^{\infty}(\tilde{D}^{\varepsilon})}T(\tilde{D}^{\varepsilon})^{\delta_{d, p, q, r}}.
            \end{equation}
            \item When $p > 1$, there is a $\varkappa_{1} = \varkappa_{1}(d, p, q, r) > 1 + \delta_{d, p, q, r}$ such that for every $1 + \delta_{d, p, q, r} < \varkappa < \varkappa_{1}$, $f_{\delta_{\varkappa}} \in W^{1, p'}(D)$ and 
            \begin{equation}
                T(\tilde{D}^{\varepsilon})^{\varkappa - 1} \lesssim_{\substack{d, C_{1}, U,\\ \alpha, \beta, \varkappa, q, r}} \n{\frac{R_{A^{\varepsilon}, \tilde{D}^{\varepsilon}}f_{\delta_{\varkappa}}}{T(\tilde{D}^{\varepsilon})} - \frac{w_{A^{\varepsilon}, \tilde{D}^{\varepsilon}}}{T(\tilde{D}^{\varepsilon})}f_{\delta_{\varkappa}}}_{L^{q}(D)} \lesssim_{\substack{d, \alpha,\\ \beta, B'}}\n{\nabla f_{\delta_{\varkappa}}}_{L^{p'}(\tilde{D}^{\varepsilon})}T(\tilde{D}^{\varepsilon})^{\delta_{d, p, q, r}}.
            \end{equation}
        \end{itemize}
        Here, $f_{\delta}(x) := |x|^{\delta}$.
    \end{enumerate}
\end{example}

\begin{proof}
    Item 1 follows from Example \ref{ex: optimality} (1). Item 2 follows from Example \ref{ex: optimality} (2) and Proposition \ref{thm: rougher source term}.
\end{proof}

The conditions in Proposition \ref{thm: rougher source term} also enable us to describe the asymptotic behavior of the energy density for rougher source terms in stronger topologies than the sources and topologies we have considered in Proposition \ref{thm: convergence of energy}. As an example, we formulate

\begin{proposition}
    \label{thm: H-1 strong convergence of the scaled energy density}
    Let $A \in M(\alpha, \beta)$, $\mu \in \mathcal{M}_{\mathrm{cap}}^{b}$, $1 \le p \le \frac{d}{d - 1}$, and $f \in L^{p'}(\R^{d})$ with $u := R_{A, \mu}f$. Then
    \begin{equation}
        \label{eq: quantitative part for H-1 convergence of energy}
        \sup_{\substack{\varphi \in C_{c}^{\infty}(\R^{d}),\\ \n{\nabla \varphi}_{L^{2}(\Omega_{\mu})} \le 1}}\left| \int \varphi A\nabla u \cdot \nabla u + \int \varphi u^{2} \, d\mu - \int \varphi u f \right| \lesssim_{d, \alpha, \beta} \n{f}_{L^{p'}(\Omega_{\mu})}^{2}\n{w_{\mu}}_{\infty}^{\frac{d}{2}\left( \frac{1}{d} - \frac{1}{p'} \right) + \frac{1}{2}}\n{w_{\mu}}_{\frac{p}{2 - p}}^{\frac{1}{2}}.
    \end{equation}
    As a consequence, if $(\mu^{\varepsilon})$ satisfies the conditions in Proposition \ref{thm: rougher source term} for $r = 2$, then 
    \begin{equation}
        \label{eq: H-1 strong convergence}
        T(\mu^{\varepsilon})^{-1}\left( A^{\varepsilon} \nabla u^{\varepsilon} \cdot \nabla u^{\varepsilon} \mathcal{L}^{d} + |u^{\varepsilon}|^{2} \, \mu^{\varepsilon} \right) \xrightarrow{H^{-1}(\R^{d})} kf^{2}
    \end{equation}
    strongly for every $f \in L^{p'}(\R^{d})$ with $p' > d$ and $u^{\varepsilon} := R_{A^{\varepsilon}, \mu^{\varepsilon}}f$.
\end{proposition}

\begin{proof}
    Let $\varphi \in C_{c}^{\infty}(\R^{d})$ such that $\n{\nabla \varphi}_{L^{2}(\Omega_{\mu})} \le 1$. We bound
    \begin{equation}
        \begin{split}
            \left| \int \varphi A\nabla u \cdot \nabla u + \int \varphi u^{2} \, d\mu - \int \varphi u f \right| & \underset{\text{\eqref{eq: convergence of energy Leibniz}}}{\le} \n{u}_{\infty}\n{A\nabla u}_{2}\\
            &\lesssim_{\beta} \n{u}_{\infty}\n{f}_{L^{p'}(\Omega_{\mu})}^{\frac{1}{2}}\n{u}_{p}^{\frac{1}{2}}\\
            &\le \n{f}_{L^{p'}(\Omega_{\mu})}^{2}\n{R_{A, \mu}}_{p' \to \infty}\n{R_{A, \mu}}_{p' \to p}^{\frac{1}{2}}.
        \end{split}
    \end{equation}
    From Lemma \ref{thm: Linfty Lr estimate for solutions}, we get the bound $\n{R_{A, \mu}}_{p' \to \infty} \lesssim_{d, \alpha} \n{R_{A, \mu}}_{p' \to p'}^{\gamma_{\frac{1}{p - 1}, p}}$, where $\gamma_{\frac{1}{p - 1}, p}$ is defined by \eqref{eq: gamma t r}. Hence, by the Riesz--Thorin theorem, we may further bound the right-hand side of the last display by 
    \[C(d, \alpha, \beta)\n{f}_{L^{p'}(\Omega_{\mu})}^{2}\n{w_{\mu}}_{\infty}^{\gamma_{\frac{1}{p - 1}, p}}\n{R_{A, \mu}}_{p' \to p}^{\frac{1}{2}},\]
    since $\n{R_{A, \mu}}_{\infty \to \infty} \le \n{w_{A, \mu}}_{\infty}$ by Lemma \ref{thm: comparison} and $\n{R_{A, \mu}}_{2 \to 2} \asymp_{d, \alpha, \beta} \n{w_{A, \mu}}_{\infty} \asymp_{d, \alpha, \beta} \n{w_{\mu}}_{\infty}$ by Lemma \ref{thm: comparison between op norm of resolvent and poincare const}. We obtain the estimate \eqref{eq: quantitative part for H-1 convergence of energy} from the last two displays and Lemma \ref{thm: comparison between op norm of resolvent and poincare const}. 

    Suppose that $(\mu^{\varepsilon})$ satisfies the conditions in Proposition \ref{thm: rougher source term} for $r = 2$ and $f \in L^{p'}(\R^{d})$ with $p' > d$. Then $T(\mu^{\varepsilon})^{-1}u^{\varepsilon} \rightharpoonup kf$ in $L^{2}(\R^{d})$ weak, and hence $T(\mu^{\varepsilon})^{-1}u^{\varepsilon}f \rightharpoonup kf^{2}$ in $L^{\frac{2p'}{2 + p'}}(\R^{d})$ weak. Now the $H^{-1}(\R^{d})$ convergence \eqref{eq: H-1 strong convergence} follows from \eqref{eq: quantitative part for H-1 convergence of energy}.
\end{proof}
\appendix
\section{Auxiliary lemmas}
\begin{lemma}
    \label{thm: multi general}
    \begin{itemize}
        \item If $d \ge 3$, then for every $1 \le s \le r \le \frac{2d}{d - 2}$ and $u \in C^{\infty}_{c}(\R^{d})$,
        \begin{equation}
            \label{eq: multi d ge 3}
            \n{u}_{r} \le C(d)\n{\nabla u}_{2}^{\theta}\n{u}_{s}^{1 - \theta}.
        \end{equation}
        \item If $d = 2$, then for every $1 \le s \le r < \infty$ and $u \in C^{\infty}_{c}(\R^{d})$, 
        \begin{equation}
            \label{eq: multi nash}
            \n{u}_{r} \le C(d)^{r - s}\n{\nabla u}_{2}^{\theta}\n{u}_{s}^{1 - \theta}.
        \end{equation}
    \end{itemize}
    Here, $\theta \in [0, 1]$ is defined by
    \begin{equation}
        \frac{1}{r} = \frac{1}{s}(1 - \theta) + \frac{d - 2}{2d}\theta.
    \end{equation}
\end{lemma}

\begin{proof}
    The inequality \eqref{eq: multi d ge 3} is an immediate consequence of the Sobolev inequality $\n{u}_{\frac{2d}{d - 2}} \le C(d)\n{\nabla u}_{2}$ and an interpolation.

    The inequality $\eqref{eq: multi nash}$ can be proven in the following way (cf. \cite[p. 129]{nirenberg1959on}, see also \cite[Section 2.3.12]{Mazya2011sobolevspaces}). By interpolation, it suffices to show the case where $m :=\frac{d - 1}{d} (r - s)$ is a positive integer. Given a nonnegative $u \in C^{\infty}_{c}(\R^{d})$ and $p > 1$, by plugging $u^{p}$ in the Sobolev inequality $\n{v}_{\frac{d}{d - 1}} \le C(d) \n{\nabla v}_{1}$, and using the H\"{o}lder inequality, we get 
    \[\n{u}_{\frac{d}{d - 1}p}^{p} \le C(d)p\n{\nabla u}_{d}\n{u}_{\frac{d}{d - 1}(p - 1)}^{p - 1}.\]
    Iterating the last estimate, we obtain 
    \[\n{u}_{\frac{d}{d - 1}p}^{p} \le C(d)^{l}p\cdots (p - l + 1)\n{\nabla u}_{d}^{l}\n{u}_{\frac{d}{d - 1}(p - l)}^{p - l}\]
    for every positive integer $l$ and every real $p > l$. Thus, 
    \[\n{u}_{\frac{d}{d - 1}p} \le C(d)^{\frac{l}{p}}p^{\frac{l}{p}}\n{\nabla u}_{d}^{\frac{l}{p}}\n{u}_{\frac{d}{d - 1}(p - l)}^{1 - \frac{l}{p}} \le C(d)^{l}\n{\nabla u}_{d}^{\frac{l}{p}}\n{u}_{\frac{d}{d - 1}(p - l)}^{1 - \frac{l}{p}}.\]
    The last display with $p = \frac{d - 1}{d}r$ and $l = m$ proves the desired multiplicative inequality.
\end{proof}

\begin{lemma}[cf. {\cite[Proposition 3.1]{brasco2017compact}}]
    \label{thm: Linfty Lr estimate for solutions}
    Let $1 \le \beta < \infty$, $1 \le r < \frac{d}{d - 2}$, and $1 \le t < \infty$ ($\frac{d}{d - 2} := \infty$ if $d = 2$). Then
    \begin{equation}
        \label{eq: Linfty Lr estimate for solutions}
        \n{R_{A, \mu}f}_{\infty} \le C(d, r)\n{f}_{r'}^{1 - \gamma_{t, r}}\n{R_{A, \mu}f}_{tr}^{\gamma_{t, r}}
    \end{equation}
    for every $f \in L^{r'}(\R^{d})$, $A \in M(1, \beta)$, and $\mu \in \mathcal{M}_{\mathrm{cap}}^{b}$. Here, 
    \begin{equation}
        \label{eq: gamma t r}
        \gamma_{t, r} := \frac{(d - (d - 2)r)t}{d + (d - (d - 2)r)t},
    \end{equation}
    and for each $d \ge 2$, $C(d, r)$ appearing in \eqref{eq: Linfty Lr estimate for solutions} is bounded if $r$ is bounded away from $\frac{d}{d - 2}$.
\end{lemma}

\begin{proof}
    Put $u := R_{A, \mu}f$. It suffices to bound $\n{u_{+}}_{\infty}$ by the right-hand side of \eqref{eq: Linfty Lr estimate for solutions}. Clearly we may also assume $u_{+} \neq 0$. Fix $k \in [0, \infty[$ and put $v := (u - k)_{+}$. Since $R_{A, \mu}(f_{+})$ is a subsolution of $-\mathrm{div}(A\nabla \cdot) = f_{+}$ in $\R^{d}$ by \cite[Proposition 2.6]{dalmaso1994new} and since $r' > \frac{d}{2}$, $R_{A, \mu}(f_{+}) \in L^{\infty}(\R^{d})$. Hence, by Lemma \ref{thm: comparison}, $u_{+} \in L^{\infty}(\R^{d})$.  By using $v^{t}$ as a test function for the definition \eqref{eq: weak formulation of relaxed Dirichlet} for $u$, we obtain 
    \begin{equation}
        \int f v^{t} = \int tv^{t - 1} A \nabla u \cdot \nabla v + \int uv^{t} \, d\mu \ge \left( \frac{2}{t + 1} \right)^{2}t \int|\nabla (v^{\frac{t + 1}{2}})|^{2}.
    \end{equation}
     Thus, by the H\"{o}lder inequality,
    \begin{equation}
        \label{eq: v energy}
        \n{\nabla (v^{\frac{t + 1}{2}})}_{2} \le t^{\frac{1}{2}}\n{f}_{r'}^{\frac{1}{2}}\n{v}_{tr}^{\frac{t}{2}}.
    \end{equation}
    From Lemma \ref{thm: multi general} and \eqref{eq: v energy}, we get
    \begin{equation}
        \label{eq: reverse hoelder for v}
            \n{v}_{(t + 1)r} = \n{v^{\frac{t + 1}{2}}}_{2r}^{\frac{2}{t + 1}}
            \underset{\text{Lemma \ref{thm: multi general}}}{\le} \left( B \n{\nabla(v^{\frac{t + 1}{2}})}_{2}^{\theta}\n{v^{\frac{t + 1}{2}}}_{\frac{2tr}{t + 1}}^{1 - \theta}\right)^{\frac{2}{t + 1}}
            \underset{\eqref{eq: v energy}}{\le} B^{\frac{2}{t + 1}} t^{\frac{\theta}{t + 1}}\n{f}_{r'}^{\frac{\theta}{t + 1}}\n{v}_{tr}^{1 - \frac{\theta}{t + 1}},
    \end{equation}
    where $\frac{1}{2r} = \frac{t + 1}{2tr}(1 - \theta) + \frac{d - 2}{2d}\theta$, and $B = C(d)$ if $d \ge 3$ and $B = C(d)^{\frac{2r}{t + 1}}$ if $d = 2$. Now put $M(k) := \n{v}_{tr}^{tr}$. Then from \eqref{eq: reverse hoelder for v}, the H\"{o}lder inequality $\n{v}_{tr} \le \n{v}_{(t + 1)r}^{\tau}\left( \int_{u > k}v^{tr - 1} \right)^{\frac{1}{tr} - \frac{\tau}{(t + 1)r}}$ with $1 - \tau = (tr - 1)(\frac{1}{tr} - \frac{\tau}{(t + 1)r})$, and the relation $D_{r}M(k) := \lim_{h \searrow 0}h^{-1}(M(k + h) - M(k)) = -tr \int_{u > k}(u - k)_{+}^{tr - 1}$, we obtain
    \begin{equation}
        D_{r}M(k) \le - tr A_{1}^{t + 1} M(k)^{1 - \frac{1 - \theta}{tr}}.
    \end{equation}
    Here, we put $A_{1}^{-1} := B^{\frac{2}{t + 1}} t^{\frac{\theta}{t + 1}}\n{f}_{r'}^{\frac{\theta}{t + 1}} $. Since $M(0) > 0$ and $0 \le \theta < 1$, the differential equation 
    \[N'(k) = - tr A_{1}^{t + 1} N(k)^{1 - \frac{1 - \theta}{tr}}, \, N(0) = M(0)\]
    admits the unique solution 
    \[N(k) = \left( q\left( C - tr A_{1}^{t + 1}k \right) \right)^{\frac{1}{q}}, \, q = \frac{1 - \theta}{tr}, \, N(0) = M(0) = (qC)^{\frac{1}{q}}\]
    on $0 \le k \le (trA_{1}^{t + 1})^{-1}C =: k_{m}$. Hence, we see that $M(k) \le N(k)$ on $0 \le k \le k_{m}$ (\cite[Theorem 6.1 of Chapter I]{hale1969ordinary}). Thus, $M(k_{m}) \le N(k_{m}) = 0$, i.e., $\n{u_{+}}_{\infty} \le k_{m}$. Since $(1 - \theta)^{-1} = 1 + (1 - \frac{d - 2}{d}r)^{-1}t^{-1}$ and $t\theta = (t^{-1} + 1 - \frac{d - 2}{d}r)^{-1} \le (1 - \frac{d - 2}{d}r)^{-1}$, we bound
    \[\begin{split}
        \n{u_{+}}_{\infty} &\le k_{m} = A_{1}^{-(t + 1)}\frac{1}{tr}\frac{tr}{1 - \theta}M(0)^{\frac{1 - \theta}{tr}}\\
        &\le B^{2}e^{\frac{1}{e}\left( 1 - \frac{d - 2}{d}r \right)^{-1}}\left( 1 + \frac{1}{t}\left( 1 - \frac{d - 2}{d}r \right)^{-1} \right)\n{f}_{r'}^{\theta}\n{u_{+}}_{tr}^{1 - \theta}.
    \end{split}\]
    Finally, observe that $1 - \theta = \gamma_{t, r}$.
\end{proof}

\begin{lemma}
    \label{thm: approx by simple fcts}
    For a nonnegative finite Borel measure $m$ on $\R^{d}$, define
    \begin{equation}
        \label{def: Sm}
        \mathcal{S}_{m} := \left\{\sum_{i = 1}^{n}a_{i}\chi_{U_{i}}: n \ge 1,\, a_{i} \in \R, \text{$U_{i} \sub \R^{d}$ are open, } m\left( \bigcup_{i \neq j}(U_{i} \cap U_{j}) \cup \left( \R^{d} \setminus \bigcup_{i}U_{i} \right) \right) = 0\right\}.
    \end{equation}
    Then the closure of $\mathcal{S}_{m}$ in $L^{\infty}(m)$ is $\mathcal{C}_{m}$ (see \eqref{def: Cm}).
    If $-\infty <a < b < \infty$ and $u \in \mathcal{C}_{m}$ satisfy $a \le u \le b$ $m$-a.e., then we may choose an approximating sequence $(u_{l}) \sub \mathcal{S}_{m}$ so that $\n{u_{l} - u}_{L^{\infty}(m)} \to 0$ and $a \le u_{l} \le b$ $m$-a.e.
\end{lemma}

\begin{proof}
    Clearly $\mathcal{C}_{m}$ is closed in $L^{\infty}(m)$. Let $u \in \mathcal{C}_{m}$ and $N$ an $m$-null set such that the restriction $u \restriction \R^{d}\setminus N$ is continuous. We shall construct a desired approximating sequence $(u_{l}) \sub \mathcal{S}_{m}$. To this end, we may assume that $-1 \le u \le 1$ $m$-a.e. Let $0 < \theta < 1$ and $\varepsilon > 0$. Let $-1 = a_{0} < a_{1} < \cdots < a_{n} = 1$ such that $a_{i + 1} < a_{i} + \varepsilon$ for $0 \le i \le n - 1$ and $m(\theta u = a_{i}) = 0$ for $0 \le i \le n$. Put $A_{i} := \{a_{i} < \theta u < a_{i + 1}\} \setminus N$ and choose an open subset $U_{i}$ of $\R^{d}$ such that $U_{i} \setminus N = A_{i}$. Then $m(U_{i} \cap U_{j}) \le m(N) = 0$ for distinct $i, j$ and $m(\R^{d} \setminus \bigcup_{i = 0}^{n - 1}U_{i}) \le m(\bigcup_{i = 0}^{n}\{\theta u = a_{i}\} \cup N) = 0$. Hence, $v_{\theta, \varepsilon} := \sum_{i = 0}^{n - 1}a_{i}\chi_{U_{i}} \in \mathcal{S}_{m}$ and $-1 \le v_{\theta, \varepsilon} \le 1$ $m$-a.e. By construction, $\n{u - v_{\theta, \varepsilon}}_{L^{\infty}(m)} \le \n{u - \theta u}_{L^{\infty}(m)} + \n{\theta u - v_{\theta, \varepsilon}}_{L^{\infty}(m)} \le (1 - \theta)\n{u}_{L^{\infty}(m)} + \varepsilon$. Now set $u_{l} := v_{1 - 2^{-l}, 2^{-l}}$.
\end{proof}

\section*{Acknowledgements}
This work was partly presented at the workshop Communications on Probability, Geometry and Numerical analysis (Ishigaki, 2026), and this opportunity helped the author organize the manuscript. The author would like to thank the organizers Toshihiro Uemura, Toshihiro Shoda, and Dai Taguchi for the invitation to the stimulating workshop. The author's participation in the workshop was supported by JSPS KAKENHI Grant Number 25K07056. Part of this work was conducted while the author was a special research student at Kyoto University. The author is grateful to Professor Jun Masamune and the members of the Applied Analysis group for their warm hospitality.

\section*{Statements and Declarations}
\textbf{Competing Interests.} The author declares no competing interests.\\
\textbf{Funding.} Part of this work was supported by JSPS KAKENHI Grant Number 25K07056.\\
\textbf{Data Availability.} No datasets were generated or analyzed during the current study.\\
\textbf{AI disclosure statement.}
The author used Anthropic Claude Opus 5.5 for language editing and for checking the consistency of the exposition. All mathematical contents were verified by the author, who assumes full responsibility for the contents of the paper.
\bibliography{biblio}
\bibliographystyle{abbrv}
\end{document}